\documentclass[a4paper,11pt,fancyhdr]{article}
\usepackage[utf8]{inputenc}
\usepackage{epsfig,amsmath,amssymb,graphicx,fancyheadings,indentfirst,%
  undertilde,color,multirow,slashed,mathtools,listings,cite}
\usepackage[english]{babel}
\usepackage[colorlinks=true,urlcolor=blue]{hyperref}
\usepackage{geometry}
\usepackage[makeroom]{cancel}
\usepackage{amsthm}
\usepackage{tikz}
\usepackage{graphicx}
\usepackage{textcomp}
\usepackage{tikz-cd}
\usepackage{appendix}
\usepackage{subcaption}
\usetikzlibrary{decorations.markings}
\usepackage[ruled,vlined]{algorithm2e}
\usepackage{mathrsfs}
\usepackage{authblk}
\usetikzlibrary{arrows}

\definecolor{lightgrey}{gray}{0.9}

\newcommand{\expect}[1]{\mathbb{E} \!\left [ #1\right]}

\newcommand{\abs}[1]{\left\lvert #1 \right\rvert}

\newcommand{\innermid}{\nonscript\;\delimsize\vert\nonscript\;}
\newcommand{\activatebar}{%
  \begingroup\lccode`\~=`\|
  \lowercase{\endgroup\let~}\innermid 
  \mathcode`|=\string"8000
}

\newcommand{\C}{\mathbb{C}}

\newcommand{\N}{\mathbb{N}}
\newcommand{\Q}{\mathbb{Q}}
\newcommand{\R}{\mathbb{R}}
\newcommand{\Z}{\mathbb{Z}}
\newcommand{\PP}{\mathbb{P}}

\newcommand{\bcode}{\mathcal{B}}

\newcommand{\Vectk}{\textbf{Vect}_k}

\newcommand{\Linfty}{L^{\infty}}
\newcommand{\pvar}{p-\text{var}}
\newcommand{\Var}{\text{Var}}
\newcommand{\suchthat}{\;\ifnum\currentgrouptype=16 \middle\fi|\;}

\DeclareMathOperator\supp{supp}
\DeclareMathOperator\csch{csch}
\DeclareMathOperator\sech{sech}
\DeclareMathOperator\id{id}

\DeclareMathOperator\erfc{erfc}
\DeclareMathOperator\Pers{Pers}

\DeclareMathOperator\erfi{erfi}

\DeclareMathOperator\Lip{Lip}

\DeclareMathOperator\rk{rank}

\DeclareMathOperator\Dgm{Dgm}

\DeclareMathOperator\Real{Re}
\DeclareMathOperator\Imag{Im}

\DeclareMathOperator\updim{\overline{dim}}

\DeclareMathOperator\Li{Li}
\DeclareMathOperator\Res{Res}
\DeclareMathOperator\Herm{He}

\newcommand{\DS}{\bigoplus}

\newcommand{\al}{\alpha}

\newcommand{\veps}{\varepsilon}
\newcommand{\vp}{\varphi}

\newcommand{\bd}{\begin{displaymath}
\begin{tikzcd}}
\newcommand{\ed}{\end{tikzcd}
\end{displaymath}}
\newcommand{\bmat}{\begin{pmatrix}}
\newcommand{\emat}{\end{pmatrix}}
\newcommand{\be}{\begin{equation*}}
\newcommand{\ee}{\end{equation*}}
\newcommand{\btikz}{\begin{tikzcd}}
\newcommand{\etikz}{\end{tikzcd}}
\newcommand{\bea}{\begin{eqnarray}}
\newcommand{\eea}{\end{eqnarray}}
\newcommand{\bse}{\begin{subequations}}
\newcommand{\ese}{\end{subequations}}
\newcommand{\bc}{\begin{center}}
\newcommand{\ec}{\end{center}}

\newcommand{\nonum}{\nonumber}

\newcommand{\half}{\frac{1}{2}}
\newcommand{\bra}{\langle}
\newcommand{\ket}{\rangle}

\newcommand{\norm}[1]{\left\lVert#1\right\rVert}

\newcommand{\del}{\partial}

\newcommand{\comment}[1]{}

\newcommand{\ie}{{\it i.e. }}
\newcommand{\cf}{{\it cf. }}

\newcommand{\Lag}{{\mathcal{L}}}
\newcommand{\Mel}{{\mathcal{M}}}

\def\blob[#1]{~\parbox{#1mm}{
\begin{fmfgraph*}(#1,#1)
\fmfleft{i1}
\fmfright{o1}
\fmf{phantom}{i1,v1,o1}
\fmfblob{0.4w}{v1}
\end{fmfgraph*}}~}
\def\vertex[#1]{~\parbox{#1mm}{
  \begin{fmfgraph*}(#1,#1)
    \fmfleft{i1, i2}
    \fmfright{o1,o2}
    \fmf{phantom}{i1,i2,o2,o1,i1}
    \fmf{plain}{i1,v,i2}
    \fmf{plain}{o1,v,o2}
    \fmfforce{nw}{i1}
    \fmfforce{sw}{i2}
    \fmfforce{se}{o1}
    \fmfforce{ne}{o2}
    \fmfforce{c}{v}
    \fmfdot{v}
  \end{fmfgraph*}
  }~}
\def\othervertex[#1]{~\parbox{#1mm}{
  \begin{fmfgraph*}(#1,#1)
    \fmfleft{i1, i2}
    \fmfright{o1,o2}
    \fmf{phantom}{i1,i2,o2,o1,i1}
    \fmf{plain}{i1,v}
    \fmf{plain}{i2,v}
    \fmf{plain}{o1,v}
    \fmf{plain}{v,o2}
    \fmfforce{nw}{i1}
    \fmfforce{sw}{i2}
    \fmfforce{se}{o1}
    \fmfforce{ne}{o2}
    \fmfforce{c}{v}
    \fmfblob{0.3w}{v}
  \end{fmfgraph*}
  }~}
\def\bigvertex[#1]{~\parbox{#1mm}{
  \begin{fmfgraph*}(#1,#1)
    \fmfleft{i1, i2}
    \fmfright{o1,o2}
    \fmf{phantom}{i1,i2,o2,o1,i1}
    \fmf{plain}{i1,v,i2}
    \fmf{plain}{o1,v,o2}
    \fmfforce{nw}{i1}
    \fmfforce{sw}{i2}
    \fmfforce{se}{o1}
    \fmfforce{ne}{o2}
    \fmfforce{c}{v}
    \fmfdot{v}
  \end{fmfgraph*}
  }~}
\def\edge[#1]{~\parbox{#1mm}{
  \begin{fmfgraph*}(#1,#1)
    \fmfleft{i}
    \fmfright{o}
    \fmf{plain,l.s=left}{i,o}
  \end{fmfgraph*}
}~}
\def\curlyedge[#1]{~\parbox{#1mm}{
  \begin{fmfgraph*}(#1,#1)
    \fmfleft{i}
    \fmfright{o}
    \fmf{curly,l.s=left}{i,o}
  \end{fmfgraph*}
}~}
\def\wavyedge[#1]{~\parbox{#1mm}{
  \begin{fmfgraph*}(#1,#1)
    \fmfleft{i}
    \fmfright{o}
    \fmf{wiggly,l.s=left}{i,o}
  \end{fmfgraph*}
}~}
\def\tadpole[#1]{~\parbox{#1mm}{
 	\begin{fmfgraph*}(#1,#1)
 		\fmfleft{i1}
 		\fmfright{o1}
		\fmf{plain}{i1,v1,v1,o1}
 		\fmfdot{v1}
 	\end{fmfgraph*}}~}
\def\amputatedtadpole[#1]{\begin{fmfgraph*}(#1,2)
		\fmfleft{i1}
		\fmfright{o1}
		\fmf{phantom}{i1,v1,o1}
		\fmf{plain}{v1,v1}
		\fmfdot{v1}
	\end{fmfgraph*}}
\def\amputatedwigglytadpole[#1]{\begin{fmfgraph*}(#1,2)
		\fmfleft{i1}
		\fmfright{o1}
		\fmf{phantom}{i1,v1,o1}
		\fmf{wiggly}{v1,v1}
		\fmfdot{v1}
	\end{fmfgraph*}}
\def\amputatedcurlytadpole[#1]{\begin{fmfgraph*}(#1,2)
		\fmfleft{i1}
		\fmfright{o1}
		\fmf{phantom}{i1,v1,o1}
		\fmf{curly}{v1,v1}
		\fmfdot{v1}
	\end{fmfgraph*}}
\def\vacfirstord[#1]{~\parbox{#1mm}{
 	\begin{fmfgraph*}(#1,#1)
		\fmfleft{i1,i2}
		\fmfright{o1,o2}
		\fmf{plain,left}{v2,v1}
		\fmf{plain,left}{v1,v2}
		\fmf{plain,left}{v3,v1}
		\fmf{plain,left}{v1,v3}
		\fmf{phantom}{v2,i1}
		\fmf{phantom}{v3,o1}
		\fmfforce{(0,0)}{i1}
		\fmfforce{(w,0)}{o1}
		\fmfforce{(0.5w,0.5h)}{v1}
		\fmfforce{(0,0.5h)}{v2}
		\fmfforce{(w,0.5h)}{v3}
		\fmfforce{nw}{i2}
		\fmfforce{ne}{o2}
		\fmfforce{(0,0.5h)}{i2}
		\fmfforce{(w,0.5h)}{o2}
		\fmfdot{v1}
		\end{fmfgraph*}}~}
\def\doubletadpolehor[#1]{~\parbox{#1mm}{
\begin{fmfgraph}(#1,#1)
	\fmfleft{i1}
	\fmfright{o1}
	\fmf{plain}{i1,v1,v1,v2,v2,o1}
	\fmfdot{v1,v2}
\end{fmfgraph}
}~}
\def\tripletadpolehor[#1]{~\parbox{#1mm}{
\begin{fmfgraph}(#1,#1)
	\fmfleft{i1}
	\fmfright{o1}
	\fmf{plain}{i1,v1,v1,v2,v2,v3,v3,o1}
	\fmfdot{v1,v2,v3}
\end{fmfgraph}
}~}
\def\doubletadpolever[#1]{~\parbox{#1mm}{
\begin{fmfgraph}(#1,#1)
	\fmfleft{i1}
	\fmfright{o1}
	\fmf{plain}{i1,v1}
	\fmf{plain,left}{v1,v2}
	\fmf{plain,left}{v2,v1}
	\fmf{plain}{v1,o1}
	\fmffreeze
	\fmf{plain,left}{v2,v3,v2}
	\fmfforce{c}{v2}
	\fmfforce{(0.5w,h)}{v3}
	\fmfforce{(0.5w,0)}{v1}
	\fmfforce{sw}{i1}
	\fmfforce{se}{o1}
	\fmfdot{v1,v2}
\end{fmfgraph}
}~}
\def\amputateddoubletadpolever[#1]{~\parbox{#1mm}{
	\begin{fmfgraph}(#1,#1)
	\fmfleft{i1}
	\fmfright{o1}
	\fmf{phantom}{i1,v1}
	\fmf{plain,left}{v1,v2}
	\fmf{plain,left}{v2,v1}
	\fmf{phantom}{v1,o1}
	\fmffreeze
	\fmf{plain,left}{v2,v3,v2}
	\fmfforce{c}{v2}
	\fmfforce{(0.5w,h)}{v3}
	\fmfforce{(0.5w,0)}{v1}
	\fmfforce{sw}{i1}
	\fmfforce{se}{o1}
	\fmfdot{v1,v2}
	\end{fmfgraph}~
}
}
\def\sunset[#1]{\parbox{#1mm}{
\begin{fmfgraph}(#1,#1)
\fmfleft{i}
\fmfright{o}
\fmfforce{0,0.5h}{i}
\fmfforce{w,0.5h}{o}
\fmf{plain,tension=5}{i,v1}
\fmf{plain,tension=5}{v2,o}
\fmf{plain,left,tension=0.5}{v1,v2,v1}
\fmf{plain}{v1,v2}
\fmfdot{v1,v2}
\end{fmfgraph}
}}
\def\amputatedsunset[#1]{\parbox{#1mm}{
\begin{fmfgraph}(#1,#1)
\fmfleft{i}
\fmfright{o}
\fmf{phantom,tension=5}{i,v1}
\fmf{phantom,tension=5}{v2,o}
\fmf{plain,left,tension=0.5}{v1,v2,v1}
\fmf{plain}{v1,v2}
\fmfdot{v1,v2}
\end{fmfgraph}
}}
\def\fourpointsecondorder[#1]{~\parbox{#1mm}{
	\begin{fmfgraph}(15,#1)
	\fmfleft{i1,i2}
	\fmfright{o1,o2}
	\fmf{plain}{i1,v1}
	\fmf{plain}{i2,v1}
	\fmf{plain,right}{v1,v2}
	\fmf{plain,left}{v1,v2}
	\fmf{plain}{v2,o1}
	\fmf{plain}{v2,o2}
	\fmfdot{v1,v2}
	\end{fmfgraph}}~
}
\def\fourpointsecondordertwo[#1]{~\parbox{#1mm}{
	\begin{fmfgraph}(15,#1)
	\fmfleft{i1,i2,i3}
	\fmfright{o1}
	\fmf{plain}{i1,v1}
	\fmf{plain}{i2,v1}
	\fmf{plain}{i3,v1}
	\fmf{plain}{v1,v2,v2,o1}
	\fmfdot{v1,v2}
	\end{fmfgraph}
}~}

\newtheorem{theorem}{Theorem}[section]

\catcode`,\active

\catcode`\,12

\theoremstyle{definition}
\newtheorem{definition}[theorem]{Definition}
\newtheorem{convention}[theorem]{Convention}

\newtheorem{notation}[theorem]{Notation}
\newtheorem{lemma}[theorem]{Lemma}
\newtheorem{proposition}[theorem]{Proposition}
\newtheorem{defprop}[theorem]{Definition/Proposition}

\theoremstyle{remark}
\newtheorem{remark}[theorem]{Remark}
\theoremstyle{example}

\newtheorem{corollary}[theorem]{Corollary}

\title{$\zeta$-functions and the topology of superlevel sets of stochastic processes}

\author[1,2,3]{Daniel Perez\thanks{Email: \texttt{daniel.perez@ens.fr}}}
\affil[1]{\footnotesize D\'epartement de math\'ematiques et applications, \'Ecole normale sup\'erieure, CNRS, PSL University, 75005 Paris, France}
\affil[2]{\footnotesize Laboratoire de math\'ematiques d'Orsay, Universit\'e Paris-Saclay, CNRS, 91405 Orsay, France}
\affil[3]{\footnotesize DataShape, Centre Inria Saclay, 91120 Palaiseau, France}
\providecommand{\keywords}[1]{\textbf{\textit{Index terms---}} #1}
\providecommand{\subjclass}[1]{\textbf{MSC2020 Classification---} #1}
\date{\today}

\begin{document}
\maketitle

\begin{abstract}
We describe the topology of superlevel sets of ($\alpha$-stable) L\'evy processes X by introducing so-called stochastic $\zeta$-functions, which are defined in terms of the widely used $\Pers_p$-functional in the theory of persistence modules. The latter share many of the properties commonly attributed to $\zeta$-functions in analytic number theory, among others, we show that for $\alpha$-stable processes, these (tail) $\zeta$-functions always admit a meromorphic extension to the entire complex plane with a single pole at $\alpha$, of known residue and that the analytic properties of these $\zeta$-functions are related to the asymptotic expansion of a dual variable, which counts the number of variations of X of size $\geq \veps$. Finally, using these results, we devise a new statistical parameter test using the topology of these superlevel sets. We further develop an analogous theory, whereby we consider the dual variable to be the number of points in the persistence diagram inside the rectangle $]\infty, x]\times[x+\veps, \infty[$.
\end{abstract}
 \\
\\

\tableofcontents

\section{Introduction}
\label{intro}

The problem of the characterization of the topology of superlevel sets of random functions has been a long studied topic in the theory of random fields. While a complete description has been thus far unknown, partial descriptors of the topology of superlevel sets, such as their Euler characteristic, have been described for certain classes of random processes \cite{AdlerTaylor:RandomFields, AzaisWschebor, Baryshnikov_2019,Picard:Trees, ChazalDivol:Brownian, Neveu_1989}. Thus far, the study of the homology of superlevel sets of random functions in dimension one has focused on either smooth random (Gaussian) fields \cite{AdlerTaylor:RandomFields, AzaisWschebor}, or irregular processes which are in some sense canonical, such as Brownian motion \cite{ChazalDivol:Brownian,Picard:Trees,Baryshnikov_2019}. In this paper, following the universality reasoning detailed in \cite[\S 3]{Perez_Pr_2020}, we will adopt the second point of view while enlarging the category of processes considered to objects acting as universal limits of random processes in 1D. 

This work is another stage in a program started in \cite{Perez_2020} and later continued in \cite{Perez_Pr_2020}, which aimed to characterize the barcodes of random functions as completely as possible (in dimension one). To do this, we adopted the tree formalism originally developped by Le Gall \cite{LeGall:Trees, LeGallDuquesne:LevyTrees}, which brings benefits in the probabilistic setting.  This formalism allowed us to partially study the case of Markov or self-similar processes, and to processes admitting the two latter as limits \cite{Perez_Pr_2020}. In this paper, we further develop the theory to describe almost completely the case of ($\al$-stable) L\'evy processes.

The understanding of so-called \textit{topological noise} is an active area of research in Topological Data Analysis (TDA) (\cf section \ref{sec:PH} for a quick introduction, or \cite{Chazal:Persistence, Oudot:Persistence} for a more comprehensive one). This topological noise is characterized by the behaviour of the small bars of the barcode of a function and its role is particularly difficult to grasp. Nonetheless, topological noise has proved useful in a variety of different applications, which seem to exploit the information contained therein, sometimes directly, or indirectly through the use of Wasserstein $p$ metrics on the space or functionals such as $\Pers_p$ \cite{LipschitzStableLpPers,Mileyko_2011,Turner_2014,Divol_2019,Carriere_2017}, despite the absence of general stability results. The program described above is based around the intuition that while it might be hopeless to find a general stability theorem, there should be a form of \textit{statistical stability} of barcodes.

By statistical stability we mean that any two samples drawn from the same probability distribution should have ``close'' or ``similar'' behaviour of its topological noise. For example, this paper shows that, at least in the simple setting of 1D and for a fairly wide range of distributions, this topological noise does indeed exhibit the robustness sought. More precisely, we show that the number of bars of length $\geq \veps$, $N^\veps$, is a dual to $\Pers_p^p$, and that it exhibits the statistical robustness required, such as an almost sure asymptotic behaviour as $\veps \to 0$. 

If this notion of statistical stability is a good one, a natural subsequent question is whether it is possible to differentiate stochastic processes given their topological signature. In this paper, we partially answer this question positively through the development of a statistical test constructed with the functional $N^\veps$, which can differentiate $\al$-stable processes for different values of $\al$. In dimension one, while interesting, this development is unlikely to do better than wavelet analysis or other known techniques (\cf \cite{Dang:2015vx} and the references therein). However, further developments in this direction could eventually lead to ``\textit{topological statistics}'', \textit{i.e.} robust statistical tests for random fields, for which all known techniques do not generalize, but for which barcodes are easily computable.

This discussion and the results of this paper hint at the fact that statistical stability is a correct notion of stability to consider, and there are many open questions to be tackled, some of these questions are:
\begin{itemize}
  \item \textbf{Dimensionality}: are the statistical stability results of this paper particular to dimension one, or do they generalize in some way to higher dimensional random fields?
  \item \textbf{Signal vs. \!noise problem}: Given the regular structure of topological noise shown in this paper, does this allow us to detect the presence of an underlying topological signal? 
  \item \textbf{Statistical robustness of topological noise}: What can we say, quantitatively, about the variation of topological noise induced by perturbations of the distribution of the noise (for instance in some Wasserstein metric)? 
  \item \textbf{Best proxys for topological signatures}: much of this paper was inspired by what is used in practice, namely Wasserstein $p$ metrics and the $\Pers_p$ functional and its dual $N^\veps$. However, there is no guarantee this yields the best possible proxys to answer the two previous problems. In this regard, is it possible to prove or disprove which proxys do best in what context? 
\end{itemize}

From a more probabilistic point of view, this paper introduces so-called $\zeta$-functions associated to a stochastic process, constructed using the $\Pers_p^p$ functional we previously discussed. The main, and perhaps most important, departure from the conventional TDA theory is that we will consider this quantity for complex $p$ for reasons which will become evident throughout this paper, but which are analogous to the ones behind the complexification of the Riemann $\zeta$-function in analytic number theory. There are some similarities to the results of Pitman, Yor and Biane in \cite{pitman_yor_2003,Biane_2001,Pitman_1999} regarding the probabilistic interpretation of the $\zeta$-function and more generally $L$-functions based on connections with some families of infinitely divisible distributions connected to Brownian motion. However, the $\zeta$-functions herein are of different nature to those considered by Pitman, Yor and Biane, as they stem from a different construction. For Brownian motion, we fall back on one of the infinitely divisible distributions considered by these authors. Renewal theory \cite{Gut_2009} intervenes at many different steps in this paper and combined with the results of Pitman, Yor and Biane, it is \textit{a posteriori} perhaps not surprising that the $\zeta$-functions hereby introduced share some of the analycity properties of the Riemann $\zeta$-function.

\subsection{Our contribution}
More precisely, our contribution can be split along the following lines:
\begin{enumerate}
  \item We establish a duality relation with respect to the Mellin transform between the study of $\Pers_p^p$ and the number of leaves of a $\veps$-trimmed tree $\geq \veps$, $N^\veps$ (\cf section \ref{sec:TreeIntegration}). With the help of a correct notion of integration on trees developped in \cite{Picard:Trees}, it is possible to prove an interpolation theorem for $\Pers_p^p$ (proposition \ref{prop:Lyapunovellpp});
  \item We introduce $\zeta$-functions for stochastic processes and for persistent measures (\cf section \ref{sec:Zetafunctions} and section \ref{sec:WassersteinCvg} respectively). We show an interpolation theorem for Wasserstein $p$-distances between diagrams (proposition \ref{prop:interpolationOT}) and a characterization of convergence between diagrams in these Wasserstein distances in terms of the $\zeta$-functions introduced (theorem \ref{thm:zetaWasserstein}).
  \item We show that in the context of $\al$-stable Lévy processes, the associated (tail) $\zeta$-functions always admit a meromorphic extension to the entire complex plane, with a unique pole at $p = \al$ with known residue (theorem \ref{thm:tailzetaLevy}). By duality, this meromorphic extension implies the existence of an asymptotic series for $N^\veps$ as $\veps \to 0$, which we explicitly calculate up to superpolynomial (\ie smaller than any polynomial) corrections (theorem \ref{thm:meroextensionLevy}). An explicit form of the meromorphic continuation of $\hat\zeta$ is shown to be related to the superpolynomial corrections to the asymptotic expansion of theorem \ref{thm:meroextensionLevy} (\cf section \ref{sec:analyticcontinuation}). We also define a generating function for the length of the $k$th longest bar (\cf section \ref{sec:distlengthbars}); 
  \item We give an almost sure result detailing the asymptotics of $N^\veps$ for any continuous semimartingale (proposition \ref{prop:semimartingaleszeta}), which turns out to be related to the quadratic variation of the process, providing further evidence for the link between regularity and $\Pers^p_p$ hinted at in \cite{Picard:Trees,Polterovitch:LaplaceEigenfunctions}. We conclude from this that the $\zeta$-functions associated to continuous semimartingales always have a simple pole at $p=2$;
  \item We apply the theory above to different stochastic processes, such as Brownian motion, reflected Brownian motion. We derive explicit formul{\ae} for the respective $\zeta$-functions of these processes and infer the associated asymptotic expansions of $N^\veps$ (theorems \ref{thm:zetafunctionBrownian}, \ref{thm:ReflectedBrownianZeta} propositions \ref{prop:NvepsExpansionBM} and \ref{prop:ENtvepsforRefBM}) and in the case of Brownian motion, the explicit distribution of the length of the $k$th longest bar (\cf section \ref{sec:BMkthlongestbar}).
  \item We design a statistical test for the parameter $\al$ of $\al$-stable L\'evy processes by using the theory previously described (\cf section \ref{sec:StatisticalTest});
  \item We study local trees and introduce local $\zeta$-functions (\cf section \ref{sec:Zetafunctions}) and deduce formul{\ae} for the number of points contained in the rectangle $]\!-\!\infty, x] \times [x+\veps,\infty[$ of the persistence diagram, $N^{x,x+\veps}$, by introducing the notion of propagators, which, for Markov processes, reduces the problem of the study of $N^{x,x+\veps}$ to the study of hitting times of the process (\cf section \ref{sec:propagatorslocaltrees}), in particular, we link the regularity of these propagators with meromorphic extensions of the local $\zeta$-functions of the process (proposition \ref{prop:MeroExistsLocal});
  \item Finally, we apply the theory above to different stochastic processes, such as Brownian motion, reflected Brownian motion, Brownian motion with drift and the Ornstein-Uhlenbeck process for which explicit computations are possible. We derive explicit formul{\ae} for the respective (local, global) $\zeta$-functions of these processes and infer the associated asymptotic expansions of $N^{x,x+\veps}$ (theorems \ref{thm:Brownianlocalzeta}, \ref{thm:ReflectedBrownianZeta} propositions \ref{prop:RefBMNxxveps} and \ref{prop:ENtvepsforRefBM} and section \ref{sec:BMwithDrift}). We also infer formul{\ae} regarding the Ornstein-Uhlenbeck process, in particular concerning its local time (\cf section \ref{sec:OUprocess}). 
\end{enumerate}

\section{Generalities}
\subsection{The Mellin transform}
\begin{definition}
Let $f$ be a locally integrable function over the ray $]0,\infty[$. The \textbf{Mellin transform} of $f$ is
\be
\Mel[f(x)](s) := \int_0^\infty  x^{s-1} f(x) \; dx \,.
\ee
\end{definition}
Note that $d\log(x) = \frac{dx}{x}$ is the Haar measure of $(\R_+, \times)$. The Mellin transform reflects the Pontryagin duality with respect to this locally compact abelian group. Its theory is analogous to that of the bilateral Laplace transform, as the map $\log : (\R_+,\times) \to (\R,+)$ induces an isomorphism of abelian groups.
\begin{notation}
For convenience, we will also employ the shorthand notation $\Mel[f](s) = f^*(s)$. 
\end{notation}
\begin{definition}
The \textbf{fundamental strip} of $f$, $\bra \al, \beta \ket$ is the maximal set 
\be
\bra \al, \beta \ket := \left\{z \in \C \, |\, \al < \Real(z) < \beta\right\}
\ee
where $f^*(s)$ is well defined.
\end{definition}
The Mellin transform can be inverted by virtue of the following theorem, which follows from the Laplace inversion theorem.
\begin{theorem}[Mellin inversion,\cite{Doetsch_1950,OFFORD_1938}]
\label{thm:MellinInversion}
Let $f$ have fundamental strip $\bra \al, \beta\ket$ and let $c \in\, ]\al, \beta[$. Then 
\begin{enumerate}
  \item If $f$ is integrable and $f^*(c+it)$ is integrable, then for almost every $x$
  \be
  f(x) = \frac{1}{2\pi i} \int_{c-i\infty}^{c+i\infty} f^*(s) x^{-s} \;ds 
  \ee
  If $f$ is continuous, the equality holds everywhere. 
  \item If $f$ is locally integrable and of bounded variation in a neighbourhood of $x$, then 
  \be
  \frac{f(x^+)+f(x^-)}{2} = \lim_{T \to \infty} \frac{1}{2\pi i}\int_{c-iT}^{c+iT} f^*(s) x^{-s} \; ds 
  \ee
\end{enumerate}
\end{theorem} 
A sufficient condition for the Mellin transform to be well-defined on $\bra \al, \beta \ket$ is that the function is such that
\be
f(x) = O(x^{-\al}) \; \text{as } x \to 0 \quad \text{and} \quad f(x) = O(x^{-\beta}) \; \text{as } x \to \infty \,. 
\ee
In fact, Mellin transforms are a good tool to study asymptotic expansions as suggested by the following theorem. 
\begin{figure}[h!]
\centering
\includegraphics[width=0.45\textwidth]{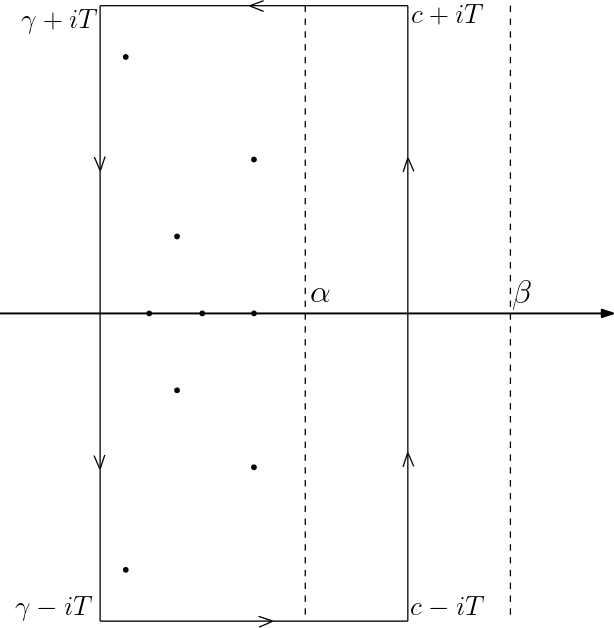}
\caption{Contour for the evaluation of the Bromwich integral of the inverse Mellin transform.}
\label{fig:ContourMellinInverse}
\end{figure}
\begin{theorem}[Fundamental correspondence, \cite{Flajolet_1995}]
\label{thm:FundamentalCorrespondence}
Let $f : \;]0,\infty[ \to \C$ be a continuous function with non-empty fundamental strip $\bra \al, \beta \ket$. Then,
\begin{itemize}
 \item Assume that $f^*(s)$ admits a meromorphic continuation to the strip $\bra \gamma, \beta \ket$ for $\gamma < \al$, that it has only a finite amount of poles there and that it is analytic on $\Real(s) = \gamma$. Assume also that there exists $\eta \in \, ]\al, \beta[$ such that along a denumerable set of horizontal segments with $\abs{\Imag(s)}= T_i$ where $T_i \to \infty$,
 \be
 f^*(s) = O(\abs{s}^{-r}) \text{ with } r>1 \text{ as } \abs{s} \to \infty \text{ and } s \in \bra \gamma, \eta \ket \,.
 \ee
 Indexing the poles on $\bra \gamma, \beta \ket$ by their location $\xi$ and by their order $k$ and denoting $c_{\xi,k}$ the $k$th coefficient in the Laurent expansion around $\xi$ of $f^*(s)$, we have an asymptotic expansion of $f$ around $0$ 
\be
f(x) \sim \sum_{(\xi,k)} c_{\xi,k} \frac{(-1)^{k-1}}{(k-1)!}x^{-\xi} \log^k(x) +O(x^{-\gamma}) \quad \text{as } x \to 0 \,.
\ee 
\item Conversely, if the function $f$ has such an asymptotic expansion around $0$, then $f^*(s)$ has a meromorphic continuation to the strip $\bra \gamma,\beta \ket$. 
\end{itemize}
Furthermore, an analogous statement holds true for asymptotic expansions around $\infty$ and meromorphic continuations beyond $\beta$.
\end{theorem}
\begin{proof}[Sketch of proof]
It suffices to perform contour integration using the contour of figure \ref{fig:ContourMellinInverse}. The estimates of the theorem allow us to discard the top and bottom integrals and to state that the integral of the path along $\Real(p) =\gamma$ is $O(x^{-\gamma})$. 
Conversely, consider
\be
f(x) \sim \sum_{(\xi,k)} c_{\xi,k} x^{\xi} \log^k(x) + O(x^{-\gamma}) \quad \text{as } x\to 0 
\ee
for some $\gamma < \al$. It follows that
\begin{align*}
f^*(s) &=  \sum_{(\xi,k)} c_{\xi,k}\frac{(-1)^k k!}{(s+\xi)^{k+1}} + \int_1^\infty x^{s-1} f(x) \; dx \\
& + \int_0^1 x^{s-1} \underbrace{\left[f(x) - \sum_{(\xi,k)} c_{\xi,k} x^{\xi} \log^k(x) \right]}_{=O(x^{-\gamma})} \;dx \,,
\end{align*}
which is well-defined on the strip $\bra \gamma, \beta \ket$.
\end{proof}
\begin{table}[h!]
\centering
\begin{tabular}{||c c c ||} 
 \hline
 $f(x)$ & $f^*(s)$ & $\bra \al, \beta \ket$  \\ [3pt]
 \hline
 $x^\nu f(x)$ & $f^*(s+\nu)$ & $\bra \al-\nu , \beta-\nu \ket$  \\ [5pt]
 $f(x^\nu)$ & $\frac{1}{\nu}f^*(\tfrac{s}{\nu})$ & $\bra \nu\al, \nu\beta \ket$  \\ [5pt]
 $f(x^{-1})$ & $f^*(-s)$ & $\bra -\beta, -\al \ket$ \\ [5pt]
 $f(\lambda x)$ & $\lambda^{-s}f^*(s)$ & $\bra \al, \beta \ket$ \\ [5pt]
 $\frac{\del}{\del x} f(x)$ & $-(s-1)f^*(s-1)$ &   \\ [5pt]
 $\int_0^x f(t)\; dt$ & $-\frac{1}{s}f^*(s+1)$ &   \\ [5pt]
 \hline
\end{tabular}
\caption{Functional properties of the Mellin transform}
\label{table:FunctionalPropertiesMellin}
\end{table}
\begin{table}[h!]
\centering
\begin{tabular}{||c c c ||} 
 \hline
 $f(x)$ & $f^*(s)$ & $\bra \al, \beta \ket$  \\ [3pt]
 \hline
 $e^{-x}$ & $\Gamma(s)$ & $\bra 0 , \infty \ket$  \\ [10pt]
 $e^{-x^2}$ & $\half \Gamma(\frac{s}{2})$ & $\bra 0 , \infty \ket $\\ [10pt]
 $\erfc(x)$ & $2^{-s}\frac{\Gamma(s)}{\Gamma(1+\frac{s}{2})}$ & $\bra 0 , \infty \ket$  \\ [10pt]
 $\csch(x)$ & $2^{1-s} \left(2^s-1\right) \Gamma (s)\zeta (s)$ & $\bra 1, \infty \ket$ \\ [10pt]
 $\csch^2(x)$ & $2^{2-s}  \Gamma (s)\zeta(s-1)$ &  $\bra 2, \infty \ket$ \\ [10pt]
 $\frac{1}{e^x -1} $ & $\Gamma(s)\zeta(s)$ & $\bra 1, \infty \ket$ \\ [7pt]
 \hline
\end{tabular}
\caption{A short dictionary of Mellin transforms.}
\label{table:MellinTransforms}
\end{table}

\subsubsection{Analytic continuation}
\label{sec:analyticcontinuation}
As stated by the fundamental correspondence (theorem \ref{thm:FundamentalCorrespondence}), the existence of an asymptotic expansion around $0$ of $f(x)$ entails a meromorphic continuation of $f^*(s)$ to a larger strip. If $f(x)$ admits a converging Laurent series (with finite singular part) on some open disk around the origin, then this extension is in fact valid over all of $\C$, and the residues of the poles of $f^*(s)$ will be related to the Laurent coefficients of $f(x)$. It turns out that in this context, one can even write an explicit integral representation for the extension of $f^*(s)$.

\begin{lemma}[Integral representation of $f^*$]
\label{lemma:IntegralRepresentationoff}
Let $f$ be a meromorphic function admitting a Laurent series at $0$, with singular part of degree $n$, holomorphic on a neighbourhood of $\R_+^*$ and integrable over the Hankel contour (\cf figure \ref{fig:HankelContour}). Suppose further that its fundamental strip $\bra n, \beta \ket$ is non-empty. Then, the function $f^*$ admits a meromorphic continuation on $\bra -\infty, \beta \ket$ given by
\begin{align*}
f^*(s) &= \frac{e^{-i\pi s}}{2i \sin(\pi s)} \oint_H z^{s-1} f(z) \;dz  \\
&= -\frac{\Gamma(s)\Gamma(1-s)}{2\pi i} \oint_H (-z)^{s-1} f(z) \;dz\,,
\end{align*}
where $H$ denotes the Hankel contour.
\end{lemma}
\begin{figure}[h!]
\centering
\includegraphics[width=0.5\textwidth]{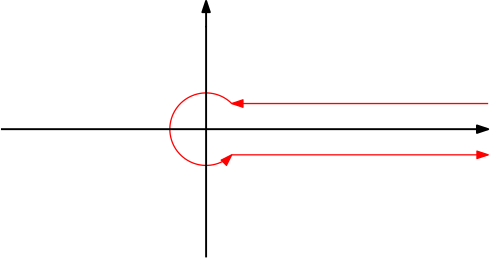}
\caption{The Hankel contour $H$.}
\label{fig:HankelContour}
\end{figure}

\begin{proof}
We start by splitting the Hankel contour into three pieces. 
\begin{enumerate}
  \item A segment from $\infty+i\veps$ to $\nu + i\veps$;
  \item A circle $C_\nu$ around the origin of radius $\nu$;
  \item A segment from $\nu-i\veps$ to $\infty - i\veps$.
\end{enumerate}
For $s \in \bra n, \beta \ket$, $f$ is holomorphic everywhere on this contour, so that we may take $\veps = 0$ according to Cauchy's theorem. Notice also that
\be
\int_{C_\nu} z^{s-1} B(z) \;dz = O(\nu^{\Real(s)-n})  \to 0 \quad \text{as } \nu \to 0 \,.
\ee
It follows that for $\Real(s) > n$
\begin{align*}
\oint_H z^{s-1} f(z) \;dz  &= \lim_{\nu \to 0} \left\{\int_\infty^\nu  + \int_{C_\nu} + \int_{\nu e^{2\pi i}}^{\infty e^{2\pi i}} \right\} z^{s-1} f(z) \;dz  \\
&= (e^{2\pi i (s-1)}-1) \int_{0}^\infty z^{s-1} f(z) \; dz \\
&= 2 i e^{i\pi s} \sin(\pi s) f^*(s) \,,
\end{align*}
as desired. The integral over the complex contour $H$ converges for all $s \in \bra -\infty, \beta \ket \setminus \{n\}$. Finally, the second expression for $f^*$ is obtained through Euler's reflection formula, namely
\be
\Gamma(p)\Gamma(1-p) = \frac{\pi}{\sin(\pi p)} \,,
\ee
which after some simplification yields the desired expression.
\end{proof}
\begin{remark}
If $f$ has fundamental strip $\bra n , \infty \ket$, then the extension given by this procedure holds over $\C$.
\end{remark}
Furthermore, if $f$ posseses a meromorphic continuation to $\C \setminus \R_+$ (\ie we admit the possibility of a branch cut on the positive real axis), then we can find a more explicit formulation for the Hankel representation of $f^*$. 
\begin{lemma}[Functional equation of $f^*$]
\label{lemma:functionalequation}
Suppose $f$ posseses a meromorphic continuation to $\C \setminus \R_+$ and denote $\mathcal{P}$ the set of poles of $f$ not including $0$. Suppose further that $f$ has the following decay condition : for all $s \in \bra n, \beta \ket$ and for some monotone increasing sequence of radii $r_n \to \infty$ as $n \to \infty$,
\be
\int_{C_{r_n,\veps}} \abs{z^{s-1}f(z)} \;dz  \xrightarrow[n \to \infty]{} 0 \,.
\ee
where $C_{r_n,\veps}$ is the circle of radius $r_n$ minus a small (symmetric) arc of length $\veps$ around the positive real axis (\cf figure \ref{fig:HankelPacMan}). Then, 
\be
f^*(s) =  \Gamma(s)\Gamma(1-s) \sum_{z_0 \in \mathcal{P}} \Res((-z)^{s-1}f(z); z_0) \,.
\ee  
\end{lemma} 
\begin{proof}
The proof relies on the use of the residue theorem by completing the Hankel contour into a Pac-Man (\cf figure \ref{fig:HankelPacMan}), whose circular contribution is going to zero, due to the assumption of the lemma.
\begin{figure}[h!]
\centering
\includegraphics[width=0.5\textwidth]{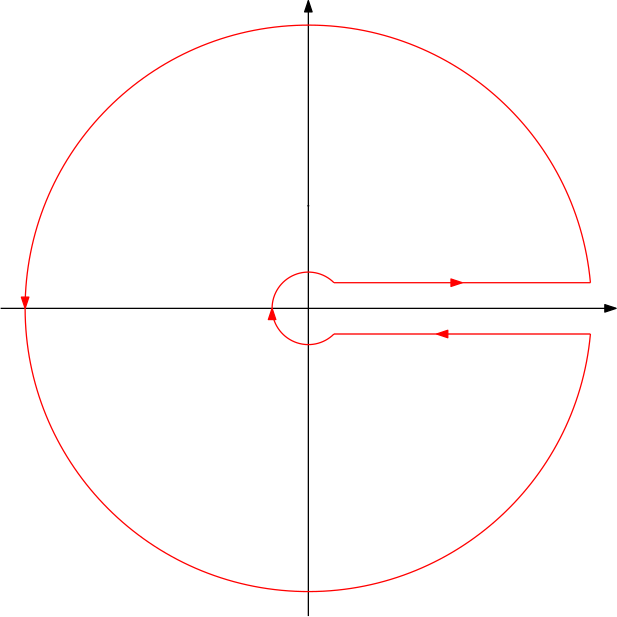}
\caption{The Pac-Man contour.}
\label{fig:HankelPacMan}
\end{figure}
By the residue theorem, we then have 
\begin{align*}
f^*(s) =  \Gamma(s)\Gamma(1-s) \sum_{z_0 \in \mathcal{P}} \Res((-z)^{s-1}f(z); z_0) \,,
\end{align*}
as desired.
\end{proof}

\subsection{Connected components of superlevel sets of stochastic processes}
Let us briefly recall the construction of a tree from a continuous function $f: [0,1] \to \R$. For a more complete description of this, the reader is welcome to consult \cite{LeGall:Trees,Perez_2020}.
\begin{defprop}[\cite{LeGall:Trees}]
Let $x<y \in [0,1]$, the function
\be
d_f(x,y) := f(x) +f(y) -2 \min_{t \in [x,y]} f(t)
\ee
is a pseudo-distance on $[0,1]$ and the quotient metric space
\be
T_f := [0,1]/\{d_f = 0\}
\ee
with distance $d_f$ is a rooted $\R$-tree, whose root coincides with the image in $T_f$ of the point in $[0,1]$ at which $f$ achieves its infimum. 
\end{defprop}
\begin{figure}[h!]
  \centering
    \includegraphics[width=0.8\textwidth]{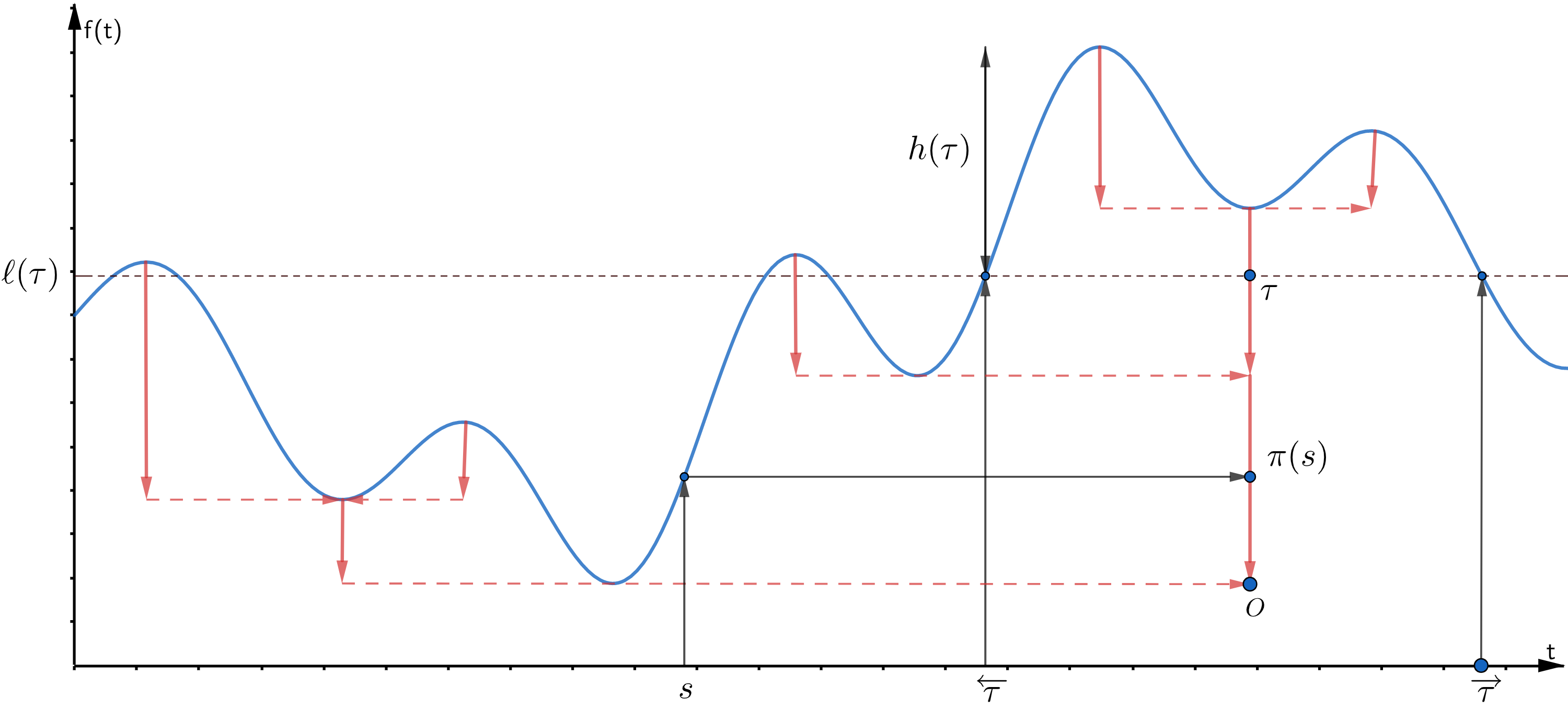}
  \caption{A function $f$ and its associated tree $T_f$ in red.}
\label{fig:treequantities}
\end{figure}
The tree $T_f$ has the particularity that its branches correspond to connected components of the superlevel sets of $f$, as illustrated by figure \ref{fig:treequantities}. 
Let us now introduce the so-called $\veps$-simplified or $\veps$-trimmed tree of $T_f^\veps$.  This object is obtained by ``giving a haircut'' of length $\veps$ to $T_f$. More precisely, if we define a function $h: T_f \to \R$ which to a point $\tau \in T_f$ associates the distance from $\tau$ to the highest leaf above $\tau$ with respect to the filtration on $T_f$ induced by $f$, then
\begin{definition}
Let $\veps \geq 0$. An \textbf{$\veps$-trimming} or \textbf{$\veps$-simplification} of $T_f$ is the metric subspace of $T_f$ defined by
\be
T_f^\veps := \{\tau \in T_f \,\vert \, h(\tau) \geq \veps\}
\ee
\end{definition}
\begin{notation}
Let us denote $N^\veps$ the number of leaves of $T_f^\veps$.
\end{notation} 


\subsection{A crashcourse in persistent homology}
\label{sec:PH}
Throughout this section, we will detail and give the ideas behind persistent homology. A proper introduction to this is out of the scope of this paper, so we encourage the reader to consult the following classical references about this topic \cite{Hatcher_2002,MacLane_1963,Chazal:Persistence,Oudot:Persistence}. This section aims nonetheless to give a brief introduction compiling the main results and intuition behind this field. To do so, it is convenient to break down the topic along its title. First, we will briefly recall what homology is and how it can be defined, and then we will explain the \textit{persistent} aspect of persistent homology. It goes without saying that a reader familiar with these concepts may skip this section entirely.

\subsubsection{Homology}
In general, the motivation behind the introduction of objects in algebraic topology such as homology is to study topological spaces through algebra. That is, to attach an algebraic object (such as a module, a group, \textit{etc.}) to a topological space, in such a way that, loosely speaking, this algebraic object remains invariant for any two homeomorphic topological spaces. Furthermore, we would like this invariant to behave well with respect to continuous maps. Namely, if we have a continuous map between two topological spaces $f: X \to Y$, we would like to have an induced morphism at the level of the two invariants we attached to $M$ and $Y$. The most famous such invariant for topological spaces is the fundamental group $\pi_1(X)$ first introduced by Poincaré. Some useful references for further reading are \cite{Hatcher_2002, MacLane_1963}. 

In the terms of category theory, the above discussion is equivalent to saying that we use \textit{functors} between the category of topological spaces, \textbf{Top}, and a category of algebraic objects, such as the category of groups, \textbf{Grp}, or that of modules over some ring $R$, $\textbf{Mod}_R$. In this sense, homology is nothing other than a functor $H_* : \textbf{Top} \to \textbf{Mod}_R$. For our purposes, it is sufficient to consider the ring $R$ to be a field $k$, so that we are really working over the category of vector spaces over this field $\Vectk$. We will not detail the precise definitions of these objects, as we will not really need them, but a good reference as an introduction to category theory is given by Mac Lane in \cite{Mac_Lane_1978}.

Recall that finding such a functor entails attaching a vector space to a topological space, in such a way that continuous maps between topological spaces induce linear maps at the level of vector spaces. To render this practical, let us first focus on \textit{triangulable} spaces. Namely, spaces which are homeomorphic to a simplicial complex (a set of \textit{oriented} simplices glued to one-another along edges, $n$-faces or points). Given a simplicial complex $M$, we can define a so-called \textit{chain complex}, which is nothing other than a sequence of vector spaces, denoted $C_*(M,k)$, called the space of \textit{chains} (the star denotes an index, which we call the \textit{degree}), where $C_n$ is the free $k$-vector space generated by the set of $n$-faces in the simplicial complex. For the sake of notational simplicity, whenever the simplicial complex we are talking about and the field over which we are working on is clear, we may drop $M$ and $k$ and denote $C_*(M,k) = C_*$.

We can define a linear map $\del: C_* \to C_*$ called the \textit{boundary map}, which is defined degree-by-degree as follows. The boundary map $\del$ sends the generator of an $n$-face (an element of $C_n$) to the (signed) sum of the generators of its boundary (which are elements of $C_{n-1}$), where the sign in front of each generator is determined by the compatibility of its orientation with the orientation of the $n$-face. With this definition $\del^2 = 0$, which reflects the fact that the boundary of a boundary is always empty. So, we can see a chain complex as some graded vector space $C_*$ along with the map $\del$. Looking at the restriction of $\del$ to each degree of $C_*$, we can write $\del$ as a chain of morphisms 
\be
\cdots \to C_2 \xrightarrow{\del} C_1 \xrightarrow{\del} C_{0} \xrightarrow{\del} 0  \,,
\ee
with the property that $\del^2=0$. This property implies in particular that $\Imag(\del) \subset \ker(\del)$. We call $\ker(\del)$ the space of \textit{cycles}, reflecting the fact that elements of $\ker(\del)$ tend to be ``loops'' or ``cycles'' of $n$-chains. Homology quantifies exactly which cycles are not boundaries. More precisely, fixing a degree $n$, we can define the $n$th homology group over the field $k$ as
\be
H_n(M,k) = \ker(\del \vert_{C_n})/\Imag(\del \vert_{C_{n+1}})\,,
\ee
where $\del\vert_{C_n}$ denotes the restriction of $\del$ to $C_n$. In some sense, this gives a definition of what we mean by an $n$-dimensional hole. All of this discussion is best illustrated by an example. In order to avoid sign problems, it is often practical to work over $k= \Z_2$, so as to simplify calculations. In general, this is restrictive, but it is enough for our purposes and let us fix the simplicial complex $M$ of figure \ref{fig:simpcomp}.
\begin{figure}[h!]
  \centering
    \includegraphics[width=0.2\textwidth]{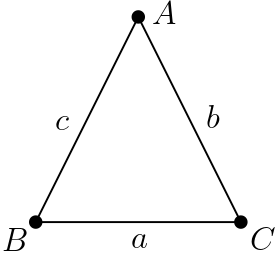}
  \caption{An example of a simplicial complex $M$.}
\label{fig:simpcomp}
\end{figure}
In this case, we see that $C_0 = \bra A, B, C \ket_{\Z_2}$ and $C_1 = \bra a,b,c \ket_{\Z_2}$ with all of the higher chains being $0$. Furthermore, the boundary operator sends 
\be
\del : a \mapsto B+C \;,\; b \mapsto A+C \;, \; c \mapsto A+B \,,
\ee 
with all other generators being sent to zero. From this, it is clear that $\ker(\del\vert_{C_1}) = (a+b+c)\Z_2$, which represents the cycle which loops around the simplicial complex  and $\ker(\del\vert_{C_0})= C_0$. It follows that $H_1(X) = (a+b+c)\Z_2$, so that we indeed detect a hole inside the complex. If we now fill that hole with a cell which fills the triangle (let us note it $\Delta$), we now have $C_2 =\bra \Delta \ket_{\Z_2}$, and its boundary $\del(\Delta)= a+b+c$, so that by filling this ``hole'', we have effectively killed the $H_1$. As for $H_0$, it always quantifies the number of connected components, as (finite) simplicial complexes are connected if and only if they are path connected, so these paths constitute cycles which map any two elements of $C_0$ to each other. In particular, if we suppose that $X$ is path connected, then $\Imag(\del\vert_{C_1})$ is always generated by all the sums of pairs of vertices in $C_0$, so we get one generator of $H_0$ for every connected component. 

At this point, given some triangulable space $Y$, the reader might be worried whether this definition depends on the triangulation we chose for $Y$, but it is a theorem that this is a well-defined invariant of topological spaces, \cf \cite{Hatcher_2002} for details. In fact, we have adopted a rather restrictive point of view throughout this discussion, as in reality we can make sense on how homology can be defined in more general settings \cite{MacLane_1963}.

\subsubsection{Persistent homology}
Now that we have roughly sketched out what homology is, let us introduce the idea behind persistent homology. Once again, we refer the reader to consult the following references if he or she desires a more detailed description of the theory \cite{Chazal:Persistence,Oudot:Persistence}. In this case, instead of dealing with a simple chain complex $C_*(M)$, we induce a filtration on this complex, which is typically done by giving a function on the underlying topological space $M$. For example, if $M$ is a differentiable manifold and $f: M \to \R$ is a smooth function, then we can filter the complex $C_*(M)$ by considering the subsets $M_r = \{f>r\}$ and considering $C_*(M_r)$. We call this filtration of the complex the \textit{superlevel filtration} (an analogous definition can be given for \textit{sublevel filtrations}). Of course, we may then compute the homology of $M_r$ for every $r$, which gives us a family of vector spaces indexed by $r$. However, by the functorial nature of $H_*$, the fact that we have a (continuous) inclusion $i_{r,s}: M_r \xhookrightarrow{} M_s$ for every $r>s$ yields a linear map between the vector spaces $H_*(i_{r,s}): H_*(M_r) \to H_*(M_s)$.

These morphisms are of interest to us, as they tell us about how the homology changes as we vary the level $r$. This motivates the study of the so-called \textit{persistent homology}, which is nothing other than the family of vector spaces $(H_*(M_r))_r$ and the family of morphisms $(H_*(i_{r,s}))_{r>s}$. This is more comfortably expressed in the language of category theory : persistent homology is a \textit{functor} $H_*: \R \to \Vectk$, where $\R$ is seen as a small category (the objects are elements of $\R$ and there is a morphism between $r \to s$ if and only if $r>s$) defined by $H_*(r) = H_*(M_r)$ and $H_*(r\to s) = H_*(i_{r,s})$. 

If the function $f$ is nice enough, for instance $C^1$ and $M$ is compact, the persistent homology induced by the superlevel filtrations of $f$ can be decomposed into so-called \textit{interval modules}. The latter are themselves functors defined as follows. Fixing a field $k$ and if $A$ is an interval of $\R$, then
\be
k_{A}(r) := \begin{cases} k & \text{if } r \in A \\ 0 & \text{else} \end{cases} \quad \quad k_A(r\to s) = \begin{cases} \id & \text{if } r,s \in A \\ 0 & \text{else} \end{cases}
\ee  
The decomposition theorem states the following (\cf Oudot's book \cite{Oudot:Persistence} for a more complete description).
\begin{theorem}[Decomposition theorem, Auslander, Ringel, Tachikawa, Gabriel, Azumaya]
Under some conditions for $M$ and $f$, if $H_*(M,f)$ denotes the persistent homology with values in $\Vectk$, then $H_*(M,f)$ it is isomorphic to a (possibly infinite) direct sum of interval modules. Moreover, this decomposition is unique up to isomorphism and permutation of the terms.
\end{theorem}
This theorem entails that if the filtration function $f$ and the space $M$ are nice enough, the persistent homology functor $H_*$ in fact decomposes as a direct sum of interval modules, more precisely, fixing a degree in homology we have
\be
H_n = \DS_{i} k_{A_i} \,,
\ee
where the $A_i$ are intervals of $\R$. Notice then that this means that the information contained in the persistence module can be encoded by these intervals $A_i$. This collection of intervals is what we call the \textit{barcode} associated to a function $f: M \to \R$, typically denoted $\bcode(f)$. Another way of representing this information is by keeping track of the endpoints of the interval. In this way, we may represent the intervals as a collection of points in the half-plane 
\be
\mathcal{X} := \{(x,y) \in (\R \cup \{\pm \infty\})^2 \, \vert \, x<y\} \,.
\ee 
This collection of points is called the \textit{persistence diagram} associated to $f$, and is typically denoted $\Dgm(f)$.

This is not exactly the full story, as there are some technical caveats to this. Indeed, the theorem requires ``nice enough'' $M$ and $f$. Throughout this paper we will be dealing with $C^0$ funtions (or in $C^0$, up to a finite number of discontinuities), for which these spaces could be of infinite dimension. Crucially, however, if $M$ is compact and $f$ is $C^0$, the rank of the maps $H_*(i_{r,s})$ is always finite. Under this condition, the decomposition theorem above applies so we may consider our modules to be decomposable (\cf \cite{Chazal_2016,Oudot:Persistence} for details). 

Specializing all of this to $H_0$ amounts to talking about connected components of superlevel sets. In this sense, $\dim H_0(M_r)$ is exactly equal to the number of connected components of $M_r$. The rank of $H_0(i_{r,s})$ corresponds to the number of connected components of $M_s$ which contain all of $M_r$. The decomposition theorem can also be easily understood in this setting : bars in the barcode (or equivalently points in the persistence diagram) indicate when a connected component was ``born'' and when it gets absorbed by another one, with the rule that the ``eldest'' connected component is the one which always ``survives''.  

\subsubsection{Persistent homology as a measure}
\label{sec:PHmeasure}
There is one last point to be discussed, which will come in useful in section \ref{sec:WassersteinCvg}. That is, it can be useful to consider persistence diagrams as \textit{measures} over the upper half-plane of $\R^2$. For q-tame modules, we can define such a measure as follows : if $a <b \leq c < d \in \R$, then
\be
\Dgm(f)([a,b] \times [c,d]) := \rk(H_*(i_{b,c})) - \rk(H_*(i_{a,c})) + \rk(H_*(i_{b,d})) - \rk(H_*(i_{a,d})) \,.
\ee 
This turns out to be a measure, once we have ironed out some details, as done in \cite{Chazal:Persistence}. As an example, if the module is decomposable, then the \textit{persistence measure} is nothing other than the sum of Dirac masses at every point of the persistence diagram.
As we will see, considering persistent diagrams as measures has considerable advantages. The reason for this is because the space of measures is a linear space, so that we can operate on the space of diagrams with greater ease than in the algebraic context. This injection onto a linear space yields some desirable properties for persistence diagrams : for example, the ability to define Fréchet means \cite{Turner_2014}, or a notion of mean and expectation. This suggests that this injection is crucial in simplifying the study of persistence diagrams in a probabilistic setting. 
\par
Furthermore, the fact that it is a space of measures allows us to use tools, such as optimal transport. In particular, there are natural notions of distances -- so-called Wasserstein distances -- which stem from this point of view. This approach has been explored in \cite{Divol_2019}, to which we refer the reader for further details. The idea is to use the diagonal as a source of infinite mass, and study the optimal transport distances between diagrams (now seen as measures), where the distance on the base space (the upper-half plane) is the $\ell^\infty$-distance on the plane, namely
\be
\label{eq:dR2infty}
d_{\R^2,\infty}((p,q),(r,s)) = \max\{\abs{p-r}, \abs{q-s}\} \,.
\ee 
\begin{definition}[Wasserstein $p$ distances for diagrams]
\label{def:pWassersteinDiags}
If $\mu, \nu$ are two persistent measures (to which we have added the diagonal $\Delta$ as a source of infinite mass), then the Wasserstein $p$ distance between two diagrams is given by 
\be
d_p(\mu, \nu) := \left[\inf_{\pi \in \Gamma(\mu,\nu)} \int_{\overline{\mathcal{X}}^2} \!\!\!d_{\R^2,\infty}^p(x,y) \;  d\pi(x,y) \right]^{1/p}
\ee
where $\Gamma(\mu,\nu)$ denotes the set of all measures on $\overline{\mathcal{X}}^2$ whose marginals are $\mu$ and $\nu$.
\end{definition}
If $p=\infty$ we take the $\sup$-norm, which is exactly the so-called bottleneck distance referred to in \cite{Oudot:Persistence, Chazal:Persistence}, which we may define as follows, for any two persistent measures $\mu, \nu$,
\be
d_\infty(\mu, \nu) := \inf_{\pi \in \Gamma(\mu,\nu)} \sup_{(x,y) \in \supp(\pi)} d_{\R^2, \infty}(x,y) \,.
\ee 
An important result testifying of why persistent homology is interesting is that it is a stable construct in the following sense.
\begin{theorem}[Stability theorem, \cite{Chazal:Persistence}]
\label{thm:Stability}
Let $(X,d)$ be a compact, triangulable metric space and let $f, g \in C^0(X,\R)$, then if we see $\Dgm(f)$ and $\Dgm(g)$ as persistence measures, then
\be
d_\infty(\Dgm(f),\Dgm(g)) \leq \norm{f-g}_\infty \,.
\ee
\end{theorem}
Fixing a smaller functional space, it is possible to prove a stability theorem for $d_p$ as well. 
\begin{theorem}[Wasserstein $p$ stability, \cite{skraba2020wasserstein}]
\label{thm:LipStability}
Let $(X,d)$ be a compact, triangulable metric space of dimension $D$ and let $\Lip(\Lambda, X)$ denote the space of real valued $\Lambda$-Lipschitz functions on $X$. Then if $f,g \in \Lip(\Lambda, X)$, for every $n \geq D$
\be
d_p(\Dgm_k(f),\Dgm_k(g)) \leq C_{X,\Lambda,p} \,\norm{f-g}_\infty^{1-\frac{n}{p}} \,,
\ee
where $\Dgm_k(f)$ denotes the persistence measure associated to the persistent homology in degree $k$ of $f$, $H_k(X,f)$.
\end{theorem}

\subsection{Trees and barcodes}
There is a correspondence between trees and barcodes described in full detail in \cite{Perez_2020}. Starting from $T_f$, we can look at the longest branch (starting from the root) of $T_f$. This branch corresponds to the longest bar of $\bcode(f)$ since branches of $T_f$ correspond to connected components of the superlevel sets of $f$. Next, we erase  this longest branch and, on the remaining (rooted) forest, look for the next longest branch. This will be the second longest bar of the barcode. Proceeding iteratively in this way, we retrieve $\bcode(f)$. An illustration of this algorithm can be found in figure \ref{fig:algorithm}. 

We can interpret $N^\veps$ geometrically as being equal to the number of leaves of $T_f^\veps$. In terms of the barcode, the same $N^\veps$ counts the number of bars of length $\geq \veps$ with the caveat that we count the infinite bar as having length equal to the range of $f$. As we will see, reasoning in terms of trees has some major advantages, so in what will follow we will adopt the following convention
\begin{convention}
\label{conv:infinitebar}
The length of the infinite bar of $\bcode(f)$ will be set to $\sup f - \inf f$.
\end{convention}
\begin{figure}[h!]
  \centering
    \includegraphics[width=0.5\textwidth]{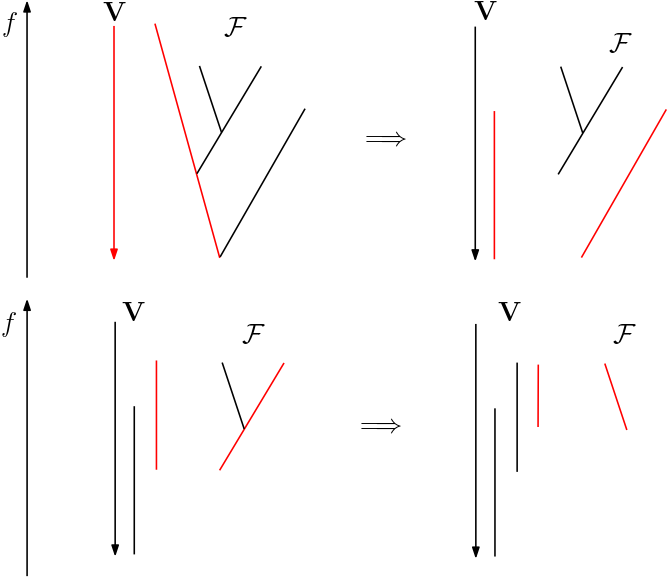}
  \caption{A depiction of the first steps of the algorithm which assigns a barcode $\bcode(f)$ to a tree $T_f$.}
\label{fig:algorithm}
\end{figure}

\subsection{Integration on trees and the duality between $N^\veps$ and $\Pers_p^p$}
\label{sec:TreeIntegration}
Let us recall the following simple remark made in \cite{Perez_2020}.
On a tree $T_f$, we can define a notion of integration by defining the unique atomless Borel measure $\lambda$ which is characterized by the property that every geodesic segment on $T_f$ has measure equal to its length. Formally, we can express $\lambda$ in two ways \cite{Picard:Trees}
\be
\lambda = \int_\R dx \sum_{\substack{\tau \in T_f \\ f(\tau) = x}} \delta_\tau \quad \text{and} \quad \lambda = \int_0^\infty d\veps \sum_{\substack{\tau \in T_f \\ h(\tau) = \veps}} \delta_\tau
\ee
By using the second way of writing $\lambda$, the identity
\be
\lambda(T_f^\veps) = \int_\veps^\infty N^a\;da
\ee
is clear, as every sum in the second expression is finite for all $\veps >0$ and has $N^\veps$ terms. Of course, we could very well have written it using the first sum, but this poses the difficulty that if $T_f$ is infinite, so is the sum considered in this formal expression for at least some value of $x$. However, the restricted sum
\be
\sum_{\substack{\tau \in T_f \\ f(\tau)=x \\ h(\tau)\geq\veps}} \delta_\tau
\ee
is finite for all $\veps >0$ and there are exactly $N^{x,x+\veps}$ terms in this sum. We thus obtain an alternative expression for $\lambda(T_f^\veps)$
\be
\lambda(T_f^\veps)= \int_\R  N^{x,x+\veps} \;dx\,.
\ee
We deduce that more information is contained in $N^{x,x+\veps}$ than in $\lambda(T_f^\veps)$ (and by extension than in $N^\veps$). The calculation above provides the connection between Chazal and Divol and Baryshnikov's functional $N^{x,x+\veps}$ and the functional detailed in this paper $N^\veps$, since $N^\veps$ is nothing other than the derivative of $\lambda(T_f^\veps)$.
\par
The study of $N^\veps$ is in fact completely equivalent to the study of $\Pers_p^p(f)$. Indeed,
\be
\Pers_p^p(f) = p \int_{T_f} h(\tau)^{p-1} \;\lambda(d\tau) =  p\int_0^\infty \veps^{p-1} N^\veps \; d\veps \,,
\ee
where $h : T_f \to \R$ associating to $\tau \in T_f$ the distance between $\tau$ and the highest leaf (with respect to the filtration of $f$) above $\tau$ in $T_f$. We immediately recognize the above integral as being the Mellin transform of $N^\veps$. Allowing for complex $p$, this integral relation can be inverted by virtue of the Mellin inversion theorem, provided that the fundamental strip of $N^\veps$ is not empty. For compact intervals and continuous functions $f$, this fundamental strip is never empty (provided $\Lag(f) < \infty$) and in fact is exactly equal to $\bra \Lag(f), \infty \ket$. Thus, for any real number $c >\Lag(f)$,
\be
N^\veps = \frac{1}{2\pi i} \int_{c-i\infty}^{c+i \infty} \Pers^p_p(f) \,\veps^{-p} \;\frac{dp}{p} \;,
\ee
which estabilished the duality relation desired. Notice also that $\Pers_p^p$ is a norm in the sense that 
\be
\Pers_p^p(f) = p\norm{h}_{L^{p-1}(\lambda)}^{p-1} \,,
\ee

For any (deterministic) continuous function $f$,  $\Pers_p^p(f)$ is nothing other than a sum of the bars of the barcode to the power $p$. An in depth explanation of this is provided in \cite[\S 2.2]{Perez_2020}, but let us briefly give some intuition for this. By the algorithm depicted in figure \ref{fig:algorithm}, if we denote $b$ any of the bars of the barcode, seen as embedded in the tree $T_f$ the length of the branch, $\ell(b)$, can be written as
\be
p\int_b h(\tau)^{p-1} \lambda(d\tau) = \ell(b)^p \,.
\ee 
The bars of the barcode partition the tree $T_f$, so that the integration present in the definition of $\Pers^p_p$ is nothing other than the sum of the $\ell(b)^p$'s.

\begin{remark}
This definition of $\Pers_p^p$ coincides perfectly with a definition of $\Pers_p^p$ typically used in persistent homology \cite{LipschitzStableLpPers,Mileyko_2011,Turner_2014,Divol_2019,Carriere_2017}, as long as we consider that the infinite bar has the length of the range (\ie the $\sup - \inf$) of the function $f$. Of course, within this framework an equally valid definition for $\Pers_p^p$ would have been to exclude the infinite bar from being counted all-together, and to consider only the bars of finite length. This approach turns out to give the correct definition for the $\Pers_p^p$-functional in the definition of \textit{tail} $\zeta$-functions (\cf definition \ref{def:tailzetaLevy}), which is necessary to study L\'evy $\al$-stable processes for $\al <2$.
\end{remark}
Additionally, by the usual inequalities of $L^p$-spaces, 
\begin{proposition}
\label{prop:Lyapunovellpp}
$\Pers_p^p$ is almost $\log$-convex, \ie let $p_0 <p_1$ and $\theta \in [0,1]$ and set $p = (1-\theta)p_0 + \theta p_1$, then, 
\be
\Pers_p^p \leq \frac{p}{p_0^{1-\theta} p_1^{\theta}}\; \Pers_{p_0}^{p_0(1-\theta)}\, \Pers_{p_1}^{p_1\theta} \,.
\ee
\end{proposition}
\begin{proof}
The statement follows directly from an application of Lyapunov's inequality for $L^p$-spaces.
\end{proof}
More generally, it is always true that one can express the $L^p(\mu)$-norm of a function $f$ as the Mellin transform of the repartition function of $\abs{f}$, $\mu(\abs{f}>x)$.

\subsubsection{Calculation of $N^\veps$ in dimension one} 

In dimension one, it is possible to use the total order of $\R$ and count $N^\veps$ by counting the number of times we go up by at least $\veps$ from a local minimum and down by at least $\veps$ from a local maximum. This idea can be formalized by the following sequence, originally introduced by Neveu \textit{et al.} \cite{Neveu_1989}.
\begin{definition}
\label{def:vepsminmax}
Setting $S_0^\veps = T_0^\veps = 0$, we define a sequence of times recursively
\begin{align*}
T_{i+1}^\veps &:= \inf\left\{\, t \geq S_i^\veps\; \Bigg\vert \; \sup_{[S_i^\veps, t]} f -f(t) > \veps\right\} \\
S_{i+1}^\veps &:= \inf\left\{\, t \geq T_{i+1}^\veps \; \Bigg\vert \; f(t) - \inf_{[T_{i+1}^\veps, t]} f > \veps\right\}
\end{align*}
\begin{figure}[h!]
  \centering
    \includegraphics[width=0.6\textwidth]{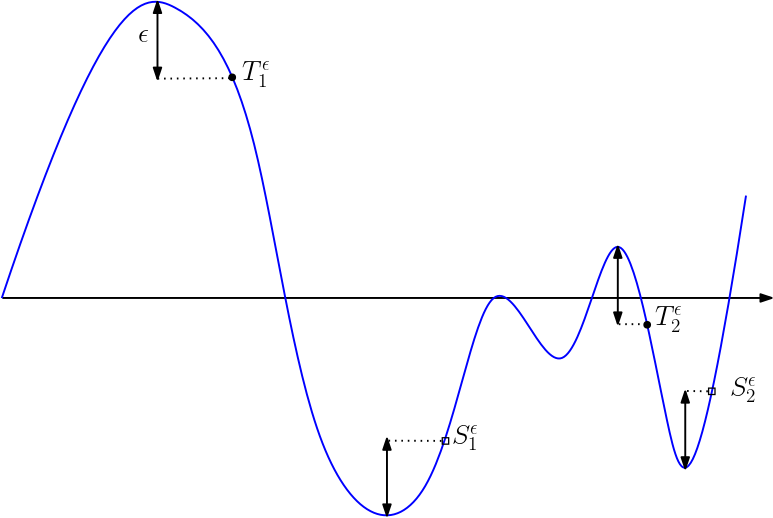}
  \caption{A function $f$ in blue along with the times $T_i^\veps$ and $S_i^\veps$ indicated. Because of the boundary this function has exactly 3 bars of length $\geq \veps$ and not just 2.}
\label{fig:TivepsSiveps}
\end{figure}
\end{definition}
Counting the number of bars of length $\veps$ is thus exactly to count the number of up and downs we make. More precisely
\begin{equation}
\label{eq:Nvepscountingformula}
N^\veps = \inf\{i \, \vert \, T_i^\veps \text{ or } S_i^\veps = \inf \emptyset\}
\end{equation}
by which we mean that it is the smallest $i$ such that the set over which $T_i^\veps$ or $S_i^\veps$ are defined as infima is empty.

\begin{notation}
We denote the range of $X$ $R$. Symbolically,
\be
R_t := \sup_{[0,t]} X - \inf_{[0,t]} X \,.
\ee
Moreoever, denote $N^\veps_t$ the number $N^\veps$ of the process $X$ restricted to the interval $[0,t]$.
\end{notation}

Intuitively, this calculation process hints at the fact that if $\veps$ is small, the number of bars $N^\veps$ should strongly depend on the regularity of the process, as ultimately $N^\veps$ counts the number of ``oscillations'' of size $\veps$. In a very precise sense, regularity almost fully determines the asymptotics of $N^\veps$ in the $\veps \to 0$ regime. This intuition is corroborated by the following theorem.
\begin{theorem}[Picard, \S 3\cite{Picard:Trees} and \cite{Perez_2020}]
\label{thm:pvarandupboxdim}
Given a continuous function $f: [0,1] \to \R$, 
\be
\mathcal{V}(f) = \Lag(f) = \updim T_f = \limsup_{\veps \to 0} \frac{\log N^\veps}{\log(1/\veps)} \vee 1
\ee
where $\updim$ denotes the upper-box dimension, $a \vee b = \max\{a,b\}$,
\be
\mathcal{V}(f) := \inf\{p \, \vert \, \norm{f}_{\pvar} <\infty\} \quad \text{and} \quad \Lag(f) := \inf\{ p \, \vert \, \Pers_p(f) < \infty\} \,.
\ee
\end{theorem}
For the rest of this paper it is exactly the functional $\Pers_p^p$ which shall occupy us. 

\begin{remark}
The $\Pers_\infty$ functional is stable under $L^\infty$ perturbations of $f$. However, it is unknown whether a similar stability result exists for $p <\infty$.
\end{remark}

\subsection{$\zeta$-functions associated to stochastic processes}
\label{sec:Zetafunctions} 
\begin{definition}
Let $f$ be a stochastic process on some compact topological space $X$. Its \textbf{$\zeta$-function} $\zeta_f$ is defined by:
\be
\zeta_f(p) := \expect{\Pers_p^p(f)} = p\int_0^\infty \veps^{p-1}\expect{N^\veps} \; d\veps\,.
\ee
for $p \in \bra \Lag(f), \infty \ket$.
\end{definition}
This is reminiscent of the structure of the $\zeta$-function, but \textit{a priori} not enough to draw any parallels. However, it turns out that this nomenclature turns out to have a meaning for stochastic processes. 
Similarly, we could consider the $\Pers_p^p$ functional of $T_f^{>x}$, which denotes the forest
\be
T_f^{>x} = \{\tau \in T_f \, | \, f(\tau) >x\}
\ee
\begin{remark}
In the tree setting the number $N^{x,x+\veps}$ is also the number of branches of length $\geq \veps$ in the forest $T_f^{>x}$.
\end{remark}
Following this analogy, it is natural to define
\begin{definition}
The \textbf{local $\zeta$-function} associated to $f$ at $x$, $\zeta_f^x$ is defined as 
\be
\zeta^x_f(p) := \expect{\Pers_p^p(T_f^{>x})} = p\int_0^\infty \veps^{p-1} \expect{N^{x,x+\veps}} \; d\veps \,.
\ee
\end{definition}
\begin{proposition}
\label{prop:LocalGlobalZeta}
The two $\zeta$-functions we have so far defined are related via the following formula
\be
\label{eq:locvsnonlocZeta}
\zeta_f(p) = p \int_\R  \zeta^x_f(p-1)\;dx \,.
\ee
\end{proposition}
\begin{proof}
Let us start by noticing that $N^\veps$ is nothing other than 
\be
N^\veps = - \frac{\del}{\del\veps} \lambda(T_f^\veps) = - \frac{\del}{\del\veps} \int_\R  N^{x,x+\veps} \;dx \,,
\ee
where this derivative is defined and locally constant almost everywhere. From this and the fact that the derivative $\frac{\del N^{x,x+\veps}}{\del \veps}$ is also defined and locally constant almost everywhere and  
\be
N^\veps = - \int_\R \frac{\del N^{x,x+\veps}}{\del \veps} \; dx \,.
\ee
Applying the Mellin transform to both sides and applying Tonelli's theorem,
\be
\int_0^\infty \veps^{p-1} N^\veps\;d\veps = -\int_\R dx \int_0^\infty \veps^{p-1} N^{x,x+\veps} \; d\veps \,.
\ee
From the derivation functional property of the Mellin transform (\cf table \ref{table:FunctionalPropertiesMellin}) we get 
\be
\int_0^\infty \veps^{p-1} N^\veps\;d\veps = \int_\R dx \; (p-1) \int_0^\infty \veps^{p-2} N^{x,x+\veps} \;d\veps \,.
\ee
Applying the expectation to both sides, multiplying times $p$ and applying Tonelli's theorem once again, we have the desired result, namely,
\be
\zeta_f(p) = p \int_\R \zeta^x_f(p-1) \; dx \,.
\ee
\end{proof}
\begin{remark}
If the process starts at $0$, it is in general easy to compute $\zeta^x_f$ for $x>0$, but more challenging to do so for $x<0$. In what will follow, we will always focus on $x>0$. 
\end{remark}

\begin{notation}
For the rest of this paper we will take the following conventions. First, we will sometimes omit the subscript $t$ of $N^\veps_t$ whenever convenient. The Laplace transform $\Lag$ is always taken with respect to the variable $t$ and its conjugate variable will always be $\lambda$. Similarly, Mellin transforms will always be taken with respect to the variable $\veps$ and its conjugate variable will be $p$. 
\end{notation}
\begin{lemma}
\label{lemma:commutingLagMel}
Let $f(x,t) : [0, \infty[^2 \to \R_+$  such that the functions $f(x,-)$ and $f(-,t)$ are monotone in their arguments. Then, denoting $\Lag_t$ the Laplace transform with respect to $t$,
\be
\Mel_x \Lag_t [f] = \Lag_t \Mel_x [f]
\ee  
\end{lemma}
\begin{proof}
The monotonicity of $f$ in its arguments ensures that $f$ is a measurable, positive function. The statement holds by virtue of Tonnelli's theorem.
\end{proof}
\begin{remark}
Notice this last lemma is applicable to $N^\veps_t$, $\expect{N^\veps_t}$, $\PP(N^\veps_t \geq k)$ and other such quantities.
\end{remark}

\subsection{Average diagrams and characterizations of Wasserstein convergence of diagrams}
\label{sec:WassersteinCvg}
\begin{notation}
For the rest of this section, denote 
\be
\mathcal{X} := \{(x,y) \in \R^2  \, \vert \, y>x\} \text{ and } \Delta := \{(x,x)\in \R^2\} \,,
\ee
which we equip with the metric $\ell^\infty$-metric on $\R^2$ recalled in equation \ref{eq:dR2infty}. Let $\veps >0$ and denote $\Delta_\veps \subset \overline{\mathcal{X}}$ an open tubular neighbourhood of radius $\veps$ around $\Delta$ inside $\overline{\mathcal{X}}$ and denote its complement in $\mathcal{X}$ by $\Delta_\veps^c$. Finally, denote 
\be
R_{x,\veps} := \,]\!-\!\infty,x]\times[x+\veps,\infty[ \,.
\ee
\end{notation}
By looking at persistence diagrams of functions as measures (\cf section \ref{sec:PHmeasure} for details), it is possible to define a notion of an average diagram of a stochastic process.
\begin{definition} 
Let $M$ be a compact, triangulable metric space, $\mathcal{E}(M)$ be a Polish metric space of (q-tame) functions on $M$ and let $\mathcal{D}$ denote the space of measures on the upper half-plane $\mathcal{X}$. Let $\phi : \mathcal{E}(M) \to \mathcal{D}$ be the map taking $f \mapsto \Dgm(f)$, where $\Dgm(f)$ is seen as a measure, and let $X: \Omega \to \mathcal{E}(M)$ be a stochastic process of law $\mu$ on $\mathcal{E}(M)$. Then, the average diagram of $X$ is given by
\be
\expect{\Dgm(X)}:= \mathbb{E}_\mu[\phi]= \int_{\mathcal{E}(M)} \phi(f) \; d\mu(f) \,.
\ee
\end{definition}
\begin{remark}
\label{rmk:densityofavg}
Note that $\expect{\Dgm(X)}$ is itself a measure. Whenever this measure is absolutely continuous with respect to the Lebesgue measure on $\mathcal{X}$, there exists $g:\mathcal{X} \to \R$ such that for any test function $f: \mathcal{X} \to \R$
\be
\expect{\Dgm(X)}(f) = \int_{\mathcal{X}} g(x,y)f(x,y) \;dx \,dy \,.
\ee
A sufficient condition to ensure the existence of $g$ is that $\del_x \del_y \expect{N^{x,y}}$ exists (this is equivalent to requiring the existence of $\del_x\del_\veps \expect{N^{x,x+\veps}}$). Instead of using birth-death coordinates, we may also express this density in terms of birth-persistence coordinates. In these coordinates, we will denote this density function by $g(x,\veps)$, somewhat abusing the notation. 
\end{remark}
\begin{definition}
The space of measures on $\overline{\mathcal{X}}$ with finite $\Pers_p$ will be denoted $\mathcal{D}_p$. Symbolically, 
\be
\mathcal{D}_p := \left\{\mu \in \mathcal{D} \,\big\vert \, \Pers_p(\mu) := d_p(\mu,\Delta) < \infty \right\} \,,
\ee
where $d_p$ is the Wasserstein $p$-distance on the space of diagrams defined in \cite{Divol_2019} and explicited in definition \ref{def:pWassersteinDiags}.
\end{definition}  
This allows us to define a $\zeta$-function for $\mu \in \mathcal{D}_p$ as follows.
\begin{definition}
Let $\mu \in \mathcal{D}_p$ and suppose that for some $q >p$
\be
\mu(\Delta_\veps^c) = O(\veps^{-q}) \quad \text{as } \veps \to \infty \,,
\ee
then define the \textbf{$\zeta$-function associated to $\mu$} by
\be 
\zeta_\mu(p) := \Pers_p^p(\mu) = p\,\Mel\!\left[\mu(\Delta_\veps^c)\right](p) \,.
\ee
\end{definition}
\begin{remark}
\textit{A priori}, $\zeta_\mu(p)$ is defined on the strip $\bra p, q\ket \subset \C$. This could have also been guaranteed by replacing the condition of decay of $\mu(\Delta_\veps^c)$ by requiring that $\mu \in \mathcal{D}_p \cap \mathcal{D}_q$.
\end{remark}
As before, we may also define a local $\zeta$-function.
\begin{definition}
The \textbf{local $\zeta$-function associated to $\mu$} at $x$ is defined as
\be
\zeta^x_\mu(p) := p \,\Mel[\mu(R_{x,\veps})](p) \,.
\ee
\end{definition}
\begin{remark}
These definitions are compatible with the notions of $\zeta$-functions defined for a stochastic process. Seeing $\expect{\Dgm(X)}$ as a measure on $\overline{\mathcal{X}}$, $\zeta_X = \zeta_{\expect{\Dgm(X)}}$ (and the same holds for local $\zeta$-functions).
\end{remark}
A characterization of the topology metrized by the distance $d_p$ is useful and has been investigated by Divol and Lacombe in \cite{Divol_2019}. The reader familiar with optimal transport will recognize this as an adaptation of the known characterization for probability measures of Wasserstein topology by vague convergence and convergence of $p$th-moments. That this equivalence holds for measures of \textit{a priori} infinite mass is, however, a non-trivial extension.
\begin{lemma}[Characterization of the topology metrized by $d_p$, \cite{Divol_2019}]
\label{lemma:DivolWasserstein}
Let $(\mu_n)_n \subset \mathcal{D}_p$ and $\mu \in \mathcal{D}_p$. Then, the following equivalence holds 
\be
\left\{d_p(\mu_n,\mu) \xrightarrow{n\to\infty} 0\right\} \iff \left\{
        \mu_n \xrightarrow[n\to \infty]{v} \mu \text{ and } \Pers_p(\mu_n) \xrightarrow[n\to \infty]{} \Pers_p(\mu)
\right\} \,.
\ee
\end{lemma}
\begin{remark}
Of course, given our choice of notation, if $p<\infty$, we can rewrite $\Pers_p(\mu_n) \to \Pers_p(\mu)$ as $\zeta_{\mu_n}(p) \to \zeta_{\mu}(p)$.
\end{remark}
Furthermore, it is possible to show that 

\begin{proposition}[Interpolation for optimal transport]
\label{prop:interpolationOT}
Let $1 \leq p < q \leq \infty$ and $\theta \in\, ]0,1[$. Define $p_\theta$ by 
\be
\frac{1}{p_\theta} = \frac{\theta}{p} + \frac{1-\theta}{q} \,.
\ee
Then, for $\mu,\nu \in \mathcal{D}_{p}\cap \mathcal{D}_{q}$ 
\be
d_{p_\theta}(\mu,\nu) \leq  2^{1-\theta} \; d_{p}^\theta(\mu,\nu) \,(\Pers_q(\mu) +\Pers_q(\nu))^{1-\theta}\,.
\ee
Consequently, if $p\leq r \leq q$, then $\mathcal{D}_{r} \subset \mathcal{D}_{p}\cap \mathcal{D}_{q}$.
\end{proposition}
\begin{remark}
For probability measures, Wasserstein interpolation follows trivially from an application of Jensen's inequality. However, since diagrams are \textit{a priori} of infinite mass, this interpolation result needs to be shown.
\end{remark}
\begin{proof}
Let $\pi$ be an optimal transport for $d_{p}$. Applying Littlewood's inequality,
\begin{align*}
d_{p_\theta}(\mu,\nu) &\leq \norm{d_{\R^2,\infty}}_{L^{p_\theta}(\pi)} \leq \norm{d_{\R^2,\infty}}_{L^{p}(\pi)}^{\theta} \norm{d_{\R^2,\infty}}_{L^{q}(\pi)}^{1-\theta} \\
&= d_{p}(\mu,\nu)^\theta \left[\int_{\overline{\mathcal{X}}^2} d_{\R^2,\infty}^q(z,z') \; d\pi(z,z') \right]^{\frac{1-\theta}{q}} \\
&\leq d_{p}(\mu,\nu)^\theta \left[2^q \int_{\overline{\mathcal{X}}^2} d_{\R^2,\infty}^q(z,\Delta) + d_{\R^2,\infty}^q(\Delta,z') \; d\pi(z,z') \right]^{\frac{1-\theta}{q}} \,,
\end{align*}
where this equality holds everywhere on the support of $\pi$. This can be shown by defining 
\be
S=\{(z,z') \in \overline{\mathcal{X}}^2 \cap \supp(\pi) \, \vert \, d_{\R^2,\infty}(z,z') > d_{\R^2,\infty}(z,\Delta) + d_{\R^2,\infty}(z',\Delta) \} \,.
\ee
This set $S$ either has null or positive measure. If it has positive measure, then we can modify the transport plan $\pi$ by sending the projections of $S$ to the diagonal, thereby producing a transport plan of strictly inferior cost to that of $\pi$, which is a contradiction. Hence, $S$ is of null measure, so the equality holds over the support of the measure. Finally, this entails
\begin{align*}
d_{p_\theta}(\mu,\nu) = 2^{1-\theta}\; d_{p}(\mu,\nu)^\theta \,(\Pers_q(\mu)+\Pers_q(\nu))^{1-\theta} \,.
\end{align*}
If $q =\infty$, since $\pi$ is an optimal transport between $\mu$ and $\nu$ and $\mu,\nu \in \mathcal{D}_\infty$, $\pi$ itself must have compact support and the diameter of the support is bounded above by $\Pers_\infty(\mu) \vee \Pers_\infty(\nu)$, so the inequality of the proposition follows.
\end{proof}
\begin{lemma}[Sequential continuity of the inverse Mellin transform]
\label{lemma:seqctyofInverseMellin}
Let $f_k : \;]0,\infty[ \to \C$ be a sequence of functions uniformly bounded by a function $g: \,]0,\infty[ \to \R_+$, whose Mellin transform $g^*$ is defined over some non-empty strip $\bra \al, \beta \ket \subset \C$. Suppose further that there is a function $f : \;]0,\infty[ \to \C$ such that $f^* : \bra \al, \beta \ket \to \R_+$ and $f_k^* \to f^*$ uniformly on every compact set of $\bra \al, \beta \ket$. Then, $f_k \to f$ almost everywhere. 
\end{lemma}
\begin{proof}
For every $k$, $\abs{f_k} \leq g \vee \abs{f}$ and so the sequence $(f_k)_k$ is bounded in a weighted $L^1$ space. Up to extraction of a subsequence, $f_k$ converges a.e. to some function $h$, also bounded above by $g \vee \abs{f}$. By dominated convergence, along this subsequence, $f_k^* \to h^*$, which entails that $h = f$ a.e. since the Mellin transform is injective and $(h-f)^* = 0$ identically on $\bra \al, \beta \ket$, since along \textit{any} subsequence $f_k^* \to f^*$ uniformly on every compact set of $\bra \al, \beta \ket$. It follows that the sequence $(f_k)_k$ has $f$ as its only accumulation point, finishing the proof.
\end{proof}
\begin{remark}
It is sufficient to consider that for all $s \in \,]\al, \beta[$, the sequence $(x^{s-1}f_k(x))_k$ is absolutely uniformly integrable. 
\end{remark}

Both of these lemmas allow for a comprehensive characterization of the objects we have thus far been concerned with throughout this paper. 
\begin{theorem}[$\zeta$ characterization of Wasserstein $p$-convergence]
\label{thm:zetaWasserstein}
Let $(\mu_n)_n \subset \mathcal{D}_{p} \cap \mathcal{D}_q$ be a sequence of q-tame measures and $\mu \in \mathcal{D}_{p} \cap \mathcal{D}_q$ and suppose that the sequence $(\mu_n(\Delta_\veps^c))_n$ can be uniformly bounded above by a function $g: \; ]0,\infty[\, \to \R_+$ such that for $\veps \in \;]0,1]$  $g(\veps) = O(\veps^{-p})$ as $\veps\to 0$ and on $[1,\infty[$, $g(\veps) = O(\veps^{-q})$ as $\veps \to \infty$. Then, the following are equivalent:
\begin{enumerate}
  \item There exists $p<r<q$ such that $d_r(\mu_n,\mu) \xrightarrow{n\to\infty} 0$.
  \item There exists $p<r<q$ such that $\mu_n \xrightarrow[n\to \infty]{v} \mu$ and $\zeta_{\mu_n}(r) \to \zeta_\mu(r)$.
  \item For almost every $x\in \R$ and $\veps >0$, $\mu_n(\Delta_\veps^c) \to \mu(\Delta_\veps^c)$ and $\mu_n(R_{x,\veps}) \to \mu(R_{x,\veps})$.
  \item For almost every $x \in \R$, $\zeta^x_{\mu_n} \to \zeta^x_\mu$ and $\zeta_{\mu_n} \to \zeta_\mu$ uniformly on every compact of $\bra p,q \ket$. 
  \item For all $p < r < q$, $d_r(\mu_n,\mu) \xrightarrow{n\to\infty} 0$.
\end{enumerate}
\end{theorem}
\begin{proof}
Let us break down the proof in different steps.
\begin{itemize}
  \item $(1) \iff (2)$ by lemma \ref{lemma:DivolWasserstein}. 
  \item $(2) \implies (3)$. By the Portmanteau theorem, the vague convergence of $\mu_n$ to $\mu$ entails that for any continuity set $A$ of $\mu$,
  $\mu_n(A) \to \mu(A)$. In particular, since the $\mu_n$ and $\mu$ are q-tame, $R_{x,\veps}$ is a continuity set of $\mu$ for almost every $x$ and $\veps$, so $\mu_n(R_{x,\veps}) \to \mu(R_{x,\veps})$ and $\mu_n(\Delta_\veps^c) \to \mu(\Delta_\veps^c)$ a.e..
  \item $(3) \implies (4)$. Since $\mu_n(\Delta_\veps^c) \to \mu(\Delta_\veps^c)$ a.e., if we let $s \in \bra\al,\beta\ket$, then
  \begin{align*}
  \abs{\zeta_{\mu_n}(s) - \zeta_\mu(s)} &\leq \abs{s}\int_0^\infty \veps^{\Real(s)-1} \abs{\mu_n(\Delta_\veps^c)-\mu(\Delta_\veps^c)} \;d\veps \\
  &= \abs{s}\left\{\int_0^1 +\int_1^\infty\right\} \veps^{\Real(s)-1} \abs{\mu_n(\Delta_\veps^c)-\mu(\Delta_\veps^c)} \;d\veps
   \end{align*}
  The domination conditions of the theorem on $\mu_n(\Delta_\veps^c)$ and $\mu(\Delta_\veps^c)$ guarantee that this quantity is integrable as soon as $p < \Real(s) <q$. By dominated convergence, $\zeta_{\mu_n} \to \zeta_\mu$ on every compact of $\bra p, q \ket$. We apply the same reasoning to the $R_{x,\veps}$ by noting that the measure of $R_{x,\veps}$ is always dominated by that of $\Delta_\veps^c$.
  \item $(4) \implies (5)$. Fix $r \in\; ]p,q[$ and take a compact set $K$ around $r$ contained in $\bra p,q\ket$. It suffices thus to show that the convergence of these $\zeta$-functions entails vague convergence of the $\mu_n$, which will show the result by lemma \ref{lemma:DivolWasserstein}. Denoting $A_{\veps} \in \{R_{x,\veps},\Delta_\veps^c\}$, the boundedness condition of the theorem on $\mu_n(A_\veps)$ entails that, by lemma \ref{lemma:seqctyofInverseMellin}, $\mu_n(A_\veps) \to \mu(A_\veps)$ almost everywhere. Since the collection of $R_{x,\veps}$ and $\Delta_\veps^c$ together forms a $\pi$-system, so for every continuity set of $\mu$, we have convergence of their measures $\mu_n$ to $\mu$, and so by applying the Portmanteau theorem once again, $\mu_n \to \mu$ vaguely, which shows the result.
  \item $(5) \implies (1)$ trivially.
\end{itemize}
\end{proof}
\begin{remark}
Of course, by interpolation, it is sufficient that for all $\mu$ and $\Real(s)>p$, $\zeta_{\mu_n}(s)$ be bounded to guarantee that for any $s \in \;] p, \infty[$, $d_s(\mu_n,\mu) \to 0$, if $d_p(\mu_n,\mu)\to 0$. 
\end{remark}
\begin{remark}
Theorem \ref{thm:zetaWasserstein} applies to average diagrams from any stochastic process satisfying the hypotheses of the theorem. This entails that if the $\zeta$-functions of a process converge on some open set $\bra \al, \beta \ket \subset \C$, the underlying expected diagrams themselves converge in some Wasserstein distance.
\end{remark}
\section{$\zeta$-functions of L\'evy processes and semimartingales}

\subsection{Semimartingales}
A first important motivating result regarding $\zeta$-functions is the following. 
\begin{proposition}
\label{prop:semimartingaleszeta}
Let $X$ be a continuous semimartingale $X = M + A$ on the interval $[0,t]$ such that for $s \geq 1$
\be
\expect{[M]_{t}^{s/2} + \left(\int_0^t \abs{dA}_s\right)^s \,} < \infty \,.
\ee 
Then, the (local) $\zeta$-function of $X$ is meromorphic on $\Real(p)\geq 2$ (resp. $\Real(p)\geq 1$) with a single simple pole at $p=2$ (resp. $p=1$). Furthermore, if $X$ is a continuous semimartingale (and some conditions) and $[X]_t < \infty$, then in $L^s$
\be
N^\veps_X \sim \frac{[X]_t}{2\veps^2} \quad \text{as } \veps \to 0 \,.
\ee
\end{proposition}
\begin{remark}
By quadratic variation, we mean the quadratic variation in the probabilistic sense, \ie in the sense that we take the limit of the variation of the process as the mesh of the partition of the interval $[0,t]$ goes to zero, as opposed to the \textit{real} quadratic variation, where the supremum over all partitions is considered.
\end{remark}
\begin{proof}
The statement holds by the almost sure existence of continuous modifications of local times for continuous semimartingales and the fact that uniformly in $x$ and $t$, it is known \cite[Ch. VI Thm. 1.10]{Revuz_1999} that under the technical hypothesis of the theorem,
\be
2\veps N^{x,x+\veps}_X \xrightarrow[\veps \to 0]{L^s} L^x_X(t) \,.
\ee
Recall the density occupation formula for the local times of a continuous semimartinglale $X$ \cite[Cor. 9.7]{LeGall:BrownianMotion}, which states that, almost surely, for every $t \geq 0$ and any non-negative measurable function $\phi$ on $\R$
\be
\int_0^t \phi(X_\tau) \; d[X]_\tau = \int_{\R} \phi(a) L^a_X(t) \;da \,.
\ee
Taking $\phi$ to be the constant function $1$, we get that in $L^s$
\begin{align*}
\lambda(T_f^\veps) &= \int_\R N^{x,x+\veps} \; dx \sim \frac{1}{2\veps}\int_\R L_X^x(t) \; dx  \;\; \text{as } \veps \to 0\\
&= \frac{1}{2\veps} \int_0^t d[X]_\tau  = \frac{[X]_t}{2\veps} + o(\veps^{-1}) \;\; \text{as }\veps \to 0 \,.
\end{align*}
But for $\delta>0$ small enough, by monotonicity of $N^\veps$,
\be
N^{(1+\delta)\veps} \leq \frac{\lambda(T_f^\veps)- \lambda(T_f^{(1+ \delta)\veps})}{\delta \veps} = \frac{1}{\delta\veps}\int_\veps^{(1+\delta)\veps} N^a \; da  \leq N^\veps  \,,
\ee
so that in $L^s$ for every $\delta >0$ small enough,
\be
\frac{[X]_t}{2\veps^2}\frac{1}{1+\delta} \lesssim N^\veps \lesssim \frac{[X]_t}{2\veps^2} \frac{1}{1-\delta} \quad \text{as }\veps \to 0 \,.
\ee
Since $\delta$ can be taken to be arbitrarily small
\be
N^\veps \sim \frac{[X]_t}{2\veps^2} \quad \text{as } \veps \to 0 \,
\ee
in $L^s$ as desired.
\end{proof}
\begin{remark}
If the continuous semimartingale is Brownian motion, the same arguments as well as a result regarding the representation by downcrossings of the local time by It\^o \cite[\S 2.4]{Ito_1996} show that this asymptotic relation holds in fact almost surely. 
\end{remark}
\begin{corollary}
If $X$ is a continuous semimartingale, then in expectation $\Pers_p^p(X)$ admits a pole of order $1$ at $p =2$ of residue $[X]_t$.
\end{corollary}

In \cite{Perez_Pr_2020}, we have already studied the functional $N^\veps$ for Markov processes. Let us briefly recall some useful known facts about $N^\veps$. 
\begin{proposition}[P, \cite{Perez_Pr_2020}]
\label{prop:p2resume}
Using summation by parts, it is possible to write
\be
\expect{(N^\veps_t)^s} = \sum_{k\geq 1} (k^s - (k-1)^s) \PP(N^\veps_t \geq k)
\ee
For processes on the interval which are not periodic (in the sense of \cite{Perez_Pr_2020}), if $k \geq 2$
\begin{equation}
\label{eq:NtSk}
\PP(N_t^\veps \geq k)  = \PP(S_{k-1}^\veps \leq t) \,,
\end{equation}
and $\PP(N_t \geq 1) = \PP(R_t \geq \veps)$. Furthermore if $X$ has the strong Markov property,
\be
\expect{N_t^\veps} \sim \PP(R_t \geq \veps) \quad \text{as } \veps \to \infty\,.
\ee  
Finally, for $k \geq 2$ the Laplace transform (with respect to time, as per our convention) of equation \ref{eq:NtSk} is 
\begin{equation}
\label{eq:by339}
\Lag(\PP(N_t^\veps \geq k))(\lambda) = \frac{\expect{e^{-\lambda S_{k-1}^\veps}}}{\lambda} \,.
\end{equation}
\end{proposition}
\begin{proof}[Proof of proposition \ref{prop:p2resume}]
The only thing to prove is the result of equation \ref{eq:by339}, as the previous statements are all proved in \cite{Perez_Pr_2020}. Since we are dealing with processes which are not periodic in the sense of \cite{Perez_Pr_2020}, then
\be
\PP(N_t^\veps \geq k) = \PP(S_{k-1}^\veps \leq t) \,,
\ee
since as soon as the hitting time $S_{k-1}^\veps$ is attained we have at least $k$ bars (due to the boundary of the interval). Using standard functional properties of the Laplace transform it is easy to see that 
\be
\Lag[\PP(S_{k-1}^\veps \leq t)](\lambda) = \frac{\Lag[\PP(S_{k-1}^\veps = t)](\lambda)}{\lambda} \,,
\ee 
where $\PP(S_{k-1}^\veps =t)$ denotes the probability density function of $S_{k-1}^\veps$. However, the Laplace transform of this density function is nothing other than the moment generating function $\expect{e^{-\lambda S_{k-1}^\veps}}$, since $S_{k-1}^\veps$ is a positive random variable.
\end{proof}
\begin{remark}
\label{rmk:whyLaplace}
If the process has the strong Markov property, we can write $\expect{e^{-\lambda S_{k-1}^\veps}}$ as the product of the Laplace transform of the distribution of its increments
\be
S_{k-1}^\veps = \sum_{i=0}^k (S_i^\veps- T_i^\veps) + (T_i^\veps- S_{i-1}^\veps) \,.
\ee
The expression of $\expect{e^{-\lambda S_{k-1}^\veps}}$ is particularly simple as soon as these increments are independent and identically distributed.
\end{remark}
\begin{remark}
Ordering the bars of the barcode of a function $f$ by their length, and denoting the length of the $k$th longest branch by $\ell_k$, the following equivalence holds
\be
N_t^\veps \geq k \iff \ell_k \geq \veps \,.
\ee
The probability distribution of both of the events above are thus the same. Consequently, there is a one-to-one correspondence between the elements of the sums
\be
\expect{\Pers^p_p} = \sum_{k \geq 1} \expect{\ell_k^p} = p\sum_{k \geq 1} \Mel[\PP(N^\veps_t \geq k)](p)
\ee
whenever these quantities are defined. In particular, the distribution of each bar is in principle readily available, since $\expect{\ell_k^{p-1}}$ is the Mellin transform of the distribution of $\ell_k$. We will later see that in particular cases, we can gain access to the explicit distribution of bars in this way (\cf section \ref{sec:distlengthbars}). 
\end{remark}
\subsection{Renewal theory}
Throughout this section, we shall consider that the stopping times $S_{k-1}^\veps$ are such that 
the sequence $(U_k^\veps)_{k \geq 0}$ is a sequence of i.i.d. atomless random variables, where
\be
U_k^\veps := S_k^\veps- S_{k-1}^\veps \,.
\ee 
We shall adapt and enounce some theorems from renewal theory (\cf Allan Gut's book \cite{Gut_2009} for a detailed account of this) Let us define $\hat{N}^\veps_t$ by 
\be
\hat{N}^\veps_t := \max\{k \, \vert \, S_k \leq t\} \,.
\ee 
Tautologically, for every $k$, $\hat{N}^\veps_t \geq k \iff S_{k}^\veps \leq t$. In fact, adopting the convention that the infinite bar has length the range of the process, we can write, always under convention \ref{conv:infinitebar} that
\be
N_t^\veps  = 1_{\{R_t \geq \veps\}} + \hat{N}^\veps_t 
\ee 
\begin{remark}
Notice that convention \ref{conv:infinitebar} for the infinite bar might not be the most natural or more convenient from this point of view. Indeed, if we had adopted the convention that the infinite bar \textit{always} has length at least $\veps$, then
\be
N_t^\veps = 1 + \hat{N}^\veps_t 
\ee
and $N^\veps_t$ is a first passage time process defined by 
\be
N^\veps_t := \min\{k \, \vert \, S_k > t\} \,.
\ee
Nonetheless, as previously stated, we will keep adopting convention \ref{conv:infinitebar} for the rest of this paper. 
\end{remark}
According to renewal theory, $\hat{N}^\veps_t$ posses various desirable limit theorems which we shall briefly recall and refer the reader to \cite{Gut_2009} for complete proofs of the latter. 
\begin{theorem}[Strong Law of Counting Processes, Theorem 5.1 \cite{Gut_2009}]
\label{thm:SLCP}
Let $0<\expect{U_1^\veps}<\infty$, then 
\be 
\frac{\hat{N}^\veps_t}{t} \xrightarrow[a.s]{t \to \infty} \frac{1}{\expect{U_1^\veps}}  \quad \text{and} \quad  \frac{\expect{(\hat{N}^\veps_t)^s}}{t^s} \xrightarrow{t \to \infty} \frac{1}{(\expect{U_1^\veps})^s}
\ee 
for all $s >0$. If $\expect{U^\veps_1} = \infty$, then the limits are $0$.
\end{theorem}

\begin{theorem}[CLT for Counting Processes, Theorem 5.2 \cite{Gut_2009}]
\label{thm:CLTCP}
Let $0 < \expect{U_1^\veps} = \mu <\infty$ and $\sigma^2 = \Var(U_1^\veps) <\infty$, then 
\be
\frac{\hat{N}_t^\veps - t/\mu}{\sqrt{\frac{\sigma^2 t}{\mu^3}}} \xrightarrow[\PP]{t \to \infty} \mathcal{N}(0,1)
\ee
Furthermore, 
\begin{align*}
\expect{\hat{N}_t^\veps} &= \frac{t}{\mu} + \frac{\sigma^2 - \mu^2}{2\mu^2} + o(1) \quad \text{as }  t\to \infty \\
\Var(\hat{N}_t^\veps) &= \frac{\sigma^2 t}{\mu^3} + o(t)\quad \text{as } t\to \infty \,.
\end{align*}
\end{theorem}
These results are in particular applicable for Lévy processes, where it is possible to show that the requirement that the sequence $(U_k)_k$ is i.i.d. is satisfied. In particular, the theorems above give the asymptotic long-time behaviour of the number of bars. 

\subsection{L\'evy processes}
For L\'evy processes, the small scale asymptotics of $N^\veps$ can also be studied up to the following caveat : a wide range of L\'evy processes have almost surely discontinuous paths (but nonetheless càdlàg), but our construction of trees (as done in \cite{Perez_2020}) is based on continuous functions. For this reason, it is necessary to define what tree we associate to a process $X$ when $X_t$ has almost surely discontinuous paths. Luckily, this caveat has been treated for càdlàg processes in \cite{LeGall:Trees, Picard:Trees}. We will adopt the approach taken by Picard in \cite{Picard:Trees}, where the reader can find the details of the construction. Loosely speaking, Picard's approach consists in ``completing'' the function at the discontinuity points by joining an imaginary line linking the points of discontinuity (\cf figure \ref{fig:DiscontinuousTree}). 
\begin{figure}[h!]
\centering
\includegraphics[width=0.7\textwidth]{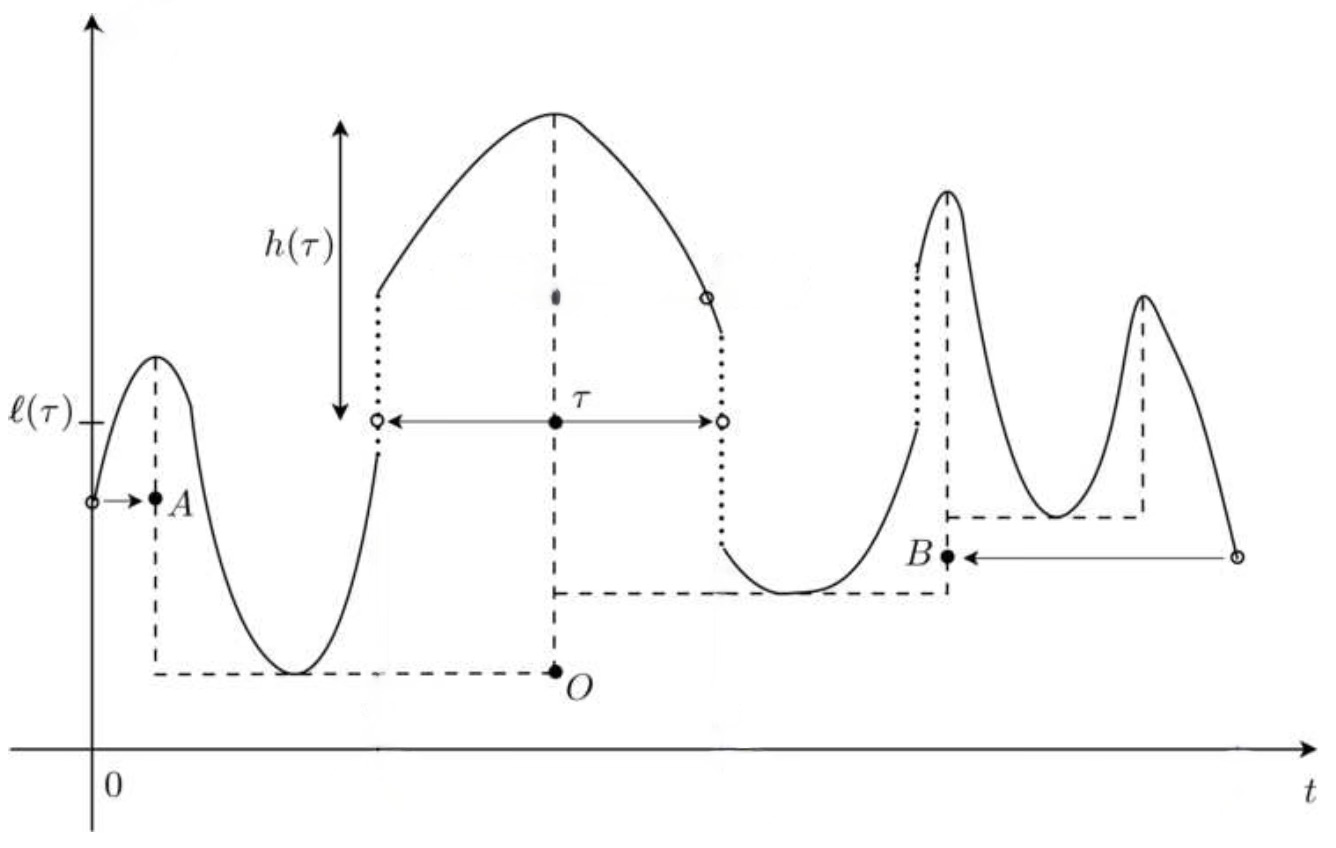}
\caption{A depiction of the construction of a tree associated to a càdlàg function. The figure is taken from \cite{Picard:Trees}}
\label{fig:DiscontinuousTree}
\end{figure}

In any case, it has been shown that on some fixed interval $[0,t]$, it is possible to obtain the behaviour of the number of bars of length $\geq \veps$ as $\veps \to 0$. 
\begin{proposition}[Picard, \S 3 \cite{Picard:Trees}]
\label{prop:PicardProp}
Let $X$ be a L\'evy process and suppose that, almost surely $X$ has no interval on which it is monotone. Define
\be
\xi(\veps) := \expect{S^\veps + T^\veps}
\ee
for 
\be
S^\veps := \inf\{t \, | \, X_t - \inf_{[0,t]} X > \veps\} \quad \text{ and } T^\veps := \inf\{t \, | \, \sup_{[0,t]} X - X_t > \veps\} \,,
\ee
then $\xi(\veps) N^\veps \to 1$ as $\veps \to 0$ in probability. If $\xi(\veps) = O(\veps^\al)$ for some $\al$, then the convergence is almost sure.
\end{proposition}
\begin{remark}
The hypothesis on $X$ is satisfied if $X$ or $-X$ is not the sum of a subordinator and a compound Poisson process, in which case $T_X$ is finite, so $N^\veps$ is bounded. Furthermore, the convergence is always almost sure for $\al$-stable processes for which $\abs{X}$ is not a subordinator by the scaling property. In fact, in that case there exists a constant $C_\al$ such that almost surely,
\be
N^\veps  \sim \frac{C_\al}{\veps^\al} \quad \text{as }\veps \to 0 \,.
\ee
If we can quantify correction terms to this asymptotic relation in $L^1$, this gives rise to a statistical test for $\al$ by using the stability results discussed in \cite{Perez_Pr_2020}, we will explore this in more detail in section \ref{sec:StatisticalTest}. By the self-similarity of $\al$-stable process following the arguments of \cite[\S 3]{Picard:Trees}, we can already at least conclude that 
\be
\abs{\expect{N^\veps_\al}-C_\al \veps^{-\al}} \leq 1 \,.
\ee
\end{remark}
\begin{notation}
\label{notation:Uveps}
In what will follow, we will denote by $S^\veps$ and $T^\veps$ two independent random variables distributed as the analogously denoted ones in proposition \ref{prop:PicardProp}. Furthermore, define $U^\veps = T^\veps + S^\veps$. In particular, if $\veps = 1$, abusing the notation we will denote $U^1 = U$.
\end{notation}
\begin{remark}
Henceforth, unless otherwise specified, we will always assume that $X$ almost surely has no interval on which it is monotone. 
\end{remark}
This result by Picard is exactly the Strong Law of Counting Processes (theorem \ref{thm:SLCP}) applied to the stopping times we previously considered. Self-similarity allows us to trade the $t \to \infty$ limit by an $\veps \to 0$ limit. However, Picard's result is more general that what could've been deduced with self-similarity, as it applies to Lévy processes which are not necessarily $\al$-stable. In this setting, one could ask whether it is possible to prove an analogous theorem to the CLT of Counting Processes (theorem \ref{thm:CLTCP}). 
\begin{theorem}
\label{thm:meroextensionLevy}
Let $X$ be a L\'evy process such that almost surely $X$ has no interval on which it is monotone, using the notation defined in \ref{notation:Uveps}, 
\be
\expect{N^\veps_t} = \frac{t}{\expect{U^\veps}} + \left(\frac{\expect{(U^\veps)^2}}{2\expect{U^\veps}^2}-1\right) + \PP(R_t \geq \veps) + o(\rho_\veps^{-n}) \quad \text{as }\veps \to 0 \,.
\ee
for any $n \in \N$, where $\rho_\veps$ denotes the radius of convergence of the Taylor series of $\expect{e^{-\lambda U^\veps}}$ around $\lambda=0$, which can be bounded below by $-\log(\PP(T^\veps >1)\vee \PP(S^\veps >1))$, which is larger than $1$ for $\veps$ small enough. Furthermore, if $X$ is $\al$-stable, the formula above becomes 
\begin{equation}
\expect{N^\veps_t} = \frac{t}{\expect{U}\veps^\al} + \frac{\expect{U^2}}{2\expect{U}^2}  + o(\veps^{\al n}) \quad \text{as } \veps \to 0\,.
\label{eq:NvepstAlphaStable}
\end{equation}
\end{theorem}
Let us stress that the main addition over the CLT of Counting Processes (theorem \ref{thm:CLTCP}) is two-fold : 
\begin{itemize}
  \item We have found a similar limit for a processes which is not self-similar, and so for which the $t \to \infty$ limit was \textit{a priori} not interchangeable with the $\veps \to 0$ limit. 
  \item We have further specified the rate of decay of the remainder in terms of $\veps$, and this will turn out to be of importance;
\end{itemize}
To show the theorem, it is convenient to show first some technical lemmata, one of which is a slight refinement to a technical lemma proved in \cite{Picard:Trees}. 
\begin{lemma}[Picard, Proposition 3.14 \cite{Picard:Trees}]
\label{lemma:SvepsTvepsMoments}
The variables $S^\veps$ and $T^\veps$ admit finite moments of order $k$ for all $k$ and the moment generating function $\expect{e^{-\lambda U^\veps}}$ is well defined on a neighborhood of zero. Furthermore, the radius of convergence of $\expect{e^{-\lambda U^\veps}}$ around $0$, $\rho_\veps$, then 
\be
\rho_\veps \geq -\log(\PP(S^\veps >1) \vee \PP(T^\veps >1)) \,,
\ee
if $X$ is $\al$-stable, then $\rho_\veps = \rho \veps^{-\al}$, for some constant $\rho>0$ (which might be infinite). Finally, there exists $1 \leq C_k \leq  2^k \Li_{-k}(\frac{1}{2})$ such that
\be
\expect{U^\veps}^k \leq \expect{(U^\veps)^k} \leq C_k \,\expect{U^\veps}^k \,.
\ee
\end{lemma}
\begin{remark}
The bound on the constant in this lemma is not optimal. Notice also that this result entails $\rho_\veps \to \infty$ as $\veps \to 0$. 
\end{remark}

\begin{lemma}
\label{lemma:Uvepsto0}
Keeping the same notation (\cf notation \ref{notation:Uveps}) as before,
\be
U^\veps \xrightarrow[\veps \to 0]{L^r} 0  \quad \text{and} \quad U^\veps \xrightarrow[\veps \to 0]{\text{a.s.}} 0 \,.
\ee
for every $r \geq 1$.
\end{lemma}

\begin{lemma}
\label{lemma:usefulinequalities}
For every $k$, there exists a constant $D_k$ such that
\be
1- D_k e^{-\lambda} \expect{U^\veps}^k \lesssim \expect{e^{-\lambda U^\veps}} \leq 1 \quad \text{as } \veps \to 0 \,.
\ee
Furthermore, for real $\gamma$ and $\sigma$,
\be
1-\expect{e^{-\gamma U^\veps}} \leq \abs{1- \expect{e^{-(\gamma+i\sigma)U^\veps}}} \,, 
\ee
In particular,
\be
\abs{\frac{\expect{e^{-(\gamma+i \sigma)U^\veps}}}{1- \expect{e^{-(\gamma+i\sigma)U^\veps}}}} \leq \frac{\expect{e^{-\gamma U^\veps}}}{1- \expect{e^{-\gamma U^\veps}}} \,.
\ee
\end{lemma}

\begin{lemma}
\label{lemma:weakcvg}
For any $\delta >0$, weakly in $L^2([\delta,\infty[)$, for any $k \geq 0$
\be
\frac{1}{2\pi i} \int_{\gamma - iT}^{\gamma+iT} e^{\lambda t} \lambda^k \; d \lambda \xrightarrow[T \to \infty]{\mathcal{D}} 0 
\ee
at a rate $O(T^{-n})$ for any $n \in \N$. Furthermore, for any $k \geq 0$ 
\be
\frac{1}{2\pi i} \int_{\gamma - iT}^{\gamma+iT} \frac{e^{\lambda t}}{\lambda^k} \; d \lambda \xrightarrow[T \to \infty]{\mathcal{D}} \frac{t^{k-1}}{(k-1)!}
\ee
weakly at a rate $O(T^{-(n+k-1)})$ and the convergence is strong as soon as $k \geq 2$. 
\end{lemma}

\begin{proof}[Proof of lemma \ref{lemma:SvepsTvepsMoments}]
It is sufficient to show it for $S^\veps$ knowing that an analogous treatment is possible for $T^\veps$. The points in $[0,S^\veps]$ are characterized by the fact that
\be
X_t - \inf_{[0,t]}X < \veps \,.
\ee 
In particular, the supremum over all $t$ ranging within $[0,S^\veps]$ of this quantity is also less than $\veps$. 
\be
\PP(S^\veps > a) = \PP\!\left[\sup_{0\leq t\leq a} \left(X_t - \inf_{[0,t]} X\right) < \veps\right] \,.
\ee
Consider now an interval $[0,k\mu]$, where $\mu >0$ and $k$ is an integer and let us slice this interval into $k$ segments of length $\mu$. It is clear that the following inequality holds
\be
\sup_{0 \leq t \leq k\mu}\left(X_t - \inf_{[0,t]} X\right) \geq \sup_{1\leq j \leq k} \left(\sup_{(j-1)\mu \leq t \leq j\mu} \left(X_t -\inf_{[(j-1)\mu,t]} X\right)\right) \,,
\ee
since over each smaller interval we can have only smaller spread than over the entire interval. However, the right hand side is a supremum over i.i.d. random variables, so that
\begin{align*}
\PP(S^\veps > k\mu) &\leq \PP\!\left[\sup_{1\leq j \leq k} \left(\sup_{(j-1)\mu \leq t \leq j\mu} \left(X_t- \inf_{[(j-1)\mu,t]} X\right)\right) < \veps\right] \nonum \\
&= \PP\!\left[\sup_{0\leq t \leq \mu} \left(X_t -\inf_{[0,t]} X\right) < \veps\right]^k = \PP(S^\veps >\mu)^k \,.
\end{align*}
By the non-monotonicity of $X$, $\PP(S^\veps > \mu) <1$. In particular, if we let $\mu = 1$, and denote $c = \PP(S^\veps > 1)<1$, then:
\be
\lim_{k \to \infty} e^{\lambda k} \,\PP(S^\veps>k) \leq \lim_{k \to \infty} e^{\lambda k} \,\PP(S^\veps >1)^k = \lim_{k\to \infty} (e^{\lambda}c)^k =0
\ee
as soon as $\lambda<\log(1/c)$. It follows that $\expect{e^{-\lambda S^\veps}}$ is well-defined for $\lambda$ in some neighbourhood of zero and in particular all moments of $S^\veps$ are well-defined and finite. Finally, combining the above remark with an application of Markov's inequality we get
\be
\PP(S^\veps > 2k\expect{S^\veps}) \leq 2^{-k} \,.
\ee
Almost surely, $\frac{S^\veps}{2\mathbb{E}[S^\veps]} \leq G$, where $G$ is a geometric random variable. The moments of $G$ can easily be calculated, yielding the estimation in the lemma for the moments..

This shows that the radius of convergence of the Taylor series $\expect{e^{-\lambda S^\veps}}$ is bounded below by $-\log(\PP(S^\veps>1))$, since Taylor series converge up to their nearest singularity. If $X$ is $\al$-stable, we can use the relation which tells us that, in distribution $U^\veps = \veps^\al U$, so that by the ratio test,
\be
\limsup_{k \to \infty} \frac{\abs{\lambda}}{k+1} \; \frac{\expect{(U^\veps)^{k+1}}}{\expect{(U^\veps)^k}} =  \veps^\al\abs{\lambda} \limsup_{k\to \infty} \frac{1}{k+1} \; \frac{\expect{U^{k+1}}}{\expect{U^k}} \,.
\ee
This limit is equal to $0 \leq \rho^{-1}<\infty$, since $\expect{e^{-\lambda U}}$ is analytic on a non-trivial disk around $0$, by lemma \ref{lemma:SvepsTvepsMoments} and usual properties of the Laplace transform.
\be
\limsup_{k \to \infty} \frac{\abs{\lambda}}{k+1} \; \frac{\expect{(U^\veps)^{k+1}}}{\expect{(U^\veps)^k}}  <1 \,,
\ee 
whenever $\abs{\lambda} < \rho \veps^{-\al}$, which shows the desired result.

\end{proof}

\begin{proof}[Proof of lemma \ref{lemma:Uvepsto0}]
The statement in $L^r$ follows from the following observation. 
\be
0 \leq \expect{(U^\veps)^r} \leq \sum_{k=1}^\infty k^r\PP(U^\veps \geq k) \leq \sum_{k=1}^\infty k^r \PP(U^\veps \geq 1)^k
\ee
by the arguments of lemma \ref{lemma:SvepsTvepsMoments}. This sum converges, since $\PP(U^\veps \geq 1) <1$. As $\veps \to 0$, $\PP(U^\veps > 1) \to 0$, since $X$ is almost surely nowhere monotone, so that the entire sum tends to $0$. The almost sure statement follows from the fact that $U^\veps$ is monotone, since both $T^\veps$ and $S^\veps$ are monotone functions of $\veps$. Since $L^r$ convergence implies almost sure convergence along a subsequence $\veps_n$, for $\veps_{n+1} <\veps <\veps_{n+1}$ by monotonicity of $U^\veps$ 
\be
U^{\veps_{n+1}} < U^\veps < U^{\veps_{n}} \,,
\ee
so the convergence is almost sure.
\end{proof}

\begin{proof}[Proof of lemma \ref{lemma:usefulinequalities}]
The first inequality of the lemma relies on the fact that
\begin{align*}
\expect{e^{-\lambda T^\veps}} &\leq \sum_{k\geq 0} e^{-\lambda k} \PP(T^\veps > k)  \leq \sum_{k\geq 0} \left[e^{-\lambda} \PP(T^\veps > 1) \right]^k \\ 
&= \frac{1}{1- e^{-\lambda}\PP(T^\veps > 1)} \sim 1- e^{-\lambda}\PP(T^\veps > 1) \quad \text{as } \veps \to 0 \,,
\end{align*}
since $\PP(T^\veps >1) \xrightarrow{\veps \to 0} 0$. Notice an analogous inequality holds for $S^\veps$. By Markov's inequality, we know that
\be
\PP(T^\veps > 1) \leq \expect{(T^\veps)^k} \leq C_k \expect{T^\veps}^k
\ee
from which the first inequality follows by lemma \ref{lemma:Uvepsto0}. The second and third inequalities follow from noticing that for any $x$ and $y$
\be
\abs{\abs{x}-\abs{y}} \leq \abs{x-y} 
\ee
and applying Jensen's inequality.
\end{proof}

\begin{proof}[Proof of lemma \ref{lemma:weakcvg}]
Consider a test function $\vp \in C_c^\infty([\delta,\infty[)$, then integrating by parts
\begin{align}
\frac{1}{2\pi i}\int_\R \left[\int_{\gamma-iT}^{\gamma+iT} e^{\lambda t} \lambda^k \; d\lambda\right] \vp(t) dt &= \frac{(-1)^k}{2\pi i} \int_\R dt \; \vp^{(k)}(t) \int_{\gamma-iT}^{\gamma+iT}  e^{\lambda t} \; d\lambda  \nonum \\
&= (-1)^k \int_\R   \frac{e^{\gamma t}\vp^{(k)}(t)}{\pi t}  \sin(Tt) \; dt
\label{eq:ippintegral}
\end{align}
By performing the change of variables $y=Tt$, we see that the integral is weakly approaching $0$, as $\vp$ is not supported at $0$. Additionally, away from $0$, the function 
\be
\frac{e^{\gamma t} \vp^{(k)}(t)}{\pi t}
\ee 
is a compactly supported $C^\infty$-function, integrating by parts $n$ subsequent times equation \ref{eq:ippintegral} yields bounds of this integral by $C_\vp T^{-n}$, where $C_\vp$ is a constant which depends on the test function and its support.

Let us now show that 
\be
\frac{1}{2\pi i} \int_{\gamma - iT}^{\gamma+iT} \frac{e^{\lambda t}}{\lambda^k} \; d \lambda \xrightarrow[T \to \infty]{\mathcal{D}} \frac{t^{k-1}}{(k-1)!} \,.
\ee
Once again integrating by parts,
\begin{align*}
\int_\R dt \;\vp(t) \left[\frac{1}{2\pi i} \int_{\gamma-iT}^{\gamma+iT} \frac{e^{\lambda t}}{\lambda^k} \; d\lambda - \frac{t^{k-1}}{(k-1)!}\right] &= (-1)^n \int_\R dt \;\vp^{(n)}(t) \left[\frac{1}{2\pi i}\int_{\gamma-iT}^{\gamma+iT} \frac{e^{\lambda t}}{\lambda^{k+n}} \; d\lambda - \frac{t^{n+k-1}}{(n+k-1)!}\right] \,. 
\end{align*}
Applying the residue theorem to evaluate the complex integral we get, for $T> \gamma$,
\be
\frac{1}{2\pi i}\int_{\gamma-iT}^{\gamma+iT} \frac{e^{\lambda t}}{\lambda^{k+n}} \; d\lambda = \frac{t^{n+k-1}}{(n+k-1)!} + \frac{1}{2\pi i}\int_{C_T} \frac{e^{\lambda t}}{\lambda^{k+n}} \; d\lambda \;,
\ee
where $C_T$ is the circle of center $\lambda = \gamma$ and radius $T$. By the estimation lemma, the contribution of this integral is bounded by $e^{\gamma t} T^{-(n+k-1)}$\,. It follows that
\be
\abs{\int_\R dt \;\vp(t) \left[\frac{1}{2\pi i}\int_{\gamma-iT}^{\gamma+iT} \frac{e^{\lambda t}}{\lambda^k} \; d\lambda - \frac{t^{k-1}}{(k-1)!}\right]} \leq T^{-(n+k-1)}\norm{e^{\gamma t}\vp^{(n)}(t)}_{L^1} \,,
\ee
thereby giving the speed of convergence desired.
\end{proof}

\begin{proof}[Proof of theorem \ref{thm:meroextensionLevy}]
Throughout this proof, we shall denote 
\be
F(\lambda, \veps):= \frac{1}{\lambda} \, \frac{\expect{e^{-\lambda U^\veps}}}{1- \expect{e^{-\lambda U^\veps}}} \,.
\ee
The assumption of non-monotonicity of the L\'evy process ensures that, almost surely, $S^\veps$ and $T^\veps$ both tend to $0$ as $\veps \to 0$. Consider now the times $T_i^\veps$ and $S_i^\veps$ given in definition \ref{def:vepsminmax}. Since $X$ is L\'evy, $T_{i+1}^\veps - S_i^\veps$ and $S_i^\veps - T_i^\veps$ are independent from one another, and are both equal in distribution to $T^\veps$ and $S^\veps$ respectively.
\par
By lemma \ref{lemma:SvepsTvepsMoments}, $S^\veps$ and $T^\veps$ admit finite moments for all $k$ and the function $\expect{e^{-\lambda U^\veps}}$ is well defined, so, by equation \ref{eq:by339}
\begin{align}
\Lag(\expect{N^\veps_t})(\lambda) &= \Lag(\PP(R_t \geq \veps))(\lambda) + \frac{1}{\lambda} \sum_{k \geq 1} \expect{e^{-\lambda U^\veps}}^k \nonum \\
&= \Lag(\PP(R_t \geq \veps))(\lambda)  + \frac{1}{\lambda} \frac{1}{\expect{e^{-\lambda U^\veps}}^{-1}-1} \label{eq:RtplusRest}
\end{align}
If we denote $\rho_\veps$ the radius of convergence of the Taylor series at zero associated to $\expect{e^{-\lambda U^\veps}}$, for $\abs{\lambda}<\rho_\veps$, 
\be
\expect{e^{-\lambda U^\veps}} = \sum_{k=0}^\infty \frac{(-\lambda)^k \expect{(U^\veps)^k}}{k!} \,.
\ee
This radius of convergence $\rho_\veps$ can be bounded below with the results of lemma \ref{lemma:SvepsTvepsMoments} by 
\be
-\log(\PP(S^\veps >1) \vee \PP(T^\veps >1))<\rho_\veps \,,
\ee
which entails that $\rho_\veps \to \infty$ as $\veps \to 0$. We deduce from this series the Laurent series associated to $\lambda^{-1}(\expect{e^{-\lambda U^\veps}}^{-1} -1)^{-1}$, namely
\begin{align*}
F(\lambda,\veps) = \frac{1}{\lambda^2 \expect{U^\veps}}& + \frac{1}{\lambda}\left[\frac{\expect{(U^\veps)^2}}{2\expect{U^\veps}^2}-1\right] + \frac{3 \expect{(U^\veps)^2}^2 - 2\expect{U^\veps}\expect{(U^\veps)^3}}{12 \expect{U^\veps}^3}+O(\lambda) \,.
\end{align*}
where the remainder in $\lambda$ is an analytic function of $\lambda$ for $\abs{\lambda} < \rho_\veps$. By the inequalities of lemma \ref{lemma:usefulinequalities}, the function doesn't admit any poles on the half plane $\Real(\lambda)>0$, so that the Taylor series above converges over the same disk as that of $\expect{e^{-\lambda U^\veps}}$.
\par
Observe now that for some small $\gamma>0$, the inverse Laplace transform of $F(\lambda,\veps)$ can be written as 
\begin{equation}
\label{eq:InverseLaplaceofF}
\Lag^{-1}[F](t,\veps) = \frac{1}{2\pi i} \left\{\int_{\gamma - i\rho_\veps}^{\gamma+i \rho_\veps} + \int_{\gamma+i \rho_\veps}^{\gamma+i\infty} +\int_{\gamma-i\infty}^{\gamma-i \rho_\veps}\right\} e^{\lambda t} F(\lambda,\veps) \; d\lambda \,.
\end{equation}
Weakly, the integrals going off to infinity are of order $o(\rho_\veps^{-n})$ for any $n \in \N$, since for any test function $\vp \in C_c^\infty([0, \infty[)$, integrating by parts
\be
\int_\R dt \; \vp(t) \left[\frac{1}{2\pi i} \int_{\gamma +i\rho_\veps}^{\gamma + i\infty}  e^{\lambda t}F(\lambda,\veps) \; d\lambda\right] = \frac{(-1)^n}{2\pi i} \int_\R  dt \; \vp^{(n)}(t) \int_{\gamma +i\rho_\veps}^{\gamma + i\infty} d\lambda \; \frac{e^{\lambda t}}{\lambda^n} F(\lambda,\veps) \; d\lambda
\ee
But using lemma \ref{lemma:usefulinequalities},
\be
\abs{\int_{\gamma +i\rho_\veps}^{\gamma + i\infty} d\lambda \; \frac{e^{\lambda t}}{\lambda^n} F(\lambda,\veps) \; d\lambda} \leq e^{\gamma t} \int_{\gamma+i\rho_\veps}^{\gamma+i\infty} \abs{\frac{F(\lambda,\veps)}{\lambda^n}} \;d\lambda = O(\rho_\veps^{-n-2}) \,,
\ee
which entails that the integrals going to infinity in equation \ref{eq:InverseLaplaceofF} converge weakly to $0$ at a rate $o(\rho_\veps^{-n})$ for any $n \in \N$. Thus, asymptotically as $\veps \to 0$, for $t>0$, 
\be
\expect{N^\veps_t} = \frac{t}{\expect{U^\veps}} + \left(\frac{\expect{(U^\veps)^2}}{2\expect{U^\veps}^2}-1\right) + \PP(R_t \geq \veps) + o(\rho_\veps^{-n}) \,,
\ee
for any $n \in \N$. If the process is $\al$-stable, then $(X_{c^\al t})_{t\geq 0} = (c X_t)_{t\geq 0}$ in distribution for all $c$, so that $U^\veps = \veps^{\al} U$ in distribution and 
\be
\expect{N^\veps_t} = \frac{t}{\expect{U}\veps^\al} + \frac{\expect{U^2}}{2\expect{U}^2} + o(\veps^{\al n}) \quad \text{as } \veps \to 0 \,,
\ee
for all $n \in \N$. As $\veps \to 0$, $1-\PP(R_t \geq \veps) = o(\veps^n)$ for any $n$, since
\be
\PP(R_t \leq \veps) \leq \PP(T^\veps > t) \leq \frac{\veps^{\al k}\expect{(T^1)^k}}{t^k} 
\ee 
for any $k$ by Markov's inequality.
\end{proof}
\begin{remark}
\label{rmk:VarNveps}
A similar theorem can be proven in $L^s(\Omega)$ for $\al$-stable processes. For instance, if $X$ is $\al$-stable and $s=2$ one obtains that for every $n \in \N$,
\begin{align*}
\Var(N^\veps_t) &\sim \left[\frac{\Var(U)-2\expect{U}^2}{\expect{U}^3}\right]\frac{t}{\veps^\al}  \\
&+\frac{5 \Var(U)^2}{4 \expect{U}^4}+\frac{\Var(U)}{\expect{U}^2} -\frac{2 \expect{U^3}}{3 \expect{U}^3}+\frac{7}{4} + o(\veps^{\al n})\quad \text{as } \veps \to 0 \,.
\end{align*}
\end{remark}
Interestingly, there is a constant term appearing in this expansion, which can be understood as induced by the boundary. This interpretation comes from Picard's analysis of the problem \cite{Picard:Trees}, in which the first term of this asymptotic series was also derived (\cf proposition \ref{prop:PicardProp}). 

If $X$ has almost surely discontinuous paths, $X_t$ exhibits macroscopic jumps. These will turn out to bring significative contributions, so much so that
\begin{corollary}
If $\al \neq 2$, the $\zeta$-function of any $\al$-stable L\'evy process is ill-defined for any $p \in \C$. 
\end{corollary}
\begin{proof}
The $\zeta$-function of a stochastic process $X$ can be written as
\begin{equation}
\label{eq:ZetaDecomposition}
\zeta_X(p) = \expect{R^p_t} + \expect{\sum_{k \geq 2} \ell_k^p(X_t)} \,,
\end{equation}
if $X$ is $\al$-stable, the first term can be written as 
\be
\expect{R_t^p} = t^{\frac{p}{\al}} \expect{R_1^p} \geq  t^{\frac{p}{\al}} \expect{\abs{X_1}^p} \,,
\ee
where we have momentarily taken $p \in \R$. Since $X_1$ has a L\'evy $\al$-stable distribution, taking $p$ now complex, $\expect{R_t^p}$ is infinite as soon as $\Real(p)\geq \al$, since $X_1$ does not admit any moments of order (of real part) larger than $\al$. Applying theorem \ref{thm:meroextensionLevy}, we know that the second term in the above decomposition of $\zeta_X$ is only defined for $\Real(p)>\al$, so the fundamental strip of $\Mel\expect{N^\veps_t}$ is empty.
\end{proof}
In fact, it is possible to show that $\PP(X_1 > \veps) \sim \PP(R_1 >\veps)$ as $\veps \to \infty$. It turns out that the distribution of $R_1$ is dominated by the probability of having one large jump, which confirms our previous statement on the effect of the discontinuity of L\'evy processes on the distribution of $R$. This is the so-called single big jump principle.
\begin{proposition}[Single big jump principle, Bertoin, \cite{Bertoin:Levy}]
\label{prop:SingleBigJump}
If $X$ is an $\al$-stable process ($\al <2$), there exists a constant $k$ such that
\be
\PP(R_1 \geq \veps) \sim \frac{k}{\veps^{\al}}  \quad \text{as } \veps \to \infty\,.
\ee
\end{proposition}
Loosely speaking, it is intuitive to think that a corrective asymptotic power series for $\PP(R_t \geq \veps)$ of the form
\be
\PP(R_1 \geq \veps) \sim \sum_{k \geq 1} a_k \veps^{-k\al} \quad \text{ as } \veps \to \infty 
\ee 
should exist for the following reason. By the single big jump principle, the probability that the range exceeds $\veps$ for large $\veps$ is dominated by the probability of a single big jump. However, it is also possible to have $n$ large jumps of size $J_k \veps$ where $\sum_k^n J_k \geq 1$. The probability of each of these jumps happening is of order $O(\veps^{-\al})$ and by independence, the probability that $k$ jumps of size $O(\veps)$ happen is $O(\veps^{-\al k})$. In general, we cannot expect these events to be disjoint from one another, so the coefficients $a_k$ of this sum may be negative. Finally, by the scaling invariance it is sufficient to show that this is so for $R_1$. Corrective terms to the above asymptotic relation should thus in principle exist, but the explicit calculation of these terms is out of the scope of this paper.
\par
By contrast, we will now show that $\expect{N^\veps_t}-\PP(R_t \geq \veps)$ is well-behaved. This motivates the following definition
\begin{definition}
\label{def:tailzetaLevy}
The \textbf{tail $\zeta$-function} of the stochastic process $X$ on $[0,t]$ is defined as
\be
\hat\zeta_X(p) := \expect{\Pers_p^p(X)- R^p_t} \,.
\ee
\end{definition} 
\begin{theorem}
\label{thm:tailzetaLevy}
The tail $\zeta$-function associated to an $\al$-stable L\'evy process is given by
\begin{equation}
\label{eq:zetahatLevy}
\hat\zeta_X(p)  = \frac{t^{\frac{p}{\al}}}{\Gamma(\frac{p}{\al})} B^*\!\left(\frac{p}{\al}\right) \,,
\end{equation}
which extends to a meromorphic function of $p$ to the entire complex plane (since $B^*$ is itself meromorphic), with a unique simple pole at $p=\al$ of residue $\expect{U}^{-1}\al t$.
\end{theorem}
\begin{proof}[Proof of theorem \ref{thm:tailzetaLevy}]
To show that this quantity is well-defined, let us start by noticing that
\be
\Lag(\expect{N^\veps_t}-\PP(R_t \geq \veps))(\lambda) = \frac{\expect{e^{-\lambda \veps^\al U}}}{\lambda(1- \expect{e^{-\lambda \veps^\al U}})}
\ee
which for $\Real(\lambda) >0$ goes to zero (uniformly in $\lambda$) exponentially fast as $\veps \to \infty$, showing that $\expect{N^\veps_t}-\PP(R_t \geq \veps)$ does as well for $\veps \to \infty$ by an application of Markov's inequality. We can also compute the contribution of the second term of equation \ref{eq:ZetaDecomposition}. First, notice that
\be
\Lag\!\left(\expect{\sum_{k \geq 2} \ell_k^p(X_t)} \right)\!(\lambda) = \frac{p}{\lambda} \Mel\!\left[\frac{\expect{e^{-\lambda \veps^\al U}}}{1- \expect{e^{-\lambda \veps^\al U}}} \right](p)
\ee
Using the scaling property of the Mellin transform $\Mel_z[f(\lambda z)](p) = \lambda^{-p} f^*(p)$ and inverting the Laplace transform
\begin{equation}
\label{eq:MelofLevy}
\expect{\Pers_p^p(X) - R^p_t} = \frac{p t^{\frac{p}{\al}}}{\Gamma(1+\frac{p}{\al})} \Mel\!\left[\frac{\expect{e^{-\veps^\al U}}}{1- \expect{e^{-\veps^\al U}}}\right](p) \,.
\end{equation}
Finally, setting 
\be
B(z) := \frac{\expect{e^{-z U}}}{1- \expect{e^{-z U}}} \quad \text{and} \quad B^*(p) := \Mel_z[B(z)](p) \,,
\ee
the polynomial scaling property of the Mellin transform, $\Mel_z[f(z^\al)](p) = \frac{1}{\al} f^*(\frac{p}{\al})$ yields the final result.
\end{proof}
\begin{remark}
Theorem \ref{thm:tailzetaLevy} can be used to give an alternative proof for the series expansion of theorem \ref{thm:meroextensionLevy}.
\end{remark}
\begin{proof}[Alternate proof of theorem \ref{thm:meroextensionLevy}]
By lemma \ref{lemma:SvepsTvepsMoments} and the analyticity of the expresion of $B$ with respect to $\expect{e^{-zU}}$, $B$ admits a Laurent series on some non-trivial annulus around zero with a single simple pole at $z=0$. By the fundamental correspondence (theorem \ref{thm:FundamentalCorrespondence}), the existence of this Laurent expansion guarantees that $B^*(\frac{p}{\al})$ admits a meromorphic continuation to the whole complex plane with only simple poles at every $p=-n\al$ for every $n \in \N$ and at $p=\al$. The poles at the negative integer multiples of $\al$ are compensated exactly by those of the $\Gamma$-function in the denominator of the expression of $\hat\zeta$, leaving only a pole at $\al$. Now, recalling that 
\be
\hat\zeta(p) = p \Mel[\expect{N^\veps_t} - \PP(R_t \geq \veps)](p) \,,
\ee 
$\Mel[\expect{N^\veps_t}- \PP(R_t \geq \veps)]$ has a supplementary pole at $p=0$. Admitting that $\hat\zeta(p)/p$ has the decay condition to apply the fundamental correspondence by inverting the Mellin transform we get the asymptotic relation desired.  
\end{proof}
\subsubsection{Exponential corrections}
\label{sec:exponentialcorrections}
The fundamental correspondence is limited in that it allows us only to describe $\expect{N^\veps}$ asymptotically up to terms smaller than any polynomial. However, in accordance to the discussion of section \ref{sec:analyticcontinuation}, a finer study of the analytic properties of $\hat\zeta$ can yield the superpolynomial corrections to our estimate, assuming that $B(z)$ admits a meromorphic extension to the whole complex plane. Using lemmata \ref{lemma:IntegralRepresentationoff} and \ref{lemma:functionalequation},
\begin{align}
&\hat\zeta_X(p) = t^{\frac{p}{\al}} \Gamma\!\left(1-\frac{p}{\al}\right) \sum_{z_0 \in \mathcal{P}} \Res((-z)^{\frac{p}{\al}-1}B(z); z_0) \\
&\Mel(\expect{N^\veps_t} - \PP(R_t \geq \veps))(p) = -\frac{t^{\frac{p}{\al}} \Gamma\!\left(-\frac{p}{\al}\right)}{\al} \sum_{z_0 \in \mathcal{P}} \Res((-z)^{\frac{p}{\al}-1}B(z); z_0)\,.
\label{eq:MellinNvepsLevy}
\end{align}
Recognizing that
\be
\Mel_z\left[\frac{e^{z_0/z}}{z_0}\right](p) = -\Gamma(-p) \left(\frac{-1}{z_0}\right)^{1-p} \,,
\ee
we may formally invert the Mellin transform if all the $z_0$'s are simple poles to obtain the exponentially small corrections
\be
\expect{N_t^\veps} - \PP(R_t \geq \veps) - \frac{t}{\expect{U}\veps^\al} - \left[\frac{\expect{U^2}}{2\expect{U}^2}-1\right] \sim \sum_{z_0 \in \mathcal{P}} \frac{e^{tz_0/\veps^\al}}{\al z_0} \Res(B(z);z_0) \quad \text{as }\veps \to 0 \,.
\ee
Generally, the poles are not simple so the corrective terms to this series stem from residues of higher order poles (the corrections remain nonetheless superpolynomially small as $\veps \to 0$).

\subsubsection{Distribution of the length of branches}
\label{sec:distlengthbars}
The distribution of the length of the $k$th branch (in the sense of figure \ref{fig:algorithm}) can be calculated. Recall that
\be
\expect{\ell_k^p(X)} = p \Mel[\PP(N_t^\veps \geq k)](p) \,.
\ee
For $k \geq 2$,
\be
\Lag[\expect{\ell_k^p(X)}](\lambda) = \frac{p}{\lambda} \Mel\left[\expect{e^{-\lambda U^\veps}}^{k-1}\right](p) \,.
\ee
If we suppose once again that $X$ is a L\'evy $\al$-stable process, this can be simplified to yield
\be
\expect{\ell_k^p(X)} = \frac{t^{\frac{p}{\al}}}{\Gamma(\frac{p}{\al})}  \,\Mel\!\left[\expect{e^{-\veps U}}^{k-1}\right]\!\left(\frac{p}{\al}\right) \,,
\ee
Taking the Mellin transform of a power is in general difficult. Inversion is also in general complicated due to the presence of the $\Gamma$-function in the denominator of the above expression. To remediate the first problem, we can form the generating function yielding the distribution for the $k$th bar,
\be
G_\al(z;p) := \sum_{k \geq 2} \expect{\ell_k^p(X)}z^k = \frac{t^{\frac{p}{\al}}}{\Gamma(\frac{p}{\al})} \,\Mel\!\left[\frac{z}{(z \expect{e^{-\veps U}})^{-1}-1}\right]\!\left(\frac{p}{\al}\right) \,,
\ee
which allows us to express
\begin{proposition}
For $k \geq 2$, the distribution of the length of the $k$th longest branch is characterized by its Mellin transform which is given by
\be
\expect{\ell_k^p(X)} = \frac{1}{k!}\left.\frac{\del^k}{\del z^k} \right\vert_{z=0} G_\al(z;p) \,.
\ee
\end{proposition}
Whenever convenient, the expression above can also be evaluated by considering a circular contour of small enough radius $r$ around the origin $C_r$ and evaluating
\be
\expect{\ell_k^p(X)} = \frac{1}{2\pi i} \oint_{C_r} \frac{G_\al(z;p)}{z^{k+1}}\; dz \,.
\ee

\subsubsection{Statistical parameter testing for $\al$-stable processes and perspectives}
\label{sec:StatisticalTest}
What we will aim to do in this section is to illustrate by example why barcodes can be a robust statistical tools for parameter testing. Parameter testing is a widely studied subject, notably for self-similar processes, where the problem has been treated in dimension $1$ (a non-comprehensive list of references is \cite{Dang:2015vx} and the references therein). A variety of different methods, such as multi-scale wavelet analysis, have been used to produce these results (although other methods such as the ones of \cite{Dang:2015vx} have also been used), so our approach does not offer anything new in this respect. The interest of our method lies in possible applications to higher dimensional random fields, for which wavelet analysis is not an effective tool. A complete theoretical framework for this would require the study of the trees of higher dimensional random fields, which are out of the scope of this paper : instead, this section acts as a proof of concept for the use of topological estimators and their utility, by studying what happens in dimension 1.
\par
In what follows, we will consider $X$ to be an $\al$-stable L\'evy process, of which we will aim to estimate the parameter $\al$. From proposition \ref{prop:PicardProp} we know that almost surely 
\be
N^\veps_t \sim C t \veps^{-\al} \quad \text{as } \veps \to 0 \,.
\ee
In particular, given some sample we may compute the sampled value of $N^\veps_t$, which we will denote $\hat{N}^\veps_t$ explicitly. A close inspection of the behaviour of the sample mean $\overline{N}^\veps_t$ should thus yield an estimation for the parameter $\al$ of the process $X$. 
\begin{remark}
In fact, the same reasoning allows us to estimate the Hurst parameter $H$ of a fractional Brownian motion (fBM), which also exhibits self-similarity. In this case, the analogue of the asymptotic result of proposition \ref{prop:PicardProp} is \cite[\S 3]{Picard:Trees}
\be
\text{a.s.} \quad  N^\veps \sim Ct \veps^{-\frac{1}{H}} \quad \text{as } \veps \to 0 \,.
\ee
\end{remark}
More precisely, given a sample, our test consists in performing the following steps. 
\begin{enumerate}
  \item Sample $M$ paths of the stochastic process $X$ (for example at regular intervals of size $\frac{1}{N}$ for some $N$) ;  
  \item Compute the barcode of the sampled paths. To do this, first construct a filtered simplicial complex (which is in this case nothing other than a chain with $\!\sim \! N$ links) by taking each point to be a vertex of a complex and joining adjacent sampling points with an edge. The filtration on this complex is the value of the process at the edge (for an edge connecting vertex $a$ to vertex $b$, the value of the filtration is $X_a \wedge X_b$). Finally, the persistent homology of this complex can be computed with the \texttt{gudhi} package \cite{Gudhi}, which incidentally also offers a convenient implementation of filtered simplicial complexes due to Boissonnat and Maria \cite{Boissonnat_2014}. 
  \item For \textit{some range of small enough} $\veps$, and for some positive constant $c>1$ compute the quantity 
  \be
  \hat\al_M := \log_c\left[\frac{\overline{N}^{\veps/c}_t-\overline{N}^{2\veps/c}_t}{\overline{N}^\veps_t-\overline{N}^{2\veps}_t}\right] \,.
  \ee
  Here, the notion of \textit{some range of small enough} $\veps$ and the constant $c$ both depend on $N$, with the limiting condition that as $N \to \infty$, the lower bound on the range of valid $\veps$ goes to zero. 
\end{enumerate}
Our claim is that the computed quantity $\hat\al$ is a valid estimation of the parameter $\al$ (for fBM, the quantity obtained in this way is an estimate of $\frac{1}{H}$). 

\begin{lemma}[Convergence of the sample means]
\label{lemma:samplemeantoExpected}
The quotient 
\be
\frac{\overline{N}^{\veps/c}_t-\overline{N}^{2\veps/c}_t}{\overline{N}^\veps_t-\overline{N}^{2\veps}_t} \xrightarrow[M \to \infty]{\PP} \frac{\expect{N^{\veps/c}_t-N^{2\veps/c}_t}}{\expect{N^\veps_t-N^{2\veps}_t}}
\ee
at a rate $C_s M^{-s}$, for every $1\leq s \leq 2$ where $C_s$ is a constant depending on $s$ and the $s$th moment of $N^{\veps/c}$. In particular, 
\be
\hat\al_M \xrightarrow[M \to \infty]{\PP} \al+\xi(\veps)
\ee
at the same rate, where $\xi(\veps)$ is a superpolynomially small function of $\veps$.
\end{lemma}
\begin{remark}
The at first seemingly convoluted expression for the estimator $\hat\al_M$ can be explained due to the results of theorem \ref{thm:meroextensionLevy}. The substraction present in the numerator and denominator is performed so that the constant terms of equation \ref{eq:NvepstAlphaStable} vanish. Ignoring the superpolynomial contributions to this expression which remain small, we then have that the argument inside the $\log$ of the estimator is roughly
\be
c^{\hat\al_M} \approx \frac{\frac{t}{\expect{U}(\veps/c)^\al}-\frac{t}{\expect{U}(2\veps/c)^\al}}{\frac{t}{\expect{U}\veps^\al}- \frac{t}{\expect{U}(2\veps)^\al}} \approx c^\al \,.
\ee
With this in mind, let us now formally prove the statement of lemma \ref{lemma:samplemeantoExpected}.
\end{remark}
\begin{proof}
That the numerator and the denominator tend to the respective expected values holds by a simple application of the weak law of large numbers, since $N^\veps_t$ is a random variable in $L^s$ for $s\geq 1$. The rate of convergence of this limit can be obtained via a simple application of Markov's inequality, by noting first that the summands in the denominator tend to their limits faster than those of the numerator, as the latter's $s$th moments are always larger than the former's. From theorem \ref{thm:meroextensionLevy}, we see that the limit can be expressed as
\be
\frac{\expect{N^{\veps/c}_t-N^{2\veps/c}_t}}{\expect{N^\veps_t-N^{2\veps}_t}} = c^\al \;\frac{1+ g(c\veps)}{1 + g(\veps)} \,,
\ee
where $g$ is a function tending to $0$ superpolynomially fast as $\veps \to 0$, determined by the superpolynomial corrections to the results of theorem \ref{thm:meroextensionLevy}. The statement of the lemma ensues.
\end{proof}

\begin{lemma}[Probable $L^\infty$-distance of sampling]
\label{lemma:LinftySampling}
Denote $\hat{X}$ the trajectory samples of the $\al$-stable process $X$ at every interval of length $\frac{1}{N}$. More precisely, somewhat abusing the notation we can write,
\be
\hat{X}_t  = \sum_{n=0}^{N-1} 1_{[\frac{n}{N},\frac{n+1}{N}[}(t)\; X_{\frac{n}{N}} \,.
\ee 
There exists a constant $k$ such that 
\be
\PP\left(\sup_{t \in \R} \abs{X_t - \hat{X}_t} \leq \veps \right) \gtrsim 1- \left(\frac{k}{N\veps^{\al}}\right)^N \quad  \text{ as } \veps N^{1/\al} \to \infty\,.
\ee
\end{lemma}
\begin{remark}
The asymptotic dependence above fixes admissible values of $\veps$ as a function of $N$ as holding whenever the asymptotic dependence above is valid (it must be valid between $\veps/c$ and $2\veps$). Furthermore, the parameter $c$ we chose above is also further constrained by the requirement that the asymptotic relation of theorem \ref{thm:meroextensionLevy} holds between $\veps/c$ and $2\veps$. More precisely, we fix $c$ and $\veps$ such that the superpolynomial contributions in the expansion of theorem \ref{thm:meroextensionLevy} are negligible with respect to the term in $\veps^{-\al}$ \textit{and} by imposing that $\veps N^{1/\al}$ is large enough so that the asymptotic relation of lemma \ref{lemma:LinftySampling} simultaneously holds within the range $[\frac{\veps}{c},2\veps]$. In practice, we may fix $c$ and $\veps$ by looking at a log-log chart of $N^\veps$, the regime of validity of $\veps$ and the value of $c$ become clear, as shown in figure \ref{fig:HistLevySim}.  
\end{remark}
\begin{remark}
With respect to the barcode, linear interpolation between values of $X$ or the consideration of the process $\hat{X}$ is equivalent.
\end{remark}
\begin{remark}
It is not a priori obvious that the event above is measurable. However, continuity in probability and the a.s. existence of a càdlàg modification of the process allows us to interpret this event to be a supremum over every $t \in \Q$, rendering the event measurable. 
\end{remark}
\begin{proof}
It suffices to show the result over the interval $[0,1]$. Noticing that the sampling coincides with the value of the path at every $t= \frac{1}{N}$, it suffices to evaluate the probability that over $N$ intervals of length $\frac{1}{N}$ the real sampled path $X_t(\omega)$ (notice the absence of a hat) strays away from the sampled path $\hat{X}_t(\omega)$.  
Focusing on a single interval $[0,\frac{1}{N}]$, the single big jump principle \ref{prop:SingleBigJump} states that there exists a constant $k$ such that this probability of straying away in this interval is
\be
\PP(R_{\frac{1}{N}} \geq \veps) \sim \frac{k\veps^{-\al}}{N} \quad  \text{ as } \veps N^{1/\al} \to \infty\,.
\ee
By independence, over $N$ such intervals
\be
\PP\left(\sup_{t \in \R} \abs{X_t - \hat{X}_t} \leq \veps \right) \gtrsim 1- \left(\frac{k}{N\veps^{\al}}\right)^N \sim 1  \quad  \text{ as } \veps N^{1/\al} \to \infty\,,
\ee
as desired.
\end{proof}
\begin{figure}[h!]
\centering
\includegraphics[width=0.6\textwidth]{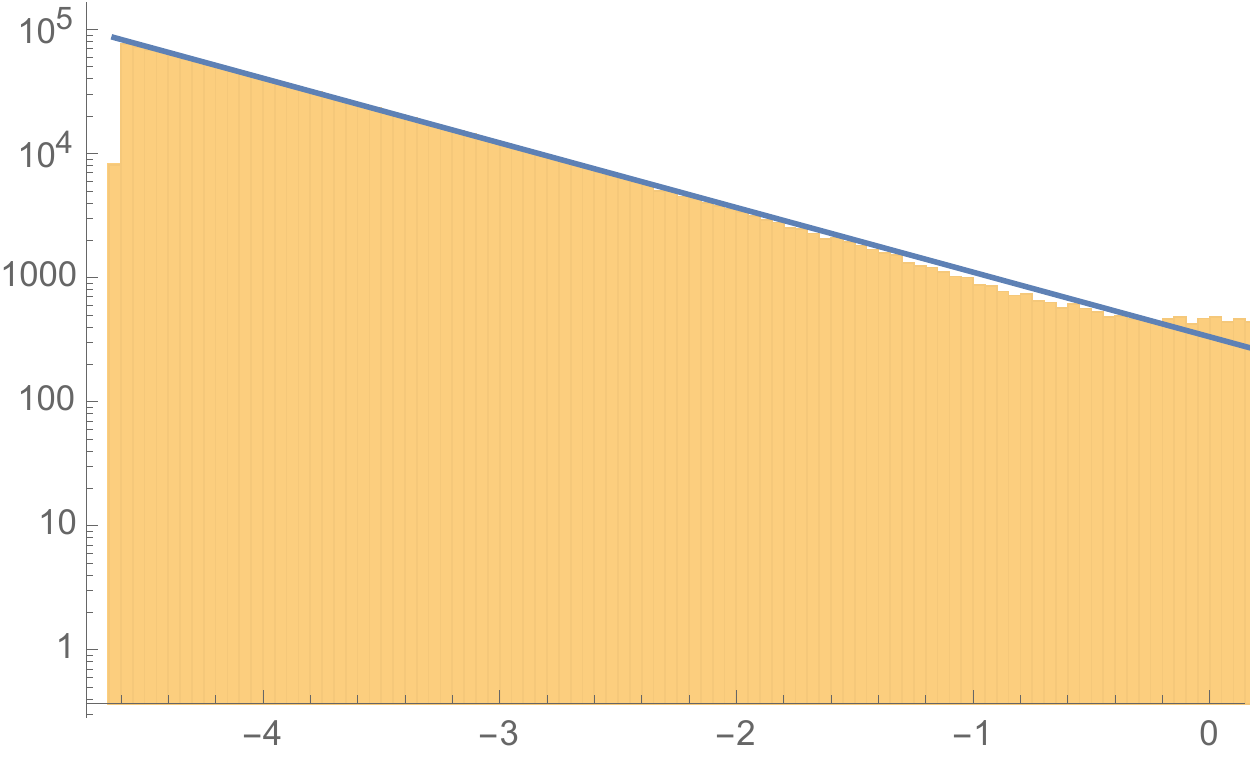}
\caption{In orange, a histogram of the number of bars of length $\geq \veps$, $N^\veps$ found as a function of $\log \veps$ from a simulation of a Lévy $1.2$-stable process as a random walk. In blue, the function $C_{1.2} \veps^{-1.2}$.}
\label{fig:HistLevySim}
\end{figure}
Now, let us recall the following theorem.
\begin{theorem}[P, Theorem 3.1 \cite{Perez_Pr_2020}]
\label{thm:btsnbtmatching}
Let $\delta_N := \norm{X-\hat{X}}_{\Linfty}$, then there exists a $\delta_N$-matching between the barcodes of $\hat{X}$ and $X$. In particular, for any $\veps \geq 2\delta_N$
\be
N_X^{\veps + \delta_N} \leq N_{\hat{X}}^\veps \leq N_X^{\veps-\delta_N}
\ee
Moreover, if $\expect{N_X^\veps}$ is continuous with respect to $\veps$, then 
\be
N_{\hat{X}}^\veps \xrightarrow[N\to \infty]{L^1} N_X^\veps \quad \text{and} \quad N_{\hat{X}}^\veps \xrightarrow[n\to \infty]{\PP} N_X^\veps \,, \nonum
\ee
which at fixed $N$ quantitatively translates to
\be
\expect{\abs{N_X^\veps- N_{\hat X}^\veps}} \leq 2\omega_\veps(\delta_N) \quad \text{and} \quad \PP(\abs{N_X^\veps- N_{\hat X}^\veps} \geq k) \leq \frac{2\omega_{\veps}(\delta_N)}{k}
\ee
where $\omega_\veps$ is the modulus of continuity of $\expect{N_X^\veps}$ on the interval $[\veps -\delta_N,\veps+\delta_N]$. Finally, the following inequalities also hold
\begin{align*}
N^{\delta_n}_{\hat{X}} \geq  N^{2\delta_N}_X \quad \text{ and } \quad N^{\delta_N}_X \geq N_{\hat{X}}^{2 \delta_N} \;.
\end{align*}
\end{theorem}
This statement can be specialized given our two lemmas above. The theorem provides bounds on $N^{\veps}_{\hat X}$ provided that we know the value of $\delta_N$, since if $\veps$ is small enough, $N^\veps_X$ has some almost sure asymptotic behaviour. On the other hand, by virtue of lemma \ref{lemma:LinftySampling} we have a \textit{probable} estimate of $\delta_N$, \textit{i.e.} with probability $q$, we may give a bound of $\delta_N$, rendering the statement quantitative. The second part of the statement of theorem \ref{thm:btsnbtmatching} provides bounds on the $L^1$ distance between $N^\veps_{\hat{X}}$ and $N^\veps_X$, provided that we know that $\expect{N^\veps_X}$ is continuous. This happens to be the case for Brownian motion, as shown in \cite{Perez_Pr_2020}. Showing it in full generality for L\'evy processes requires a closer study of the range of L\'evy processes and the continuity of the inverse Mellin transform of $\hat\zeta_X(p)/p$. However, for the purposes of the construction of our statistical test, lemma \ref{lemma:samplemeantoExpected} suffices, as it provides us with a quantitative guarantee that the parameter $\al$ is well estimated by our estimator $\hat{\al}_M$.

\subsection{Propagators and local $\zeta$-functions}
\label{sec:propagatorslocaltrees}

In dimension one, it is possible to use the total order of $\R$ and count $N^{x,x+\veps}$ by counting the number of times we go up from $x$ to $x+\veps$. This idea can be formalized by the following sequence of stopping times already introduced in the literature of classical probability theory (\cf for instance \cite{Revuz_1999}).
\begin{definition}
\label{def:vepsminmax}
Setting $S_0^{x,\veps} = T_0^{x,\veps} = 0$, we define a sequence of times recursively
\begin{align}
T_{i+1}^{x,\veps} &:= \inf\left\{\, t \geq S_i^{x,\veps} \; \Bigg\vert \; f(t) \leq x \wedge (x+\veps) \right\} \nonum\\
S_{i+1}^{x,\veps} &:= \inf\left\{\, t \geq T_{i+1}^{x,\veps}\; \Bigg\vert \; f(t) \geq x \vee (x+\veps) \right\} \,. \label{eq:TiSixveps}
\end{align}

\end{definition}
Counting the number of bars of length $\veps$ is thus exactly to count the number of up and downs we make. More precisely,
\begin{equation}
\label{eq:Nvepscountingformula}
N^{x,x+\veps} = \inf\{i \, \vert \, T_i^{x,\veps} \text{ or } S_i^{x,\veps} = \inf \emptyset\} \,,
\end{equation}
by which we mean that it is the smallest $i$ such that the set over which $T_i^\veps$ or $S_i^\veps$ are defined as infima is empty.

With the aforementioned, it is possible to define a Feynman-like formalism to perform the computation of $\expect{N^{x,x+\veps}}$.
\begin{definition}
Let $X$ be a (strong Markov) stochastic process and let 
\be
T^a := \inf \{t \geq 0 \, \mid \, X_t > a \} \,
\ee 
be the hitting time of $a$ by $X$. We define the \textbf{propagator from $x$ to $y$} by
\be
\bra x | y \ket := \mathbb{E}_x\!\left[e^{-\lambda T^y}\right] \,
\ee
whenever this exists.
\end{definition}
\begin{remark}
Whenever convenient, we may take $\lambda = i \omega$, and modify the subsequent expressions appropriately. 
\end{remark}
If the process $X$ has the strong Markov property and the increments between the stopping times $T_i^{x,\veps}$ and $S_i^{x,\veps}$ are identically distributed, we can once again apply renewal theory. In particular, the Laplace transform of the occupation numbers $\Lag N^{x,x+\veps}_t$ can be understood in terms of the propagators above. For $x>0$,
\begin{align*}
\Lag(\expect{N^{x,x+\veps}_t})(\lambda) &= \frac{\bra 0 | x+\veps \ket }{\lambda} \sum_{k \geq 0} (\bra x+\veps |x \ket \bra x | x+\veps \ket)^k \\
&= \frac{\bra 0 | x +\veps \ket}{\lambda(1-\bra x+\veps |x \ket \bra x | x+\veps \ket)}
\end{align*}
Similarly, all moments of the distribution can in principle be calculated
\begin{align*}
\Lag(\expect{(N^{x,x+\veps}_t)^{s-1}})(\lambda) = \frac{\bra 0 | x+\veps \ket (1- \bra x+\veps |x \ket \bra x | x+\veps \ket)}{\lambda} \Li_{-s+1}(\bra x+\veps |x \ket \bra x | x+\veps \ket)
\end{align*}
where $\Li$ denotes the polylogarithm.
\begin{remark}
\label{rmk:LevyProc}
If $X_t - X_s = X_{t-s}$ in distribution, we can rewrite the above as,
\be
\Lag(\expect{(N^{x,x+\veps})^{s-1}})(\lambda) = \frac{\bra 0 | x+\veps \ket (1- \bra \veps | 0 \ket \bra 0 | \veps \ket)}{\lambda} \Li_{-s+1}(\bra \veps | 0 \ket \bra 0 | \veps \ket)
\ee 
\end{remark}
If $\bra x+\veps |x \ket \bra x | x+\veps \ket$ and $\bra 0 | x+\veps \ket$ admit an asymptotic expansion for small $\veps$ (alternatively, we can suppose that this function is a smooth enough function of $\veps$), then the Mellin transform of the expression above admits a meromorphic continuation. Furthermore, 
\be
\bra x+\veps |x \ket \bra x | x+\veps \ket = 1 + o(1) \quad \text{as }\veps \to 0
\ee
so that the order of the divergence of $\Lag(\expect{N^{x,x+\veps}_t})$ is dictated exclusively by the order of the first correction in $\veps$ to the product above. In reality, 
\begin{corollary}
\label{prop:MeroExistsLocal}
Suppose that the increments in the sequence of stopping times defined by $(T^{x,\veps}_i,S_i^{x,\veps})$ are i.i.d. and that $\bra 0 | x+\veps \ket$ and $\bra x+\veps |x \ket \bra x | x+\veps \ket$ have asymptotic expansions of the form of that of the fundamental correspondence (\cf theorem \ref{thm:FundamentalCorrespondence}), then, the function $\Mel\Lag(\expect{N^{x,x+\veps}_t})(p,\lambda)$ has a meromorphic extension on the half plane $\Real(p) > -k$ for any $k \in \N$. 
\end{corollary}
\begin{proof}
It's a simple application of the fundamental correspondence.
\end{proof}

\subsubsection{It\^o diffusions}
\label{sec:ItoDiffusions}
Computing the local $\zeta$-functions is sometimes easier than computing the $\zeta$-function associated with the process $X$. This is because the computation of the distribution of $U^\veps$ might not always be straightforward.
An example of this is the case of It\^o diffusions, \ie solutions to stochastic differential equations of the form
\be
dX_t = \mu(X_t)\,dt + \sigma(X_t)\, dB_t \,,
\ee
for some smooth functions $\mu$ and $\sigma$. These processes have the strong Markov property and the sequence of stopping times $(T_i^\veps,S_i^\veps)$ defined above are identically distributed as $\mu$ and $\sigma$ do not depend explicitly on time. Furthermore, these processes have infinitesimal generator
\be
\mathcal{G} = \mu(x) \frac{\del}{\del x} + \frac{\sigma^2(x)}{2} \frac{\del^2}{\del x^2} \,.
\ee 
We can use the theory of diffusion developped by It\^o and McKean \cite{Ito_1996} to find explicit expressions for the propagators $\bra x | y\ket$. The propagator is exactly the fundamental solution associated with the equation 
\begin{align*}
(\mathcal{G}-\lambda) \rho(x,\lambda) =  0 \text{ subject to } \begin{cases}\rho(y,\lambda) =1 \quad  \rho(x,\lambda) \xrightarrow{x \to -\infty} 0 & \text{for } x<y  \\ \rho(y,\lambda) =1 \quad \rho(x,\lambda) \xrightarrow{x \to +\infty} 0 & \text{for } x>y \end{cases}
\end{align*}
The solution to the above boundary value problem has been shown to be of the form \cite[p.130]{Ito_1996}
\be
\rho(x,\lambda) = \frac{\Psi_\lambda(x)}{\Psi_\lambda(y)} = \bra x | y \ket \,, 
\ee
where if $x<y$ (resp. $x>y$) $\Psi_\lambda(z)$ is (up to some constant) the unique increasing (resp. decreasing) positive solution of the equation 
\be
\mathcal{G}\Psi_\lambda = \lambda \Psi_\lambda \,.
\ee
Noting $\Psi_\lambda$ the solution for $x<y$ and $\Phi_\lambda$ the solution for $x>y$, for $x>0$,
\begin{equation}
\label{eq:LagNxxvepsIto}
\Lag(\expect{N^{x,x+\veps}_t})(\lambda) = \frac{\Psi_\lambda(0)}{\lambda \Psi_\lambda(x+\veps)} \frac{1}{1- \frac{\Psi_\lambda(x)}{\Phi_\lambda(x)} \frac{\Phi_\lambda(x+\veps)}{\Psi_\lambda(x+\veps)}} \,.
\end{equation}
These solutions $\Psi_\lambda(z)$ and $\Phi_\lambda(z)$ are smooth in $z$, so that the (Laplace transform of the) local $\zeta$-function $\zeta_X^x$ admits a meromorphic continuation to $\C$. 

\subsubsection{Obtention of the local time for continuous semimartingales}
\label{sec:LocalTimeFromNxxveps}
The obtention of an expression for the propagators of a semimartingale is in principle sufficient to obtain an expression for the Laplace transform (in time) of the local time. This is possible due to the following theorem.
\begin{theorem}[Revuz,Yor,  Ch. VI Theorem 1.10 \cite{Revuz_1999}]
Let $X$ be a continuous semimartingale and let $L^x_X(t)$ be its local time on the interval $[0,t]$ at level $x$. Writing $X= M+A$ the Doob decomposition of $X$, suppose that for $s \geq 1$ 
\be
\expect{[M]_t^{s/2} + \left(\int_0^t \abs{dA}_s\right)^s} < \infty \,,
\ee
then in $L^s$,
\be
2\veps N^{x,x+\veps}_t \xrightarrow[\veps \to 0]{L^s} L^x_X(t) \,,
\ee
in particular, this convergence holds in distribution as well.
\end{theorem}
\begin{notation}
Whenever the underlying process is implicitly clear, we will denote the local time by $L^x_t$.
\end{notation}
This theorem entails in particular that for $s\geq 2$,
\be
(2\veps)^{s-1}\expect{(N^{x,x+\veps}_t)^{s-1}} \xrightarrow[\veps \to 0]{} \expect{(L^x_X(t))^{s-1}} \,.
\ee
Under the technical hypothesis that if for some $a >0$, $\Real(\lambda) >a$ for every $\veps >0$ the function $e^{-\lambda t} \expect{(2\veps N^{x,x+\veps}_t)^{s-1}}$ is bounded above by some integrable function of $t$, the dominated convergence theorem entails that the Laplace transform of the limit and the limit of the Laplace transforms also coincide. Alternatively, we may also check whether $\expect{(2\veps N^{x,x+\veps}_t)^{s-1}}$ is a monotone function of $\veps$ over some neighbourhood for $\veps$ small enough and apply the monotone convergence theorem. This allows us to conclude that
\be
\Lag[\expect{(L_X^x(t))^{s-1}}](\lambda) = \lim_{\veps \to 0}(2\veps)^{s-1} \Lag[\expect{(N^{x,x+\veps}_t)^{s-1}}](\lambda) \,.
\ee
Finally, this has consequences for the distribution of $L_X^x(t)$ since $\expect{(L_X^x(t))^{s-1}}$ is exactly the Mellin transform of the distribution of the local time.

\section{Examples of applications}

\subsection{Brownian motion}
For the rest of this section, $B$ will denote a standard Brownian motion started at $0$. 
\subsubsection{Associated $\zeta$-function and asymptotic expansions for $N^\veps$}
Let us start by remarking that, in distribution 
\be
\sup_{[0,t]} B - B_t = \abs{B_t} \,.
\ee
The stopping times $T^\veps$ and $S^\veps$ of theorem \ref{thm:meroextensionLevy} are identically distributed and are distributed as the hitting times of $\veps$ by a reflected Brownian motion. An application of Doob's stopping theorem (\cf \cite[p.641]{Borodin_2002}) shows that
\begin{equation}
\label{eq:UvepsBrownian}
\expect{e^{-\lambda U^\veps}} = \sech^2(\veps \sqrt{2\lambda}) \,.
\end{equation}
The term $\PP(N^\veps_t \geq 1) = \PP(R_t \geq \veps)$ can also be computed by considering the fundamental solution of the corresponding heat equation with Dirichlet boundary conditions. We obtain \cite{Perez_Pr_2020} 
\be
\PP(R_t \geq \veps) = 4 \sum_{k=1}^\infty (-1)^{k-1} k \erfc\left[\frac{k\veps}{\sqrt{2t}}\right] \,.
\ee
Respectively, since Brownian motion is a $2$-stable L\'evy process, using equation \ref{eq:zetahatLevy} (here, $B(z)= \csch^2(\sqrt{2z})$) we can write 
\begin{equation}
\label{eq:BrestENveps}
\hat\zeta_B(p) = 2^{3-\frac{3p}{2}}t^{\frac{p}{2}}\frac{\Gamma(p)}{\Gamma(\frac{p}{2})} \zeta(p-1)\,.
\end{equation}
\begin{remark}
This can be obtained by using the functional properties of the Mellin transform (scaling, power of the argument) shown in table \ref{table:FunctionalPropertiesMellin} and by the results of table \ref{table:MellinTransforms}.
\end{remark}
Putting everything together, we get
\begin{theorem}
\label{thm:zetafunctionBrownian}
The $\zeta$-function of Brownian motion on the interval $[0,t]$ admits an meromorphic extension to the whole complex plane. Furthermore, it is exactly equal to
\be
\zeta_B(p) = \frac{4(2^p - 3)}{\sqrt{\pi}} \left(\frac{t}{2}\right)^{\frac{p}{2}} \Gamma\!\left(\frac{p+1}{2}\right)\zeta(p-1) 
\ee
for all $p$ and has a unique simple pole at $p=2$ of residue $t$.
\end{theorem}
\begin{remark}
That the Riemann $\zeta$-function appears in this expression is \textit{a posteriori} not surprising. Loosely speaking, this connection between the Riemann $\zeta$-function and Brownian motion appears through the relation of Brownian motion with Jacobi's $\vartheta$-function, which is a fundamental solution of the heat equation \cite{Biane_2001}. Work connecting stable distributions whose Laplace transform is of the form of equation \ref{eq:UvepsBrownian} have been widely studied by Pitman, Yor and Biane \cite{Pitman_1999,pitman_yor_2003,Biane_2001}. 
In \cite{Biane_2001}, one can also find connections between stable distributions of this form and number theoretical $L$-functions. It is also interesting to recall that there exists a probabilistic interpretation of the Riemann hypothesis through Li's criterion \cite{LI1997325} as detailed in \cite[\S 2.3]{Biane_2001}. 
\end{remark}
\begin{proof}[Proof of theorem \ref{thm:zetafunctionBrownian}]
Taking the Mellin transform of $\PP(R_t \geq \veps)$
\begin{align*}
\Mel (\PP(R_t \geq \veps))(p) = 2^{2-\frac{3p}{2}}(2^p-4)t^{\frac{p}{2}}\frac{\Gamma(p)}{\Gamma(\frac{p}{2}+1)} \zeta(p-1)\,.
\end{align*} 
Multiplying the above expression by $p$ and adding both terms and using the fact that $\Gamma(z+1) = z\Gamma(z)$ and the Legendre duplication formula, we obtain the result.
\end{proof}
The meromorphic extension of $\zeta$ allows us to directly compute correction terms for the asymptotic series given in \cite{Perez_Pr_2020}. $\Mel(\expect{N^\veps_t})(p)$ has only two poles, one at $p=0$ and one at $p=2$. Furthermore, along a vertical strip, $\Mel(\expect{N^\veps_t})(p)$ decays rapidly enough to use the fundamental correspondence (theorem \ref{thm:FundamentalCorrespondence}). Using contour integration and the Mellin inversion theorem, we can conclude that
\be
\expect{N^\veps_t} = \frac{t}{2\veps^2} + \frac{2}{3} + O(\veps^n) \quad \text{as } \veps \to 0 \,,
\ee
for any $n \in \N$, as prescribed by theorem \ref{thm:meroextensionLevy}. The expectations in the expression of the theorem can be read on the expansion
\be
\sech^2(\sqrt{2\veps}) = 1- 2 \veps + \frac{1}{2!}\,\frac{16}{3} \veps ^2 + O(\veps^3) \,.
\ee 
As previously shown, the analyticity of $\zeta_B$ beyond $\Real(p)= 2$ guarantees that there are no polynomial corrections in $\veps$ to $\expect{N^\veps_t}$ as $\veps \to 0$. The analyticity of $\zeta_B$ on the half plane $\Real(p)>2$ suggests that $\expect{N_t^\veps}$ is rapidly decreasing as $\veps \to \infty$. This is corroborated by the more general approximation of proposition \ref{prop:p2resume} for Markov processes found in \cite{Perez_Pr_2020}, namely $\expect{N_t^\veps} \sim \PP(R_t \geq \veps)$ as $\veps \to \infty$.
\par
Applying the observations made in section \ref{sec:analyticcontinuation}, the superpolynomial corrections to the asymptotic series can be found by looking carefully at the meromorphic extension of $\zeta_B$.
\begin{proposition}
\label{prop:NvepsExpansionBM}
For Brownian motion $\expect{N^\veps_t}$ admits the following series representations which converge well for large and small $\veps$ respectively
\begin{align*}
\expect{N^\veps_t} &= 4 \sum_{k\geq 1} (2k-1)\erfc\!\left(\frac{(2k-1)\veps}{\sqrt{2t}}\right) - k\,\erfc\!\left(\frac{2k\veps}{\sqrt{2t}}\right) \\
&= \frac{t}{2\veps^2} +\frac{2}{3} + 2\sum_{k\geq 1} (2(-1)^k -1) \frac{e^{-\pi^2 k^2 t/2\veps^2}t}{\veps^2}  \left[1 + \frac{\veps^2}{\pi^2k^2t}\right]\,.
\end{align*}
\end{proposition}
\begin{proof}
This asymptotic formula for $\expect{N^\veps_t}$ can be obtained by using the arguments of section \ref{sec:analyticcontinuation}. Indeed, $B(z)= \csch^2(\sqrt{2z})$ admits a meromorphic continuation to the entire complex plane and has poles of order two at $z= -\frac{\pi^2 n^2}{2}$ for every $n \in \Z \setminus\{0\}$. It follows that 
\be
\Res\!\left((-z)^{\frac{p}{2}-1}B(z),-\frac{\pi^2 n^2}{2}\right) = 2^{1-\frac{p}{2}} (2\pi)^{p-2} (p-1) \,.
\ee
Taking the inverse Mellin transform of equation \ref{eq:MellinNvepsLevy}, we obtain the desired result. The second expression converging fast for large $\veps$ is obtained by using the functional equation of the $\zeta$-function and taking the inverse Mellin transform of the expression obtained. 
\end{proof}
\begin{proof}[Alternative proof of proposition \ref{prop:NvepsExpansionBM}]
Note that
\begin{align}
\label{eq:LaplaceTotalNvepst}
\Lag(\expect{N^\veps_t})(\lambda) &= \frac{4}{\lambda} \sum_{k\geq 1} (2k-1)e^{-(2k-1)\veps \sqrt{2\lambda}} - k e^{-2k\veps\sqrt{2\lambda}} \\
&= \left[\frac{2\cosh(\veps \sqrt{2\lambda})-1}{\lambda}\right]\csch^2(\veps \sqrt{2\lambda}) \,.
\end{align}
By inverting the Laplace transform in equation \ref{eq:LaplaceTotalNvepst} (this can be done by first decomposing the hyperbolic expressions into a series of exponential terms, of which the inverse Laplace transform can be found by virtue of a table or using some computational software such as \texttt{Mathematica}. The normal convergence of the resulting series guarantees that the inverse transform of the expression is exactly the series of the inverse transforms of the resulting exponentials), we obtain
\begin{equation}
\label{eq:Nvepst}
\expect{N^\veps_t} = 4 \sum_{k\geq 1} (2k-1)\erfc\!\left(\frac{(2k-1)\veps}{\sqrt{2t}}\right) - k\,\erfc\!\left(\frac{2k\veps}{\sqrt{2t}}\right) \,,
\end{equation}
which converges quickly for large $\veps$. For $\veps \to 0$, we can get a quickly converging expression by recalling the Mittag-Leffler expansion of the hyperbolic cosecant,
\begin{align*}
\frac{\csch^2(\veps \sqrt{2\lambda})}{\lambda} &= \frac{1}{\lambda}\sum_{k \in \Z} \frac{1}{(\veps \sqrt{2\lambda}-i \pi k)^2}  \\
&=\frac{1}{2\veps^2\lambda^2} + \frac{1}{\lambda} \sum_{k \geq 1} \frac{4\veps^2 \lambda - 2\pi^2 k^2}{(2\veps^2 \lambda + \pi^2 k^2)^2} \,.\
\end{align*}
We can take the inverse Laplace transform termwise by using the residue theorem, due to the absolute and uniform convergence of the expression. After some algebra, this operation results in
\be
\expect{N^\veps_t} = \frac{t}{2\veps^2} +\frac{2}{3} + 2\sum_{k\geq 1} (2(-1)^k -1) \frac{e^{-\pi^2 k^2 t/2\veps^2}t}{\veps^2}  \left[1 + \frac{\veps^2}{\pi^2k^2t}\right] \,,
\ee
which confirms that $\expect{N^\veps_t}$ is extremely well approximated by $\frac{t}{2\veps^2} + \frac{2}{3}$ when $\veps$ is small. 
\end{proof}
From the alternative proof of proposition \ref{prop:NvepsExpansionBM} and the formulas above give the functional equation of $\zeta_B$ we can naturally retrieve the functional equation of the Riemann $\zeta$-function. This is part of the usual folklore of $\zeta$ and similar functions, where functional equations are essentially related by Poisson summation.
\begin{proposition}
Defining
\be
\eta_B(p) := (3 \cdot 2^p -8)(p-2)(2\pi t)^{-\frac{p}{2}} \zeta_B(p),
\ee
The functional equation of $\zeta_B$ is given by
\be
\eta_B(p)= \eta_B(3-p) \,.
\ee
In particular, as expected from the symmetry of $\zeta$, the axis of symmetry of $\eta_B$ is $\Real(p) = \frac{3}{2}$. 
\end{proposition}
Finally, we can calculate the moments of $N^\veps_t$. After some calculations, we obtain  
\be
\Lag(\expect{(N^\veps)^s})(\lambda) = \frac{1}{\lambda} \left[\sinh^2(\veps\sqrt{2\lambda}) \Li_{-s}(\sech^2(\veps\sqrt{2\lambda})) - \tanh^2(\veps\sqrt{\lambda/2})\right] \,.
\ee

\subsubsection{Distribution of the length of the $k$th longest branch}
\label{sec:BMkthlongestbar}
Following the discussion of section \ref{sec:distlengthbars}, we know that for $k \geq 2$, we can calculate the moment generating function $G_\al(z;p)$, noticing that Brownian motion is a $2$-stable process ($\al=2$). Then,
\be
\frac{z}{(z\expect{e^{-\veps U}})^{-1}-1} = \frac{2 z^2}{\cosh \left(2 \sqrt{2\veps} \right)-2 z+1} \,.
\ee
However, taking the Mellin transform of this expression is not easily feasible. To do so, we will write the above expression as a geometric series of decaying exponentials. Denoting $y := e^{-2\sqrt{2\veps}}$, we can write the expression above as
\be
\frac{4z^2y}{y^2-2(2z-1)y+1} = \frac{4z^2y}{(y-y_+)(y-y_-)} \,,
\ee
where $y_\pm$ are the roots of the polynomial in the denominator of the expression, namely
\be
y_\pm = 2z -1 \pm 2i\sqrt{z(1-z)} \,.
\ee
Partial fraction decomposition entails 
\be
\frac{4z^2y}{y^2-2(2z-1)y+1} = \frac{A(z)y}{y-y_+} - \frac{A(z)y}{y-y_-} \,,
\ee
where
\be
A(z) = \frac{4z^2}{y_+ - y_-}  =\frac{z^2}{\sqrt{z(z-1)}} \,.
\ee
Finally, we may express each of the terms above as a geometric series. Summing both terms,
\be
-A(z)\sum_{k\geq 1} \left(\frac{y_-^k - y_+^k}{y_+^ky_-^k}\right)y^k =  4z^2\sum_{k\geq 1} \frac{y_+^k - y_-^k}{y_+-y_-} y^k
\ee
Recalling that $y = e^{-2\sqrt{2\veps}}$ and taking the Mellin transform with respect to $\veps$
\be
\Mel\left[\frac{z}{(z\expect{e^{-\veps U}})^{-1}-1}\right]\!(p) = 2^{3-3p} \Gamma(2p) z^2  \frac{\Li_{2 p}(y_+(z)) - \Li_{2 p}(y_-(z))}{y_+(z)-y_-(z)} \,.
\ee
Finally, the generating function can be written as
\be
G_2(z;p) = 8\,\frac{\Gamma(p)}{\Gamma(\frac{p}{2})}  \left(\frac{t}{8}\right)^{\frac{p}{2}} z^2 \,\frac{\Li_p(y_+(z)) - \Li_p(y_-(z))}{y_+(z)-y_-(z)} \,.
\ee
When $z$ is in the vicinity of $0$, $y_+$ and $y_-$ are both complex, it is thus \textit{a priori} not obvious that the quantity defined above should remain real for real $p$. However, this must be so, since
\be
\frac{y_+^k - y_-^k}{y_+-y_-} = a_k(z) \,,
\ee
where $a_k(z)$ is the solution to the following difference equation
\begin{equation}
\label{eq:akrecursive}
a_k(z) =2(2z-1)\, a_{k-1} - a_{k-2}  \,,
\end{equation}
with seed $a_0 =0$ and $a_1 =1$. In fact, it is possible to express $a_k(z)$ as defined in equation \ref{eq:akrecursive} in terms of the Chebyshev polynomials of the second kind $U_k$
\be
a_k(z) = U_{k-1}(2z-1)  \quad \text{and} \quad a_k(0) = (-1)^{k-1}k\,,
\ee 
and incidentally $\frac{a_k}{a_{k-1}}$ corresponds to the $k$th convergent of the continuous fraction
\be
2(2z-1)-\cfrac{1}{2(2z-1)-\cfrac{1}{2(2z-1)-\cfrac{1}{\ddots}}} \quad .
\ee
Using these relations, it is possible to rewrite $G_2$ as 
\be
G_2(z;p)= 8\,\frac{\Gamma(p)}{\Gamma(\frac{p}{2})}  \left(\frac{t}{8}\right)^{\frac{p}{2}} \, z^2 \sum_{k\geq 1} \frac{a_k(z)}{k^p} \,.
\ee
Since ultimately what interests us are the derivatives of this function at $0$, we can rewrite $G_2(z;p)$ formally as
\be
G_2(z;p) = 8\,\frac{\Gamma(p)}{\Gamma(\frac{p}{2})}  \left(\frac{t}{8}\right)^{\frac{p}{2}} \, z^2 \sum_{n=0}^\infty\sum_{k\geq 1} \frac{a_k^{(n)}(0)}{n!} \, \frac{z^n}{k^p} \,.
\ee
The problem thus boils down to effectively computing the coefficients $a_k^{(n)}(0)$. To do so, we can consider augmenting the recurrence problem to phase space $(a_k,a_k^{(1)}, a_k^{(2)}, \cdots, a_k^{(n)})$ noticing the following relation
\be
a_k^{(n)}(z) = 2(2z-1) a_{k-1}^{(n)}(z) + 4n a_{k-1}^{(n-1)}(z) - a_{k-2}^{(n)}(z)
\ee
Setting $z=0$ and $a_k^{(n)}(0) = a_k^{(n)}$ in this relation yields 
\be
a_k^{(n)} = -2 a_{k-1}^{(n)} + 4n a_{k-1}^{(n-1)} - a_{k-2}^{(n)} \,.
\ee
For example, it is easy to verify that
\be
a_k^{(1)} = \frac{2}{3}  (-1)^{k} k (k^2-1) \quad \text{and} \quad a_k^{(2)} = \frac{2^2}{15} (-1)^{k-1} k (k^2-1) (k^2-2^2) \,.
\ee
In general, the following formula holds
\be
a^{(n)}_k = \frac{2^n(-1)^{k+n+1}}{k} \prod_{i=0}^n \frac{k^2 - i^2}{2i+1} \,.
\ee
Thus, the Mellin transforms of the distributions of the longest branches may be expressed as linear combinations of shifted and twisted $\zeta$-functions, since the $a_k^{(n)}$'s are polynomials of degree $2n+1$ in $k$. Computing the first few terms yields
\begin{align*}
&\expect{\ell_2^p(B)} =  \frac{2^{3-\frac{5 p}{2}}  t^{\frac{p}{2}} \Gamma (p)}{\Gamma(\frac{p}{2})} (2^p-2^2)\zeta (p-1) \\
&\expect{\ell_3^p(B)} = \frac{2^{4-\frac{5 p}{2}} t^{\frac{p}{2}}\Gamma(p)}{3 \Gamma(\frac{p}{2})} [(2^p-2^2)\zeta(p-1)-(2^p-2^4)\zeta(p-3)] \\
&\expect{\ell_4^p(B)} =\frac{2^{4-\frac{5 p}{2}} t^{\frac{p}{2}} \Gamma (p)}{15 \Gamma(\frac{p}{2})} [4(2^p-2^2) \zeta(p-1)-5(2^p-2^4)\zeta(p-3)+(2^p-2^6)\zeta(p-5)] \,,
\end{align*}
and so on. These Mellin transforms can be inverted to yield explicit expressions of $\PP(\ell_k \geq \veps)$, 
\begin{align*}
&\PP(\ell_2(B) \geq \veps) = 4\sum_{k \geq 1} k \!\left[\erfc\!\left(k \veps \sqrt{\frac{2}{t}}\right)-4 \erfc\!\left(2k\veps\sqrt{\frac{2}{t}}\right)\right] \\
&\PP(\ell_3(B) \geq \veps) =  \frac{8}{3}\sum_{k \geq 1} k \!\left[4\left(4k^2-1\right) \erfc\!\left(2k\veps\sqrt{\frac{2}{t}}\right)- \left(k^2-1\right) \erfc\!\left(k\veps \sqrt{\frac{2}{t}}\right)\right] \\
&\PP(\ell_4(B) \geq \veps) = \frac{8}{15} \sum_{k \geq 1}  k \!\left[\left(k^4-5 k^2+4\right) \erfc\!\left(k\veps\sqrt{\frac{2}{t}}\right)-16 \left(4 k^4-5 k^2+1\right) \erfc\!\left(2k\veps\sqrt{\frac{2}{t}}\right)\right] \,.
\end{align*}
These calculations can be performed for any $\ell_k$ without any additional difficulty. 

\subsubsection{Local $\zeta$-functions and corrections to the local time}
In what will follow, we will denote 
\be
\vp(x,t) := \frac{1}{\sqrt{2\pi t}} e^{-x^2/2t} \,.
\ee
\begin{theorem}
\label{thm:Brownianlocalzeta}
The local $\zeta$-function of Brownian motion at the level $x >0$ admits a meromorphic continuation to the entire complex plane given by
\be
\zeta^x_B(p) = 2^{-\frac{3 p}{2}} \left(2^p-1\right) t^{\frac{p}{2}} \zeta (p) \Gamma (p+1) \left(\frac{\, _1F_1\left(\frac{-p}{2};\frac{1}{2};\frac{-x^2}{2 t}\right)}{\Gamma \left(\frac{p}{2}+1\right)}-\sqrt{\frac{2x^2}{t}} \frac{\, _1F_1\left(\frac{1-p}{2};\frac{3}{2};\frac{-x^2}{2 t}\right)}{\Gamma \left(\frac{p+1}{2}\right)}\right) \,.
\ee
\end{theorem}
\begin{proof}
It suffices to calculate the propagatators, which in this case can be done explicitly. Using the reflecting property, for any $a>0$,
\be
\bra a | 0\ket  = \bra 0 | a \ket = e^{-a \sqrt{2\lambda}} \,.
\ee
So by our previous work
\be
\Lag(\expect{N^{x,x+\veps}_t})(\lambda) = \frac{e^{-(x+\veps)\sqrt{2\lambda}}}{\lambda(1-e^{-2\veps \sqrt{2\lambda}})} \,.
\ee
The Mellin transform of this expression can be calculated easily to be
\be
\Mel\Lag(\expect{N^{x,x+\veps}_t})(p,\lambda) = 2^{-\frac{3 p}{2}} \left(2^p-1\right) \lambda^{-1-\frac{p}{2}} e^{-x\sqrt{2\lambda}} \,\Gamma (p)\zeta(p) \,.
\ee
Using lemma \ref{lemma:commutingLagMel} and inverting the Laplace transform we get the desired expression. 
\end{proof}
We can also extract asymptotic relations for $\expect{N^{x,x+\veps}_t}$. 
\begin{proposition}
\label{prop:BMNxxveps}
For Brownian motion and $x>0$,
\begin{align*}
\expect{N^{x,x+\veps}} &= \sum_{k=1}^\infty \text{erfc}\left(\frac{x+(2 k-1)\veps}{\sqrt{2t}}\right) \\
&\sim \frac{1}{2\veps} \int_0^t \vp(x,s) \;ds +\sum_{k\geq 0} \frac{4 (-2)^{k}\!\left(2^{2k+1}-1\right) \zeta(2k+2)}{\pi^{2k+2}} \left[\frac{\del^{k}}{\del t^{k}} \vp(x,t)\right]\veps^{2k+1}  \; \text{as }\veps \to 0 \,.
\end{align*}
\end{proposition}
\begin{proof}
To deduce an asymptotic expression of $\expect{N^{x,x+\veps}}$, it is much easier to consider the problem in terms of the dual variables $p$ and $\lambda$. Notice that $\Mel\Lag(\expect{N^{x,x+\veps}})$ only has poles (in $p$) at every odd negative integer (stemming from those of $\Gamma$ and the non-presence of trivial zeros of $\zeta$) and at $1$ (stemming from the pole of the $\zeta$-function). The residues of these simple poles can be calculated
\begin{align*}
\Res(\Mel\Lag(\expect{N^{x,x+\veps}_t}),1) &= \frac{e^{-x\sqrt{2\lambda }}}{2 \sqrt{2} \lambda ^{3/2}}  \\
\Res(\Mel\Lag(\expect{N^{x,x+\veps}_t}),-(2k+1)) &= \frac{2\left(2^{2 k+1}-1\right) \zeta (-2 k-1)}{(2 k+1)!} e^{-x\sqrt{2\lambda}}(2\lambda)^{k-\frac{1}{2}}\nonum
 \,. 
\end{align*}
The rapid enough decay of $\Mel\Lag(\expect{N^{x,x+\veps}_t})$ on the upper and lower ends of a vertical strip in the complex plane allows us to apply the inverse Mellin theorem and contour integration to obtain an asymptotic series of $\Lag(\expect{N^{x,x+\veps}_t})(\lambda)$ 
\begin{align*}
\Lag&(\expect{N^{x,x+\veps}_t})(\lambda) \sim e^{-x\sqrt{2\lambda}} \left[\frac{1}{2\veps\lambda^{3/2}\sqrt{2}}  \right.  \\
&+\left. \sum_{k=0}^{n-1} \frac{4 (-1)^{k}\!\left(2^{2 k+1}-1\right) \zeta(2k+2)}{\pi^{2k+2}}(2\lambda)^{k-\frac{1}{2}}\veps^{2k+1} + O(\veps^{2n+1}) \right]\,.
\end{align*}
Noticing that 
\be
\Lag\!\left[2^{k} \frac{\del^{k}}{\del t^{k}} \vp(x,t)\right]\!(\lambda) = e^{-x\sqrt{2\lambda}}(2\lambda)^{k-\half} \,,
\ee
one can write that as $\veps \to 0$ 
\begin{align*}
&\expect{N^{x,x+\veps}_t} \sim  \frac{1}{2\veps} \int_0^t \vp(x,s) \;ds +\sum_{k\geq 0} \frac{4 (-2)^{k}\!\left(2^{2k+1}-1\right) \zeta(2k+2)}{\pi^{2k+2}} \left[\frac{\del^{k}}{\del t^{k}} \vp(x,t)\right]\veps^{2k+1} 
\end{align*}
This series formally diverges, so we consider it as an asymptotic series only. From this formula we retrieve some classical results from probability theory \cite{Ito_1996}, namely that the local time is a good first order approximation of $N^{x,x+\veps}_t$ (in fact almost surely) as well as an explicit expression for the expected value of the local time of Brownian motion. One can find a converging representation of $\expect{N_t^{x,x+\veps}}$ by expanding out its Laplace transform as a sum of exponentials as done for $N^\veps_t$. In this case, one obtains a series which converges absolutely and uniformly on every compact set of $x\geq 0$ and $\veps >0$
\be
\expect{N^{x,x+\veps}_t} = \sum_{k=1}^\infty \text{erfc}\left(\frac{x+(2 k-1)\veps}{\sqrt{2t}}\right) \,.
\ee 
\end{proof}
The distribution of $N^{x,x+\veps}_t$ is at this point in principle accessible given out previous work. After some algebra, we find that
\be
\Lag(\expect{(N^{x,x+\veps}_t)^{s-1}})(\lambda) = \frac{2e^{-x\sqrt{2\lambda}}}{\lambda}  \Li_{-s+1}\left(e^{-2\veps \sqrt{2\lambda}}\right)\sinh(\veps \sqrt{2\lambda})
\ee
where $\Li$ denotes the polylogarithm. Now, recalling the discussion of section \ref{sec:LocalTimeFromNxxveps}, 
\be
2\veps N^{x,x+\veps}_t \xrightarrow[\veps \to 0]{L^s} L^x_t(B) \,,
\ee
where $L^x_t(B)$ denotes the local time of the Brownian motion. 
\begin{remark}
In fact, for Brownian it is true that
\be
2\veps N^{x,x+\veps}_t \xrightarrow[\veps \to 0]{\text{a.s.}} L^x_t(B) \,.
\ee
\end{remark}
We have 
\be
\Lag(\expect{(2\veps N^{x,x+\veps}_t)^{s-1}})(\lambda) = 2^{\frac{1-s}{2}} \lambda^{-\frac{s+1}{2}} e^{-x\sqrt{2\lambda}}\,\Gamma (s)+o(1) \quad \text{as }\veps \to 0\,,
\ee
so in particular, following the discussion of section \ref{sec:LocalTimeFromNxxveps} this entails that
\begin{equation}
\label{eq:MelLagLocTime}
\Mel\Lag(\PP(L_t^x = w))(s, \lambda) = 2^{\frac{1-s}{2}} \lambda^{-\frac{s+1}{2}} e^{-x\sqrt{2\lambda}}\,\Gamma (s) \,.
\end{equation}
Inverting the Mellin transform can be easily done in this case. Doing so for $w>0$, we obtain
\be
\Lag(\PP(L^x_t =w))(\lambda) = \frac{e^{-(x+w)\sqrt{2\lambda}}}{\sqrt{\lambda }} \,.
\ee
Finally, the inverse Laplace transform of this now familiar expression is
\be
\PP(L^x_t =w) = 2\vp(x+w,t) = \frac{2}{\sqrt{2\pi t}} e^{-(x+w)^2/2t} \,.
\ee
Of course, this is a classical result that can be obtained in variety of different ways. Nevertheless, for more complicated processes, it might be possible to retrieve the distribution of their local times in this manner. Notice that the distribution of $L_t^x$ has an atom at $w=0$, corresponding to the probability of not reaching level $x$, which is why the integral of the probability density above is not $1$.
\begin{remark}
Inverting the Laplace transform first in equation \ref{eq:MelLagLocTime}, we immediately retrieve a formula for all the moments of this distribution.
\end{remark}

\subsubsection{Calculation of the average persistence diagram}
In accordance to the discussion of section \ref{sec:WassersteinCvg}, using the results of proposition \ref{prop:BMNxxveps} the density of the persistence diagram of Brownian motion can be computed to be 
\begin{proposition}
For $x>0$ and $\veps>0$, the density of the average persistence diagram of Brownian motion in birth-persistence coordinates (\cf remark \ref{rmk:densityofavg}) is 
\be
g(x,\veps) = \sqrt{\frac{2}{\pi t^3}}\sum_{k=1}^\infty (2k-1)(x+(2k-1)\veps) \,e^{-\frac{(x+(2k-1)\veps)^2}{2t}}
\ee
\end{proposition}

\subsection{Reflected Brownian motion}
We can carry out a similar procedure for the reflected Brownian motion. 
\subsubsection{Associated $\zeta$-function and $N^\veps$}
\begin{theorem}
\label{thm:ReflectedBrownianZeta}
The $\zeta$-function of the process $\abs{B}$ is
\be
\zeta_{\abs{B}}(p) = \frac{2^{1-\frac{p}{2}} (2^p -2) t^{\frac{p}{2}}}{\sqrt{\pi}} \Gamma\!\left(\frac{p+1}{2}\right) \zeta(p-1) 
\ee
which has a unique pole at $p=2$ of residue $t$.
\end{theorem}
\begin{proof}
The theorem immediately follows from applying proposition \ref{prop:LocalGlobalZeta}.
\end{proof}
We immediately deduce by inverting the Mellin transform that
\begin{proposition}
\label{prop:ENtvepsforRefBM}
The function $\expect{N_t^\veps}$ admits the following representations for reflected Brownian motion
\begin{align*}
\expect{N_t^\veps} &= \sum_{k \geq 1} 2 k \!\left[\erfc\left(\frac{k \veps}{\sqrt{2t}}\right)-2 \erfc\left(\frac{2 k \veps}{\sqrt{2t}}\right)\right]  \\
&= \frac{1}{2\veps^2} + \frac{1}{6} + \sum_{k\geq 1} \frac{4 \pi ^2 k^2 e^{-\frac{2 \pi ^2 k^2}{\veps^2}}+\veps^2 e^{-\frac{2 \pi ^2 k^2}{\veps^2}}-2 e^{-\frac{\pi ^2 k^2}{2 \veps^2}} \left(\pi ^2 k^2+\veps^2\right)}{\pi ^2 k^2 \veps^2} \,.
\end{align*}
\end{proposition}

\subsubsection{Local $\zeta$-function, $N^{x,x+\veps}$ and local time}
Here, calculating the propagators stem from classical results \cite{Borodin_2002},
\be
\bra x+\veps | x \ket \bra x | x+\veps \ket   = \frac{e^{-\veps\sqrt{2\lambda}}\cosh(x \sqrt{2\lambda})}{\cosh((x+\veps)\sqrt{2\lambda})} \,,
\ee
so
\be
\Lag(\expect{N_t^{x,x+\veps}})(\lambda) = \frac{e^{-x\sqrt{2\lambda}}}{\lambda} \csch(\veps \sqrt{2\lambda})
\ee
Inverting the Laplace transform, we get 
\begin{proposition}
\label{prop:RefBMNxxveps}
For $x>0$ and $\veps >0$, we can express $\expect{N^{x,x+\veps}}$ in the following form
\be
\expect{N^{x,x+\veps}}= 2 \sum_{k \geq 0} \erfc\left[\frac{x+(2k+1)\veps}{\sqrt{2t}}\right] \,.
\ee
\end{proposition}
We deduce from this 
\be
\Lag (\zeta_{\abs{B}})(\lambda,p) = 2^{1-\frac{3 p}{2}} \left(2^p-1\right) \lambda ^{-1-\frac{p}{2}} e^{-x\sqrt{2\lambda}} \Gamma (p+1)\zeta (p) 
\ee
\begin{theorem}
\label{thm:ReflectedBrownianZeta}
The $\zeta$-function of the process $\abs{B}$ is
\be
\zeta_{\abs{B}}(p) = \frac{2^{1-\frac{p}{2}} (2^p -2) t^{\frac{p}{2}}}{\sqrt{\pi}} \Gamma\!\left(\frac{p+1}{2}\right) \zeta(p-1) 
\ee
which has a unique pole at $p=2$ of residue $t$.
\end{theorem}
\begin{proof}
The theorem immediately follows from applying equation \ref{eq:locvsnonlocZeta}.
\end{proof}
We immediately deduce by inverting the Mellin transform that
\begin{proposition}
\label{prop:ENtvepsforRefBM}
The function $\expect{N_t^\veps}$ admits the following representations for reflected Brownian motion
\begin{align*}
\expect{N_t^\veps} &= \sum_{k \geq 1} 2 k \!\left[\erfc\left(\frac{k \veps}{\sqrt{2t}}\right)-2 \erfc\left(\frac{2 k \veps}{\sqrt{2t}}\right)\right]  \\
&= \frac{1}{2\veps^2} + \frac{1}{6} + \sum_{k\geq 1} \frac{4 \pi ^2 k^2 e^{-\frac{2 \pi ^2 k^2}{\veps^2}}+\veps^2 e^{-\frac{2 \pi ^2 k^2}{\veps^2}}-2 e^{-\frac{\pi ^2 k^2}{2 \veps^2}} \left(\pi ^2 k^2+\veps^2\right)}{\pi ^2 k^2 \veps^2} \,.
\end{align*}
\end{proposition}

\subsection{Brownian motion with drift}
\label{sec:BMwithDrift}
Throughout this section, we will denote $B^{\mu,\sigma}_t$ the Lévy process defined by 
\be
B^{\mu, \sigma}_t = \mu t + \sigma B_t \,.
\ee
We will assume that $\sigma >0$ and without loss of further generality that $\mu \geq 0$. By the scale invariance and almost sure $(\half-\delta)$-H\"older continuity of Brownian motion, it follows that the process $B^{\mu,\sigma}$ almost surely leaves and never returns to any compact set $[0,x]$ when studied over the ray $[0,\infty[$. Using this and the Markov property of $B^{\mu,\sigma}$, the following was shown by Baryshnikov.
\begin{proposition}[Baryshnikov, \cite{Baryshnikov_2019}]
For $\sigma =1$ and $x>0$, $B^{\mu,1}$ satisfies 
\be
\expect{N^{x,x+\veps}_{B^{\mu,1}}}=  \frac{1}{e^{2\mu \veps}-1}
\ee
on the infinite ray $[0, \infty[$.
\end{proposition}
Following this result, we may deduce the local $\zeta$-function associated to $B^{\mu,1}$ is given by the following expression 
\be
\zeta^x_{B^{\mu,1}}(p)= (2\mu)^{-p} \Gamma(p+1) \zeta(p) \,.
\ee
Of course, we would like to have a similar result over a compact set $[0,t]$. However, the computations in this restricted space turn out to be somewhat more challenging. Applying our theory of propagators detailed in section \ref{sec:propagatorslocaltrees}, and noticing that $B^{\mu,\sigma}$ is Lévy, it suffices to compute $\bra 0 |a \ket$ to gain access to the Laplace transform of $\expect{N^{x,x+\veps}_t}$. 
\begin{proposition}
The propagators of the process $B^{\mu,\sigma}$ for $a>0$ are given by
\begin{align*}
\bra 0 | a \ket &= e^{\frac{a\mu - a \sqrt{\mu^2+2\sigma^2\lambda}}{\sigma^2}}\\
\bra a | 0 \ket &= e^{\frac{-a\mu -a\sqrt{\mu^2+2\sigma^2\lambda}}{\sigma^2}}\,.
\end{align*}
\end{proposition}
\begin{proof}
Since $B^{\mu,\sigma}$ is a Lévy process, we start by noticing that $\bra a | 0 \ket = \bra 0 | -a \ket$, so that it suffices to calculate $\bra 0 | a \ket $ and $\bra 0 | -a \ket$. Denote $M^{\theta}_t$ the martingale defined by 
\be
M^{\theta}_t = \frac{e^{\theta B_t^{\mu,\sigma}}}{\expect{e^{\theta B_t^{\mu,\sigma}}}} \,.
\ee
We may compute the expectation in the denominator to be exactly
\be
\expect{e^{\theta B_t^{\mu,\sigma}}} = e^{\theta\mu t + \half \theta^2 \sigma^2 t} \,.
\ee
Further denoting 
\be
T^a = \inf\{t \geq 0 \, | \, B^{\mu,\sigma}_t = a\} \,,
\ee
we can calculate the propagators by considering the two following cases 
\begin{enumerate}
    \item $\bra 0 |a \ket$. In this case, $M_t^\theta$ is bounded for $\theta >0$ and by an application of Doob's stopping theorem, 
    \be
    \bra 0 | a \ket = \expect{e^{-\lambda T^a}} = e^{\frac{a\mu - a \sqrt{\mu^2+2\sigma^2\lambda}}{\sigma^2}} \, ;
    \ee
    \item $\bra 0 | -a \ket$. In this case, $M_{t \wedge T^a}^\theta$ is bounded for $\theta <0$, so that once again by the stopping theorem
    \be 
    \bra 0 | \!-\!a \ket = \expect{e^{-\lambda T^{-a}}} = e^{\frac{-a\mu -a\sqrt{\mu^2+2\sigma^2\lambda}}{\sigma^2}} \,,
    \ee 
\end{enumerate}
which finishes the proof.
\end{proof}
Following the discussion of section \ref{sec:propagatorslocaltrees},
\be
\Lag(\expect{N^{x,x+\veps}_t})(\lambda) = \frac{e^{x\frac{\mu -\sqrt{\mu ^2+ 2 \lambda  \sigma ^2}}{\sigma ^2}}}{2 \lambda } e^{\frac{\veps \mu}{\sigma^2}}\csch\left(\frac{\veps \sqrt{\mu^2 +2\lambda\sigma^2}}{\sigma^2}\right) \,.
\ee
Taking the Mellin transform we get 
\begin{align*}
\Mel\Lag(\expect{N^{x,x+\veps}_{B^{\mu,\sigma}}})(p,\lambda) &= \frac{e^{\frac{x \left(\mu -\sqrt{\mu ^2+ 2 \lambda  \sigma ^2}\right)}{\sigma ^2}}}{\lambda} \left(\frac{\sigma^2}{2\sqrt{\mu^2+2\lambda\sigma^2}}\right)^p \Gamma(p) \zeta\!\left(p,\half - \frac{\mu}{2\sqrt{\mu^2+2\lambda \sigma^2}}\right) \,,
\end{align*}
where $\zeta$ now denotes the Hurwitz $\zeta$-function, namely
\be
\zeta(p,z) = \sum_{k \geq 1} \frac{1}{(k+z)^p} \,,
\ee
which also admits a meromorphic continuation to the entire complex plane, so we can conclude that $\zeta^x_{B^{\mu,\sigma}}$ does as well. Again, we notice the presence of poles at every negative integer and at $p=1$. However, inverting the Laplace transform of this expression is difficult as soon as $\mu >0$, however, as expected, setting $\mu =0$ and $\sigma=1$, we retrieve the local $\zeta$-function of Brownian motion.

\subsubsection{Distribution of $N^{x,x+\veps}_{B^{\mu,\sigma}}$ and the local time $L_t^x(B^{\mu,\sigma})$}
As before, it is possible to calculate the moments of the occupation numbers $N^{x,x+\veps}_{B^{\mu,\sigma}}$, the result of this calculation is
\be
\Lag(\expect{(N_{B^{\mu,\sigma}}^{x,x+\veps})^{s-1}})(\lambda)=\frac{2 e^{\frac{x}{\sigma^2}\left(\mu -\sqrt{\mu^2 +2\lambda\sigma^2}\right)}e^{\frac{\mu\veps}{\sigma^2}}}{\lambda}  \Li_{-s+1}\!\left(e^{-\frac{2\veps}{\sigma^2} \sqrt{\mu^2 + 2\lambda\sigma^2}}\right) \sinh\!\left(\frac{\veps \sqrt{\mu^2 + 2\lambda \sigma^2}}{\sigma^2}\right) \,,
\ee
which implies that for $x>0$,
\be
\Mel_w \Lag(\PP(L_t^x=w))(\lambda,s)  = e^{\frac{x \left(\mu -\sqrt{2 \lambda  \sigma ^2+\mu ^2}\right)}{\sigma ^2}} \left(\frac{\sigma^2}{\sqrt{\mu^2 + 2\lambda \sigma^2}}\right)^{s-1}\!\!\!\! \Gamma(s) 
\ee
So for $w>0$,
\be
\Lag(\PP(L_t^x=w))(\lambda) =\frac{e^{\frac{x\mu}{\sigma^2}}\sqrt{\mu^2 + 2\lambda \sigma^2}}{\lambda \sigma^2} e^{-\frac{(x+w)}{\sigma^2}\sqrt{\mu^2 +2\lambda\sigma^2}}
\ee
Leading to 
\begin{align*}
\PP&(L_t^x =w) = e^{\frac{x\mu}{\sigma^2}} \int_0^t \frac{(x+w)^2-s\sigma^2}{\sqrt{2\pi s^5 \sigma^6}} e^{-\frac{(x+w)^2 + \mu^2 s^2}{2s\sigma^2}} \;ds \,.
\end{align*}

\subsection{Ornstein-Uhlenbeck process}
\label{sec:OUprocess}
Let us now consider the case of the Ornstein-Uhlenbeck process, \ie the It\^o diffusion satisfying the following SDE (for $\theta >0$ and $\sigma >0$)
\be
dX_t = -\theta X_t dt + \sigma dB_t \,.
\ee
From our discussion in section \ref{sec:ItoDiffusions}, it suffices to find the solutions to the differential equation 
\be
\frac{\sigma^2}{2} \frac{\del^2 \Psi}{\del x^2} -\theta x \frac{\partial \Psi}{\partial x}   = \lambda \Psi \,.
\ee
The solutions to the above equations are known, as this is nothing other than Hermite's differential equation. We can identify the two linearly independent solutions satisfying the respective boundary conditions to be
\be
\Psi_\lambda(x) = {}_1F_1\!\left(\frac{\lambda}{2\theta};\half;\frac{x^2\theta}{\sigma^2}\right) \quad \text{ and } \Phi_\lambda(x) = \Herm\!\left(-\frac{\lambda}{\theta};\frac{x\sqrt{\theta}}{\sigma}\right) \,.
\ee 
where ${}_1F_1$ denotes the confluent hypergeometric function and $\Herm(\mu;z)$ denotes the $\mu$th Hermite polynomial. The propagators of the Ornstein-Uhlenbeck process are readily given products of these functions, as mentionned in section \ref{sec:ItoDiffusions}, and expressions for $\Lag\expect{N^{x,x+\veps}_t}$ can be deduced from equation \ref{eq:LagNxxvepsIto}. For $x=0$, the expression obtained simplifies considerably to give
\be
\Lag[\expect{N^{0,\veps}_t}](\lambda) =\frac{\sigma }{\lambda \veps \sqrt{\theta}} \,\frac{\Gamma \left(\frac{\lambda }{2 \theta }\right)}{\Gamma \left(\frac{\theta +\lambda }{2 \theta }\right) \, _1F_1\left(\frac{\theta +\lambda }{2 \theta };\frac{3}{2};\frac{\veps^2 \theta }{\sigma ^2}\right)} \,.
\ee
This Laplace transform can in principle be formally inverted by virtue of the residue theorem. The expression above admits poles at every $\lambda = -2\theta k$, but there are also poles stemming from the zeros of the confluent hypergeometric function in the denominator, which are more difficult to locate. Nonetheless denoting $\mathcal{P}$ the set of poles stemming from ${}_1F_1$, we get the following formal expression 
\begin{align*}
\expect{N^{0,\veps}_t} &= \frac{\sqrt{\pi } \sigma  (2 \theta  t+\log (4)) \erfi\left(\frac{\veps\sqrt{\theta } }{\sigma }\right)-2 \veps\sqrt{\theta }  {}_1F_1^{(1,0,0)}\left(\frac{1}{2};\frac{3}{2};\frac{\theta  \veps^2}{\sigma ^2}\right)}{\pi ^{3/2} \sigma  \text{erfi}\left(\frac{\veps\sqrt{\theta } }{\sigma }\right)^2}   \\
&- \frac{\sigma}{2\veps\sqrt{\theta}}\sum_{k\geq 1} \frac{(-1)^k  e^{-2 \theta  k t}}{ k^2 \, \Gamma \left(\frac{1-2k}{2}\right) \Gamma (k) \, {}_1F_1\left(\frac{1-2k}{2};\frac{3}{2};\frac{\veps^2 \theta }{\sigma ^2}\right)} + \sum_{z \in \mathcal{P}} \Res(e^{\lambda t}\Lag\expect{N^{0,\veps}}, z) \,.
\end{align*}
More interestingly, equation \ref{eq:LagNxxvepsIto} can be used to yield the distribution of the local time of the Ornstein-Uhlenbeck process. The expresions for $x>0$ are rather involved, so we will study the case $x=0$ explicitly only. The simplification that occurs for the $0$ level is due to the supplementary scaling symmetry of the latter. Taking the limit as $\veps \to 0$ of $\Lag\expect{(N_t^{0,\veps})^{s-1}}$ we obtain 
\be
\Mel_w\Lag(\PP(L^0_t = w))(\lambda) = \frac{\Gamma (s)}{\lambda } \left(\frac{\sigma  \Gamma\left(\frac{\lambda }{2 \theta }\right)}{\sqrt{\theta}\Gamma\!\left(\frac{\theta +\lambda }{2 \theta }\right)}\right)^{s-1} \,.
\ee
This expression yields the Laplace transform of all the moments of $L_t^0$. For integer $s\geq 2$, these Laplace transforms can be formally inverted by using the residue theorem and computing the residues of the expression above. For example, 
\begin{align*}
\expect{L_t^0} &=\frac{\sigma}{\sqrt{\theta}} \frac{(2 \theta  t+\log (4))}{\sqrt{\pi} } + \frac{\sigma}{\sqrt{\theta}}\sum_{k \geq 1} \frac{(-1)^{k+1}  e^{-2 \theta  k t}}{ k^2 \Gamma (\frac{1-2k}{2}) \Gamma (k)} \\
\expect{(L_t^0)^2 } &= \frac{2\sigma^2}{\theta}\frac{\left(6 (\theta  t+\log (4))^2-\pi ^2\right)}{3 \pi  } - \frac{2\sigma^2}{\theta} \sum_{k \geq 1} \frac{ e^{-2 \theta  k t} \left(-2 k H_{-k-\frac{1}{2}}+2 k H_k+2 \theta  k t+1\right)}{ k^4 \Gamma(\frac{1-2k}{2})^2 \Gamma (k)^2} \,,
\end{align*}
and so on. The $H_k$ above denote the $k$th Harmonic number, \ie
\be
H_k = \sum_{n=1}^k \frac{1}{n} = \int_0^1 \frac{1-z^k}{1-z} \;dz\;, 
\ee 
where the second representation of $H_k$ holds for non-integer $k$. Finally, taking the inverse Mellin transform we get
\be
\Lag[\PP(L^0_t = w)](\lambda) = \frac{\sqrt{\theta } \Gamma \left(\frac{\theta +\lambda }{2 \theta }\right) }{\lambda  \sigma  \Gamma \left(\frac{\lambda }{2 \theta }\right)} \exp\left[-\frac{w\sqrt{\theta }  \Gamma \left(\frac{\theta +\lambda }{2 \theta }\right)}{\sigma  \Gamma \left(\frac{\lambda }{2 \theta }\right)}\right] \,.
\ee
This expression exhibits essential singularities at every $\lambda = -(2k-1)\theta$ for every $k \in \N$. For $x>0$, it is possible to treat the problem analogously, but the expressions involved -- while expressible in terms of Hermite polynomials and confluent hypergeometric functions -- remain quite involved. Nonetheless, the elementary methods introduced in this paper allow us to calculate such quantities, which have not been quantitatively studied until recently in \cite{Kishore_2021} via perturbative methods.
\par
For the sake of completeness we give the expression of $\Mel_w\Lag\PP(L_t^x=w)(\lambda,s)$ for $x>0$, 
\be
\Mel_w\Lag\PP(L_t^x=w)(\lambda,s) = \frac{\Gamma(s)}{{}_1F_1(\frac{\lambda}{2\theta};\half;\frac{x^2\theta}{\sigma^2}) \lambda} \left[\frac{\lambda}{\sigma \sqrt{\theta}} \frac{\Herm(-\frac{\lambda+\theta}{\lambda},\frac{x\sqrt{\theta}}{\sigma})}{\Herm(-\frac{\lambda}{\theta},\frac{x\sqrt{\theta}}{\sigma})} + \frac{x\lambda}{\sigma^2}\frac{{}_1F_1\left(1+\frac{\lambda }{2 \theta };\frac{3}{2};\frac{x^2 \theta }{\sigma ^2}\right)}{{}_1F_1\left(\frac{\lambda }{2 \theta };\frac{1}{2};\frac{x^2 \theta }{\sigma ^2}\right)}\right]^{1-s}\,.
\ee
Studying the poles in $\lambda$ of this expression, it is in principle possible to give expressions for the moments of the local times at $x$, as we did in the case of $x=0$.

\section{Acknowledgements}
The author would like to thank Pierre Pansu and Claude Viterbo for helping with the redaction of the manuscript as well as their guidance. Many thanks are also owed to Pascal Massart, Robert Adler, Fr\'ed\'eric Chazal and Shmuel Weinberger for the fruitful discussions regarding this work.

\bibliographystyle{abbrv}
\bibliography{PhDThesis}

\end{document}